\documentclass[10pt]{amsart}
\usepackage{amssymb}
\usepackage{bm}
\usepackage{graphicx}
\usepackage[centertags]{amsmath}
\usepackage{amsfonts}
\usepackage{amsthm}
\usepackage{amsbsy}
\usepackage{mathtools}
\usepackage{mathrsfs}
\usepackage{cases}
\usepackage[all]{xy}
\usepackage[linktocpage=true]{hyperref}
\usepackage{hyperref}
\hypersetup{
    colorlinks=true,
    linkcolor=blue,
    filecolor=blue,
    citecolor=blue,
    urlcolor=cyan}
\usepackage[margin=4cm]{geometry}
\newtheorem{thm}{Theorem}[section]
\newtheorem{cor}[thm]{Corollary}
\newtheorem{lem}[thm]{Lemma}
\newtheorem{sbl}[thm]{Sublemma}
\newtheorem{prop}[thm]{Proposition}
\newtheorem{claim}[thm]{Claim}
\newtheorem{fact}[thm]{Fact}
\newtheorem{observation}[thm]{Observation}

\newtheorem{defn}[thm]{Definition}

\theoremstyle{remark}
\newtheorem{rem}[thm]{Remark}
\newtheorem{examp}[thm]{Example}

\newcommand{\rr}{\mathbb{R}}
\newcommand{\nn}{\mathbb{N}}
\newcommand{\ee}{\varepsilon}

\newcommand{\meg}{\geqslant}
\newcommand{\mik}{\leqslant}
\newcommand{\ave}{\mathbb{E}}
\newcommand{\prob}{\mathbb{P}}

\newcommand{\bbx}{\boldsymbol{X}}
\newcommand{\xtensor}{\boldsymbol{X}=\langle X_i: i\in [n]^d\rangle}
\newcommand{\bbth}{\boldsymbol{\theta}}
\newcommand{\thtensor}{\boldsymbol{\theta}=\langle \theta_i: i\in [n]^d\rangle}
\newcommand{\extension}{\sqsubseteq}
\newcommand{\osc}{\mathrm{osc}}

\newcommand{\seminorm}[1]{{\left\vert\kern-0.25ex\left\vert\kern-0.25ex\left\vert #1
    \right\vert\kern-0.25ex\right\vert\kern-0.25ex\right\vert}}
\begin{document}

\title[Anticoncentration and Berry--Esseen bounds for random tensors]{Anticoncentration
and Berry--Esseen bounds for random tensors}

\author{Pandelis Dodos and Konstantinos Tyros}

\address{Department of Mathematics, University of Athens, Panepistimiopolis 157 84, Athens, Greece}
\email{pdodos@math.uoa.gr}

\address{Department of Mathematics, University of Athens, Panepistimiopolis 157 84, Athens, Greece}
\email{ktyros@math.uoa.gr}



\begin{abstract}
We obtain estimates for the Kolmogorov distance to appropriately chosen gaussians, of linear functions
\[ \sum_{i\in [n]^d} \theta_i X_i \]
of random tensors $\bbx=\langle X_i:i\in [n]^d\rangle$ which are symmetric and exchangeable, and
whose entries have bounded third moment and vanish on diagonal indices. These estimates are expressed
in terms of intrinsic (and easily computable) parameters associated with the random tensor $\bbx$
and the given coefficients $\langle \theta_i:i\in [n]^d\rangle$, and they are optimal in various regimes.

The key ingredient---which is of independent interest---is a combinatorial CLT for high-dimensional
tensors which provides quantitative non-asymptotic normality under suitable conditions, of statistics
of the form
\[ \sum_{(i_1,\dots,i_d)\in [n]^d} \boldsymbol{\zeta}\big(i_1,\dots,i_d,\pi(i_1),\dots,\pi(i_d)\big) \]
where $\boldsymbol{\zeta}\colon [n]^d\times [n]^d\to\rr$ is a deterministic real tensor,
and $\pi$ is a random permutation uniformly distributed on the symmetric group $\mathbb{S}_n$.
Our results extend, in any dimension $d$, classical work of Bolthausen who covered
the one-dimensional case, and more recent work of Barbour/Chen who treated the
two-dimensional case.
\end{abstract}

\maketitle

\tableofcontents


\part{Introduction} \label{part1}


\section{Normal approximation of linear functions of random tensors} \label{sec1}

\numberwithin{equation}{section}

\subsection{Framework} \label{subsec1.1}

We will be working with tensors, and as such it is useful to begin by introducing
some notation. In what follows, let $n,d$ be positive integers with $n\meg 2d$.

We denote by $[n]^d$ the set of all sequences of length $d$ taking values in the discrete
interval $[n]\coloneqq \{1,\dots,n\}$. We view, however, every element of $[n]^d$ also as
a function from $[d]$~to~$[n]$. Thus, for every $i\in [n]^d$ and every $r\in [d]$ by $i(r)$
we denote the $r$-th element of $i$; moreover, if $s$ is a positive integer with $s\mik d$
and $p\in [n]^s$, then we write $p\extension i$ to denote the fact that $i$ extends $p$,
that is, if $p(r)=i(r)$ for every $r\in [s]$. We also denote by~$[n]^d_{\mathrm{Inj}}$
the set of all one-to-one functions in $[n]^d$ (equivalently, the set of all finite
sequences in~$[n]^d$ with distinct entries).

\subsubsection{\!} \label{subsubsec1.1.1}

Our objects of study are random tensors $\xtensor$ whose entries satisfy
\[ \ave[X_i]=0, \ \ \ \ \ave[X_i^2]\mik 1 \ \ \ \text{ and } \ \ \ \ave\big[|X_i|^3\big]<\infty, \]
and which are
\begin{enumerate}
\item[---] \textit{symmetric}, that is, $X_{(i_1,\dots,i_d)}=X_{(i_{\tau(1)},\dots,i_{\tau(d)})}$
for every $(i_1,\dots,i_d)\in [n]^d$ and every permutation $\tau$ of $[d]$, and
\item[---] \textit{exchangeable}, that is, for every permutation $\pi$ of $[n]$ the random
tensors $\bbx$ and $\bbx_\pi\coloneqq\langle X_{\pi\circ i}: i\in [n]^d\rangle$ have the same distribution.
\end{enumerate}
(Here and in the rest of this paper, we set $\pi\circ i\coloneqq\big(\pi(i_1),\dots,\pi(i_d)\big)\in [n]^d$
for every $i=(i_1,\dots,i_d)\in [n]^d$ and every permutation $\pi$ of $[n]$.)
This is, arguably, a large class of random tensors which encompasses the following examples.
\medskip

\noindent $\bullet$ It includes random tensors of the form $X_i=h(\xi_{i(1)},\dots,\xi_{i(d)})$,
where $(\xi_k)$ is a sequence of i.i.d. random variables which take values in a measurable space $\mathcal{E}$,
and $h\colon\mathcal{E}^d\to\rr$ is a measurable symmetric function; see \cite{Ald83,Kal05}.
These random tensors are, of course, ubiquitous in probability and statistics.
\medskip

\noindent $\bullet$ It also includes random tensors of the form
$X_i=\prod_{\ell=1}^d \zeta_{i(\ell)} - \ave\big[ \prod_{\ell=1}^d \zeta_{i(\ell)}\big]$,
where $(\zeta_1,\dots,\zeta_n)$ is an exchangeable random vector which take values in $[0,1]^n$.
An important special case of this class of examples is obtained by considering boolean random vectors
$(\zeta_1,\dots,\zeta_n)$ which are uniformly distributed on a ``slice"
$\binom{[n]}{k}\coloneqq \{A\subseteq [n]: |A|=k\}$; see, \textit{e.g.}, \cite{FKMW18}.

\begin{rem}
Note that the class of symmetric and exchangeable random tensors  is closed under mixtures.
This is a basic property which is significant both from a theoretical as well as an applied point of view.
\end{rem}

\subsubsection{ } \label{subsubsec1.1.2}

Now let $\thtensor$ be a (not necessarily symmetric) deterministic real tensor whose diagonal terms
vanish---that is, $\theta_i=0$ if $i\notin [n]^d_{\mathrm{Inj}}$---and consider the random variable
\begin{equation} \label{e1.1}
\langle \bbth,\bbx\rangle \coloneqq \sum_{i\in [n]^d} \theta_i X_i.
\end{equation}
Our goal is to estimate the quantity
\begin{equation} \label{e1.2}
d_K\big(\langle \bbth,\bbx\rangle,\mathcal{N}(0,\sigma^2)\big)
\end{equation}
where $\sigma^2$ denotes the variance of $\langle \bbth,\bbx\rangle$ and $\mathcal{N}(0,\sigma^2)$
denotes the normal random variable with zero mean and variance $\sigma^2$; here, and in the rest of this paper,
for a pair $X,Y$ of real-valued random variables, by $d_K(X,Y)$ we denote their \textit{Kolmogorov distance}
\begin{equation} \label{e1.3}
d_K(X,Y)\coloneqq\sup_{x\in\rr} \big|\prob(X\mik x)-\prob(Y\mik x)\big|.
\end{equation}
Notice that, since $\bbx$ is symmetric, we may write
\begin{equation} \label{e1.4}
\langle \bbth,\bbx\rangle =\langle \boldsymbol{a}, \boldsymbol{Y}\rangle
\end{equation}
where $\boldsymbol{a}=\langle a_i: i\in [n]^d\rangle$ is a symmetric real tensor whose diagonal
terms vanish, and $\boldsymbol{Y}=\langle Y_i:i\in [n]^d\rangle$ is the random tensor defined by
setting $Y_i=X_i$ if $i\in [n]^d_{\mathrm{Inj}}$, and $Y_i=0$ otherwise. (In particular, $\boldsymbol{Y}$
is also symmetric and exchangeable.) Thus, in what follows, we make the following basic assumptions
on $\bbx$ and $\bbth$.
\begin{enumerate}
\item[($\mathcal{A}1$)] \label{A1} We have $\ave[X_i]=0$, $\ave[X_i^2]\mik 1$
and $\ave\big[|X_i|^3\big]<\infty$ for every $i\in [n]^d$.
\item[($\mathcal{A}2$)] \label{A2} The random tensor $\bbx$ is symmetric, exchangeable
and its diagonal terms vanish.
\item[($\mathcal{A}3$)] \label{A3} The real tensor $\bbth$ is symmetric
and its diagonal terms vanish.
\end{enumerate}

\subsection{Relevant parameters} \label{subsec1.2}

We shall estimate the distance in \eqref{e1.2} using some intrinsic parameters which are associated
with $\bbth$ and $\bbx$ respectively.

\subsubsection{\!} \label{subsubsec1.2.1}

First, for every $s\in \{0,1,\dots,d\}$ we set
\begin{equation} \label{e1.5}
\seminorm{\bbth}_s\coloneqq \Big( \sum_{j\in[n]^s}\big(\!\sum_{j\sqsubseteq i\in[n]^d} \!
\theta_i \big)^2 \Big)^{1/2}
\end{equation}
with the convention that the first sum vanishes if $s=0$. Thus, $\seminorm{\cdot}_s$ is a seminorm
which interpolates between the summing seminorm ($s=0$) and the euclidean\footnote{In this context,
the euclidean norm is also referred to as the Hilbert--Schmidt norm.}  norm ($s=d$).

Moreover, for every $s\in [d]$ we set
\begin{equation} \label{e1.6}
\delta_s=\delta_s(\bbx)\coloneqq \ave[X_{(1,\dots,d)}X_{(1,\dots,s,d+1,\dots,2d-s)}]
\end{equation}
and
\begin{equation} \label{e1.7}
\delta_0=\delta_0(\bbx)\coloneqq \ave[X_{(1,\dots,d)}X_{(d+1,\dots,2d)}].
\end{equation}
Notice that, since $\bbx$ is exchangeable, the quantities $\delta_0,\delta_1,\dots,\delta_d$
completely determine the correlation matrix of $\bbx$.

In order to see the relevance of the parameters introduced so far, we shall use them to
compute the variance of $\langle \bbth,\bbx\rangle$; the proof is given in Subsection \ref{subsec8.2}.
\begin{prop} \label{p1.2}
Let $\bbx,\bbth$ which satisfy \emph{(\hyperref[A1]{$\mathcal{A}$1})},
\emph{(\hyperref[A2]{$\mathcal{A}$2})} and \emph{(\hyperref[A3]{$\mathcal{A}$3})}.
Then we have
\begin{equation} \label{e1.8}
\mathrm{Var}\big(\langle \bbth,\bbx\rangle\big)=
\sum_{s=0}^{d} \binom{d}{s}^2\, s!\, \Big( \sum_{t=0}^s \binom{s}{t} (-1)^{s-t}\,\delta_t\Big)\,
\seminorm{\boldsymbol{\theta}}_s^2.
\end{equation}
\end{prop}

\subsubsection{\!} \label{subsubsec1.2.2}

Next, for every $s\in \{0,1,\dots,d\}$ set
\begin{equation} \label{e1.9}
\Sigma_s=\Sigma_s(\bbx)\coloneqq \sum_{t=0}^s \binom{s}{t}\, (-1)^{s-t}\, \delta_t.
\end{equation}
By \eqref{e1.8}, it is clear that these quantities are related to the variance of
$\langle \bbth,\bbx\rangle$, thought their role is most transparently seen in the case
of random tensors that admit a Hoeffding decomposition. That said, we have the following information
for general exchangeable random tensors; see Lemma \ref{l10.5} for a more precise result.
\begin{fact} \label{f1.3}
Let $\bbx$ be a random tensor which satisfies \emph{(\hyperref[A1]{$\mathcal{A}$1})}
and \emph{(\hyperref[A2]{$\mathcal{A}$2})}.
Then, we have $\Sigma_s\meg - \frac{8d^2 2^d}{n}$ for every $s\in\{0,\dots,d\}$.
\end{fact}

\subsubsection{\!} \label{subsubsec1.2.3}

We will need one last parameter which is associated with the random tensor $\bbx$.
Specifically, we define the \textit{oscillation} of $\bbx$ by
\begin{equation} \label{e1.10}
\osc(\bbx)\coloneqq  \Big{\|} \frac{1}{n}\sum_{j=1}^{n} \Big( \frac{1}{n^{d-1}}
\sum_{\substack{i\in[n]^d\\i(1) = j}}X_i \Big)^2 - \delta_1\Big{\|}_{L_1}.
\end{equation}
The oscillation of random vectors appeared, for instance, in recent work of Bobkov,
Chistyakov and G\"{o}tze \cite{BCG18} albeit with different terminology.
In higher dimensions, it can be thought of as a quantitative measure
of dissociativity; of course, in order to estimate the distance
in \eqref{e1.2}, some information of this form is necessary. The advantage of the
oscillation is that it leads---more often than not---to optimal results and, more importantly,
it can be fairly easily estimated for several classes of random tensors, thus making our task
computationally feasible. We shall discuss these issues in detail in Section \ref{sec3}.

\subsection{Main estimate} \label{subsec1.3}

We are now ready to state the first main result of this paper.
\begin{thm} \label{t1.4}
Let $\bbx,\bbth$ which satisfy \emph{(\hyperref[A1]{$\mathcal{A}$1})},
\emph{(\hyperref[A2]{$\mathcal{A}$2})} and \emph{(\hyperref[A3]{$\mathcal{A}$3})},
and such that $\seminorm{\bbth}_1=1$. Set $\kappa=\kappa(d)\coloneqq 20 d^3 18^d (2d)!$ and
$B\coloneqq \big\|\frac{1}{n^d}\sum_{i\in [n]^d} X_i\big\|_{L_2}^2$, and let
$\alpha \in (0,1)$ such that the following non-degenericity condition holds true
\begin{equation} \label{e1.11}
\delta_1\meg \max\Big\{ \mathrm{osc}(\bbx)^\alpha, B^\alpha,
\Big(\frac{\kappa}{n}\Big)^\alpha\Big\}.
\end{equation}
Then, setting $\sigma^2\coloneqq \mathrm{Var}\big(\langle \bbth,\bbx\rangle\big)$, we have
\begin{equation} \label{e1.12}
d_K\big( \langle\bbth,\bbx\rangle, \mathcal{N}(0,\sigma^2)\big) \mik E_1+E_2+E_3
\end{equation}
where
\begin{align}
\label{e1.13} E_1 & \coloneqq 5\mathrm{osc}(\bbx)^{1-\alpha}+
5|\delta_0|^{1-\alpha} +
\Big|\frac{\delta_0}{d^2\delta_1}\, (\seminorm{\bbth}^2_0-1)\Big| +
\frac{6\kappa}{n^{1-\alpha}} +
4\, \frac{\seminorm{\bbth}^2_0}{n} \\
\label{e1.14} E_2 & \coloneqq 2^{36}\, \frac{\ave\big[|X_{(1,\dots,d)}|^3\big]}{\delta_1^{3/2}}
\Big( \sum_{j=1}^n \Big|\!\sum_{\substack{i\in[n]^d\\i(1) = j}} \theta_i\Big|^3\Big) \\
\label{e1.15} E_3 & \coloneqq 3\kappa\,
\frac{1}{d\sqrt{\delta_1}}\, \sum_{s=2}^d \binom{d}{s} \sqrt{s!} \,
\sqrt{\Sigma_s+ \frac{16d^2 2^d}{n}} \, \seminorm{\bbth}_s.
\end{align}
\end{thm}
\begin{rem} \label{r1.5}
The proof of Theorem \ref{t1.4} actually yields a better estimate by comparing the Kolmogorov distance of
$\langle\bbth,\bbx\rangle$ with a gaussian whose variance is an appropriately selected approximation of
$\mathrm{Var}\big(\langle \bbth,\bbx\rangle\big)$; see Proposition \ref{p11.1} for details.
\end{rem}
\begin{rem} \label{r1.6}
The assumption in Theorem \ref{t1.4} that $\seminorm{\bbth}_1=1$ is, of course, a normalization and it can always
be achieved by rescaling $\bbth$.
\end{rem}
\begin{rem} \label{r1.7}
The dependence of the constant $\kappa$ on $d$ is, most likely, non-optimal. It is an
interesting problem to determine its dependence on the dimension $d$.
\end{rem}
\begin{rem} \label{r1.8}
We note that a non-degenericity condition like \eqref{e1.11} is natural at this level of generality;
indeed, Theorem \ref{t1.4} can be seen as a quantitative generalization of Hoeffding's CLT for degenerate
symmetric U-statistics \cite{Ho48}. (See Paragraph \ref{subsubsec12.2.1} for more details.)
On the other hand, classical results from the theory of  U-statistics---see, \textit{e.g.},
\cite[Section 5.5]{Se80}---suggest that, in the degenerate case, the random variable
$\langle\bbth,\bbx\rangle$ is typically far from normal. That said, we note that normal approximation
for degenerate statistics is also heavily investigated; see, \textit{e.g.}, \cite{NP12} and the references therein.
\end{rem}

\subsubsection{\!} \label{subsubsec1.3.1}

We close this section by briefly discussing the nature of the bound \eqref{e1.12}.

The first term of $E_1$ is, essentially, the oscillation of $\bbx$. The second and third
terms are quantitative measures of the correlation of the entries of $\bbx$, and they can be absorbed
in the rest of the error terms if $\delta_0=O(1/n)$. The fourth term is related to the non-degenericity
assumption \eqref{e1.11}, and it can also be absorbed in the rest of the error terms if the parameter
$\alpha$ is less than or equal to $1/2$. The last term in $E_1$ is more subtle,
and it is related to the extendability\footnote{See Paragraph \ref{subsubsec12.2.1}
for the definition of extendability.} of $\bbx$.

In order to appreciate the error terms $E_2$ and $E_3$, consider the important special case
of a random tensor $\bbx$ whose entries are functions of i.i.d. random variables.
In this case, applying Hoeffding's decomposition, we may decompose the random variable
$\langle \bbth,\bbx\rangle$ as a linear term $L$ plus a remainder $R$; then the bound
$E_2$ is proportional to the classical Berry--Esseen bound for the linear term $L$,
while the bound $E_3$ is proportional to the ratio of the standard deviations of $R$ and $L$.
For this particular case, this bound is known and it follows (with better constants)
from the powerful nonlinear Berry--Esseen theorem of Chen/Shao \cite{CS07}.
Theorem \ref{t1.4} essentially asserts that one can reach similar conclusions by merely
assuming exchangeability instead of independence (as long as one can control the oscillation)
despite the fact that in this context there is no canonical decomposition like Hoeffding's decomposition.
Thus, we may loosely describe Theorem~\ref{t1.4}  as a ``nonlinear Berry--Esseen theorem without independence".


\section{Main tool: combinatorial CLT for high-dimensional tensors} \label{sec2}

\numberwithin{equation}{section}

The starting observation of the proof of Theorem \ref{t1.4} is that the distribution of the random
variable $\langle \bbth,\bbx\rangle$ can be expressed as a mixture of certain tensor permutation statistics.
(See Subsection \ref{subsec11.1} for a more detailed high-level overview of the argument.)
Thus, Theorem \ref{t1.4} is effectively reduced to a quantitative combinatorial central
limit theorem for high-dimensional tensors which we are about to describe.

\subsection{Tensor permutation statistics} \label{subsec2.1}

Let $n,d$ be positive integers with $n\meg d$, and let $\boldsymbol{\zeta}\colon [n]^d\times [n]^d\to\rr$
be a (deterministic) real tensor. With $\boldsymbol{\zeta}$ we associate the statistic
\begin{equation} \label{e2.1}
Z=\sum_{i\in [n]^d} \boldsymbol{\zeta}(i,\pi\circ i)
\end{equation}
where $\pi$ is a random permutation which is uniformly distributed on the symmetric group~$\mathbb{S}_n$.
(Recall that for every $\pi\in\mathbb{S}_n$ and every $i=(i_1,\dots,i_d)\in [n]^d$ we set
$\pi\circ i\coloneqq\big(\pi(i_1),\dots,\pi(i_d)\big)\in [n]^d$.) These statistics
are classical---see, \textit{e.g.}, \cite{Da44}---and appear in a variety of disciplines
with pure as well as applied orientation.

One drawback of tensor permutation statistics is the computational difficulty of their basic parameters,
like the variance. This defect can be fixed by restricting our attention to the following class of tensors.
\begin{defn}[Hoeffding tensor] \label{d2.1}
Let $n,d$ be positive integers with $n\meg d$, and let $\boldsymbol{\xi}\colon [n]^d\times [n]^d\to\rr$.
We say that $\boldsymbol{\xi}$ is a \emph{Hoeffding tensor} if for every $r\in [d]$,
every $j_0,q_0\in [n]^{[d]\setminus \{r\}}$ and every $i_0,p_0\in [n]^d$ we have
\begin{equation} \label{e2.2}
\sum_{j_0\extension i\in [n]^d} \!\! \boldsymbol{\xi}(i,p_0)=0 \ \ \ \ \text{ and } \ \ \ \
\sum_{q_0\extension p\in [n]^d} \!\! \boldsymbol{\xi}(i_0,p)=0
\end{equation}
where $[n]^{[d]\setminus \{r\}}$ denotes the set of all maps from $[d]\setminus \{r\}$
to $[n]$. $($See also Section \emph{\ref{sec4}}.$)$
\end{defn}
In order to see the relevance of Hoeffding tensors in this context note that,
by applying an appropriate decomposition which goes back to Hoeffding  \cite{Ho51}, the statistic \eqref{e2.1}
can be written as\footnote{Note that $\pi\circ i\coloneqq\big(\pi(i_1),\dots,\pi(i_s)\big)\in [n]^s_{\mathrm{Inj}}$
for every $\pi\in\mathbb{S}_n$ and every $i=(i_1,\dots,i_s)\in [n]^s_{\mathrm{Inj}}$.}
\begin{equation} \label{e2.3}
W=\sum_{s=1}^d \sum_{i\in [n]^s_{\mathrm{Inj}}} \boldsymbol{\xi}_s(i,\pi\circ i)
\end{equation}
where $\boldsymbol{\xi}_s\colon [n]^s\times [n]^s\to \rr$ is a Hoeffding tensor for every $s\in [d]$.
(For the class of tensors which are relevant to Theorem \ref{t1.4}, we describe this transformation
in Section \ref{sec9}.) Therefore, in what follows we shall focus on statistics of the form \eqref{e2.3}.

\subsection{Permutation statistics of order one} \label{subsec2.2}

Classical results for matrix permutation statistics were obtained by
Wald/Wolfowitz \cite{WW44} and Hoeffding \cite{Ho51} who established asymptotic normality
under general conditions.

The problem of establishing quantitative, non-asymptotic, normality of $W$-statistics of order one,
was more delicate. The optimal result in this direction was eventually obtained by Bolthausen \cite{Bo84}
who showed that
\begin{equation} \label{e2.4}
d_K\Big( \sum_{i=1}^n \boldsymbol{\xi}\big(i,\pi(i)\big), \mathcal{N}(0,1)\Big)\mik
\frac{C_1}{n} \sum_{i,j=1}^n |\boldsymbol{\xi}(i,j)|^3
\end{equation}
for every Hoeffding tensor $\boldsymbol{\xi}\colon [n]\times [n]\to\rr$ which satisfies
$\sum_{i,j=1}^n \boldsymbol{\xi}(i,j)^2=n-1$; here, $C_1\meg 1$ is an absolute constant.
(It was shown in \cite{CF15} that we can take $C_1=451$.) Bolthausen's work was one of
the earliest and most successful applications of Stein's method of normal
approximation \cite{CGS11,St82,St86}.

\subsection{Permutation statistics of order two} \label{subsec2.3}

$W$-statistics of order two are also studied systematically in the literature;
see, \textit{e.g.}, \cite{BE86,ZBCL97} and the references therein.
The strongest quantitative normal approximation was obtained by Barbour and Chen \cite{BC05}
who showed\footnote{This result was not explicitly isolated in \cite{BC05},
but it follows fairly straightforwardly from the methods developed therein.}
that if $\boldsymbol{\xi}_1\colon [n]\times [n]\to\rr$ and
$\boldsymbol{\xi}_2\colon [n]^2\times [n]^2\to \rr$ are Hoeffding tensors with
$\sum_{i,j=1}^n \boldsymbol{\xi}_1(i,j)^2=n-1$, and $W$ is the statistic associated with
$\boldsymbol{\xi}_1,\boldsymbol{\xi}_2$ via \eqref{e2.3}, then
\begin{equation} \label{e2.5}
d_K\big(W,\mathcal{N}(0,1)\big)\mik \frac{a C_1}{n} \sum_{i,j=1}^n |\boldsymbol{\xi}_1(i,j)|^3 +
C_2 \sqrt{\frac{1}{n^2}\sum_{i,p\in [n]^2} \!\! \boldsymbol{\xi}_2(i,p)^2}
\end{equation}
where $C_1\meg 1$ is the same constant as in \eqref{e2.4}, and $a,C_2 \meg 1$ are absolute
(and effective) constants. The work of Barbour and Chen was based on the Stein/Chen method
of normal approximation via concentration, and it also used Bolthausen's estimate \eqref{e2.4}.
Note that \eqref{e2.5} should be interpreted as a perturbation result, and when viewed as such,
it is essentially optimal; see \cite[Section 1]{BC05} for a detailed discussion.

\subsection{Permutation statistics of arbitrary order} \label{subsec2.4}

High-dimensional statistics related to \eqref{e2.1} and \eqref{e2.3} have been studied,
for instance, in \cite{BG93,BoG02,Lo96}. That said, however, so far no high-dimensional analogue
of \eqref{e2.4} and \eqref{e2.5} has been obtained. The following theorem---which is
the second main result of this paper---fills in this gap,
and extends these estimates to $W$-statistics of arbitrary order.
\begin{thm} \label{t2.2}
Let $n,d$ be positive integers such that $n\meg 4d^2$. For every $s\in [d]$ let
$\boldsymbol{\xi}_s\colon [n]^s\times [n]^s\to\rr$ be a Hoeffding tensor, and set
\begin{equation} \label{e2.6}
\beta_s\coloneqq \sum_{i,p\in [n]^s} \!\! \boldsymbol{\xi}_s(i,p)^2.
\end{equation}
Assume that $\beta_1=n-1$, and let $W$ be the statistic associated with
$\boldsymbol{\xi}_1,\dots,\boldsymbol{\xi}_d$ via \eqref{e2.3}. Then we have
\begin{equation} \label{e2.7}
d_K\big(W, \mathcal{N}(0,1)\big)\mik \frac{2^{18} C_1}{n} \sum_{i,j=1}^n |\boldsymbol{\xi}_1(i,j)|^3 +
C_d\, \sum_{s=2}^d \sqrt{\frac{\beta_s}{n^s}}
\end{equation}
where $C_1\meg 1$ is as in \eqref{e2.4}, and $C_d$ is a positive constant that depends only~on~$d$.
In fact, we can take $C_d=5d^2 e^d (2d)!$.
\end{thm}
As expected, the proof of Theorem \ref{t2.2} relies on the work of Bolthausen, and Barbour/Chen.
In higher dimensions, the additional difficulty is to obtain tight estimates for the $L_2$ distance
of certain exchangeable pairs of random variables, a task that becomes more and more combinatorially intricate
as the dimension $d$ increases. These estimates are, in fact, responsible for the scaling $n^s$ that
appears in \eqref{e2.7}; this scaling is, in turn, important for the proof of Theorem \ref{t1.4}
and its applications.


\section{Applications} \label{sec3}

\numberwithin{equation}{section}


\subsection{One-dimensional examples: exchangeable random vectors} \label{subsec3.1}

The simplest instance of Theorem \ref{t1.4} concerns random vectors. Specifically,
let $n\meg 2$ be an integer, let $\bbx=(X_1,\dots,X_n)$ be an exchangeable random vector whose entries
satisfy~(\hyperref[A1]{$\mathcal{A}$1}), and let $\bbth=(\theta_1,\dots,\theta_n)$ be a
vector in $\rr^n$. (Assumption (\hyperref[A3]{$\mathcal{A}$3}) is superfluous if $d=1$.)
It is straightforward to check that
\begin{enumerate}
\item[(i)] $\seminorm{\bbth}_0=|\theta_1+\dots+\theta_n|$,
\item[(ii)] $\seminorm{\bbth}_1=\|\bbth\|_{\ell_2}=(\theta_1^2+\dots+\theta_n^2)^{1/2}$,
\item[(iii)] $\delta_0=\ave[X_1 X_2]$ and $\delta_1=\ave[X_1^2]$.
\end{enumerate}
Specializing the bound obtained by Theorem \ref{t1.4} to random vectors,
we obtain the following corollary. (See Subsection \ref{subsec12.1} for the proof.)
\begin{cor} \label{c3.1}
Let $\bbx$ be an exchangeable random vector in $\rr^n$ $(n\meg 2)$ which satisfies
\begin{equation} \label{e3.1}
 \ave[X_1]=0, \ \ \ \ \ave[X_1^2]= 1 \ \ \text{ and } \ \ \ \ave\big[|X_1|^3\big]<\infty.
\end{equation}
Then for every unit vector $\bbth\in \rr^n$, setting
$\sigma^2\coloneqq \mathrm{Var}(\theta_1X_1+\dots+\theta_n X_n)$, we have
\begin{align} \label{e3.2}
d_K\Big( \sum_{i=1}^n \theta_i X_i, \mathcal{N}(0,\sigma^2)\Big) & \mik
5\mathrm{osc}(\bbx)+ 6|\delta_0|+ \frac{\kappa_1}{n} +
\Big(|\delta_0|+\frac{1}{n}\Big)\cdot \Big(\sum_{i=1}^n\theta_i\Big)^2 + \\
& \hspace{2cm} + \kappa_1\, \ave\big[|X_1|^3\big] \, \sum_{i=1}^n |\theta_i|^3. \nonumber
\end{align}
Here, $\kappa_1\meg 1$ is an absolute constant. In fact, we can take $\kappa_1= 4320$.
\end{cor}
It follows, in particular, that if $\bbx$ is as in Corollary \ref{c3.1} and it satisfies
$|\delta_0|\mik C/n$ for some positive constant $C$, then for every positive integer $k\mik n$ we have
\begin{equation} \label{e3.3}
d_K\Big( \frac{X_1+\dots+X_k}{\sqrt{k}}, \mathcal{N}(0,1)\Big) \mik
5 \mathrm{osc}(\bbx)+ (8C+\kappa_1+1)\, \frac{k}{n} +
\kappa_1\, \ave\big[|X_1|^3\big]\, \frac{1}{\sqrt{k}}.
\end{equation}
This bound is optimal up to universal constants; see Paragraph \ref{subsubsec12.1.1} for details.

\subsubsection{Estimating the oscillation} \label{subsubsec3.1.1}

By Corollary \ref{c3.1}, the problem of establishing normal approximation of the random variable
$\theta_1 X_1+\cdots+\theta_n X_n$ reduces to that of estimating the oscillation
$\mathrm{osc}(\bbx)$ of the exchangeable random vector $\bbx$.

To this end, we first observe that if the entries of $\bbx$ have zero mean, unit variance
and finite fourth moment, then we have
\begin{equation} \label{e3.4}
\mathrm{osc}(\bbx) \mik \sqrt{ \big|\ave[X_1^2 X_2^2]-1\big|} + \frac{\ave[X_1^4]^{1/2}}{\sqrt{n}}.
\end{equation}
(See Fact \ref{f12.2}.) The quantity $\big|\ave[X_1^2X_2^2]-1\big|$
that appears in \eqref{e3.4} is very natural in this context: it controls the variance
$\mathrm{Var}\big(\|\bbx\|^2_{\ell_2}\big)$ of the square of the euclidean norm of~$\bbx$;
see, \textit{e.g.}, \cite{BCG18}. On the other hand, as we shall see in Proposition \ref{p12.3},
without further assumptions on the existence of moments, the basic hypothesis~\eqref{e3.1} yields~that
\begin{equation} \label{e3.5}
\mathrm{osc}(\bbx) \mik \sqrt{ \big|\ave[X_1^2 X_2^2]-1\big|} +
\frac{4\ave\big[|X_1|^3\big]}{\sqrt[4]{n}}.
\end{equation}

\subsubsection{Exchangeable and isotropic random vectors} \label{subsubsec3.1.2}

Combining \eqref{e3.2} and \eqref{e3.4}, we see that if $\bbx$ is an exchangeable and
isotropic\footnote{Recall that a random vector $\bbx$ in $\rr^n$ is called \textit{isotropic}
if its entries are uncorrelated random variables with zero mean and unit variance.}
random vector in $\rr^n$ whose entries have finite fourth moment, then, setting
$\tau=\tau(\bbx)\coloneqq n\big|\ave[X_1^2 X_2^2]-1\big|$, for every unit
vector $\bbth\in \rr^n$~we~have
\begin{align} \label{e3.6}
d_K\Big( \sum_{i=1}^n \theta_i X_i, \mathcal{N}(0,1)\Big) & \mik
\big(5\sqrt{\tau}+5\ave[X_1^4]^{1/2}+\kappa_1\big)\, \frac{1}{\sqrt{n}} +
\frac{1}{n}\, \Big(\sum_{i=1}^n \theta_i\Big)^2 + \\
& \hspace{3cm} + 2\kappa_1\, \ave\big[|X_1|^3\big]\, \sum_{i=1}^n |\theta_i|^3. \nonumber
\end{align}
The estimate \eqref{e3.6} extends and improves a result of Bobkov \cite[Proposition 6.1]{Bob04}.
It~shows, among others, that under the thin-shell condition $\big|\ave[X_1^2 X_2^2]-1\big|=O(1/n)$
we have the Berry--Esseen bound
\[ d_K\Big( \frac{X_1+\dots+X_k}{\sqrt{k}}, \mathcal{N}(0,1)\Big) =O\Big(\frac{1}{\sqrt{k}}\Big) \]
in the regime $k= O(n^{2/3})$. This range is also optimal; see Example \ref{ex12.5}.


\subsection{High-dimensional examples} \label{subsec3.2}

We proceed to discuss the general case of high-dimensional random tensors which satisfy
(\hyperref[A1]{$\mathcal{A}$1}) and (\hyperref[A2]{$\mathcal{A}$2}). It is clear that the remaining
task is to estimate their oscillation. Not surprisingly, to this end we will need an analogue of the
quantity $\big|\ave[X_1^2X_2^2]-1\big|$ that appears in \eqref{e3.4} and \eqref{e3.5}.

Specifically, let $n,d$ be positive integers with $n\meg 4d$, let
$\xtensor$ be a random tensor, and define the \emph{parallelepipedal correlation} of $\bbx$ by
\begin{equation} \label{e3.7}
\mathrm{pc}(\bbx)\coloneqq \big|\ave[ X_{(1,\dots,d)}X_{(1,d+1,\dots,2d-1)} X_{(2d,\dots,3d-1)}
X_{(2d,3d,\dots,4d-2)}]-\delta_1^2\big|.
\end{equation}
For example, if $d=2$, then $\mathrm{pc}(\bbx)=\big|\ave[X_{(1,2)}X_{(1,3)}X_{(4,5)}X_{(4,6)}]-\delta_1^2\big|$.

\begin{figure}[htb] \label{figure1}
\centering \includegraphics[width=.37\textwidth]{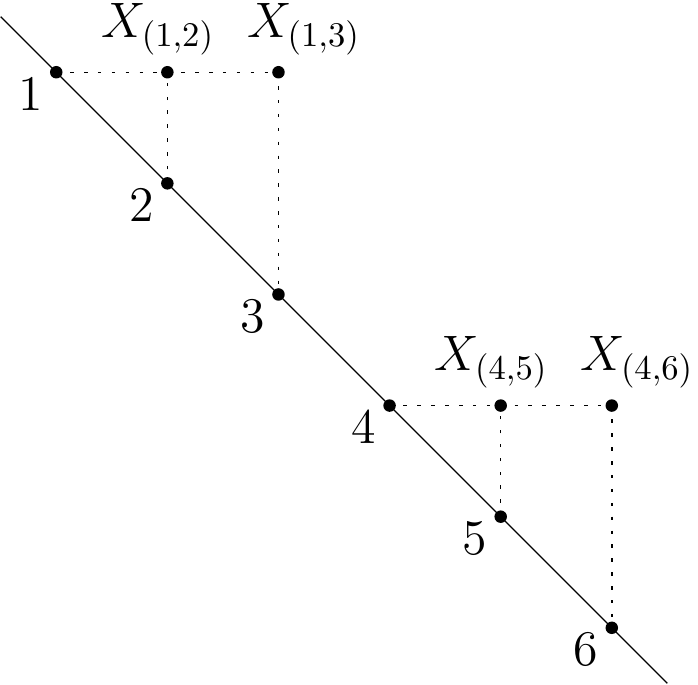}
\caption{The parallelepipedal correlation.}
\end{figure}

\noindent The parameter $\mathrm{pc}(\bbx)$ has already appeared---though, not explicitly
isolated---in \cite[Theorem~4]{EW78}. Notice that if $\bbx$ is a random vector,
then $\mathrm{pc}(\bbx)=\big|\ave[X_1^2X_2^2]-\delta_1^2\big|$; in particular,
the parallelepipedal correlation $\mathrm{pc}(\bbx)$ of random vectors which satisfy \eqref{e3.1}
coincides with the aforementioned quantity $\big|\ave[X_1^2X_2^2]-1\big|$.

\subsubsection{Random tensors whose entries have bounded fourth moment} \label{subsubsec3.2.1}

We first observe that if a random tensor~$\bbx$ satisfies (\hyperref[A1]{$\mathcal{A}$1}) and (\hyperref[A2]{$\mathcal{A}$2})
and its entries have bounded fourth moment, then we have the following analogue of \eqref{e3.4},
\begin{equation} \label{e3.8}
\osc(\bbx)\mik \sqrt{\mathrm{pc}(\bbx)}+
\frac{5d}{\sqrt{n}} \big(1+\ave[X_{(1,\dots,d)}^4]\big)^{1/2}.
\end{equation}
(See Fact \ref{f12.6} for a slightly more precise estimate.)

\subsubsection{Dissociated random tensors and their mixtures} \label{subsubsec3.2.2}

Another natural class of random tensors whose oscillation can be effectively estimated is that
of exchangeable and dissociated random tensors. Recall that a random tensor $\xtensor$ is called
\textit{dissociated} if for every pair $K,L$ of disjoint subsets of $[n]$ with $|K|,|L|\meg d$,
the random tensors $\bbx_K\coloneqq \langle X_i : i\in K^d\rangle$
and $\bbx_L\coloneqq \langle X_i : i\in L^d\rangle$ are independent.
Dissociativity plays an important role in the general theory of exchangeable
random tensors; see \cite{Ald81,Hoo79,Kal05}. To see its relevance in this context,
notice that $\mathrm{pc}(\bbx)=0$ for every exchangeable and dissociated random tensor.
That said, we have
\begin{equation} \label{e3.9}
\osc(\bbx)\mik \frac{16d\, \big(\ave\big[|X_{(1,\dots,d)}|^3\big]+1\big)}{\sqrt[4]{n}}
\end{equation}
for every dissociated random tensor $\bbx$ which satisfies (\hyperref[A1]{$\mathcal{A}$1})
and (\hyperref[A2]{$\mathcal{A}$2}); more generally, if $\bbx$ is a mixture of exchangeable,
symmetric and dissociated random tensors and the entries of $\bbx$ have finite third moment, then
\begin{equation} \label{e3.10}
\osc(\bbx)\mik \sqrt{\mathrm{pc}(\bbx)} + \frac{16d\, \big(\ave\big[|X_{(1,\dots,d)}|^3\big]+1\big)}{\sqrt[4]{n}}.
\end{equation}
(See Proposition \ref{p12.7}). These estimates together with Theorem \ref{t1.4} cover a wide
range of exchangeable random tensors, including all infinitely extendible random tensors which
satisfy (\hyperref[A1]{$\mathcal{A}$1}) and (\hyperref[A2]{$\mathcal{A}$2})---see
Paragraph \ref{subsubsec12.2.1} for more details.

\subsection{Anticoncentration of polynomials of boolean random variables} \label{sec3.3}

Recall that, given a real-valued random variable~$X$, its \textit{anticoncentration}
\begin{equation} \label{e3.11}
\sup_{x\in \rr} \prob(X=x)
\end{equation}
is a quantitative measure of its discreteness. A closely related, and more informative,
quantity is the \textit{L\'{e}vy concentration function of $X$},
\begin{equation} \label{e3.12}
\mathcal{L}_X(\ee)\coloneqq \sup_{x\in \rr} \prob(x\mik X\mik x+\ee)
\end{equation}
which bounds the probability that $X$ lies in an interval of length $\ee>0$.
The first results on anticoncentration were discovered by Littlewood/Offord \cite{LO43}
and Erd\H{o}s \cite{Erd45} who obtained optimal anticoncentration of linear functions of
random vectors with i.i.d. Rademacher entries.

Much more recently, Costello, Tao and Vu \cite{CTV06} have put forth\footnote{We note
that closely related questions have been studied earlier---see, \textit{e.g.}, \cite{RS96}.}
a higher-degree version of the classical Littlewood--Offord theory whose main goal is to understand
the anticoncentration of random variables of the form $f(\boldsymbol{\xi})$ where
$f\colon \rr^n\to\rr$ is a polynomial with real coefficients, and $\boldsymbol{\xi}$
is a random vector in $\rr^n$ with a well-behaved distribution. Examples of random
vectors which have been studied in this context include:
\begin{enumerate}
\item[$\bullet$] random vectors with i.i.d. gaussian/Rademacher/Bernoulli entries, and
\item[$\bullet$] random vectors which are uniformly distributed on a slice;
\end{enumerate}
see, \textit{e.g.}, \cite{CTV06,CW01,FKMW18,FM19,MNV16,MOO10,KST19,LLTTY17,RV13}.
The topic is quite diverse, and it has found a variety of applications in analysis,
combinatorics, discrete probability and theoretical computer science; we refer the
reader to \cite{NV13,Vu14} for recent expositions.

\subsubsection{The L\'{e}vy concentration function of polynomials of boolean random variables}
\label{subsec3.4.1}

Theorem \ref{t1.4} can be used to estimate the L\'{e}vy concentration function of homogeneous
polynomials of an important class of random vectors with boolean but not independent entries.

Specifically, let $n\meg d$ be positive integers, and let $f\colon \rr^n\to\rr$
\begin{equation} \label{e3.13}
f(x_1,\dots,x_n)= \sum_{F\in \binom{[n]}{d}} a_F \, \prod_{i\in F}x_i
\end{equation}
be a homogeneous multilinear polynomial of degree $d$, with no constant term
and real coefficients $\boldsymbol{a}=\langle a_F : F\in \binom{[n]}{d}\rangle$.
For every $s\in \{0,1,\dots,d\}$ we set
\begin{equation} \label{e3.14}
\seminorm{\boldsymbol{a}}_s\coloneqq \Big( \sum_{G\in \binom{[n]}{s}}
\big(\!\!\sum_{G\subseteq F\in \binom{[n]}{d}} \!\! a_F \big)^2 \Big)^{1/2}
\end{equation}
where we use again the convention that the first sum vanishes if $s=0$; these seminorms
are, of course, the analogues of the seminorms introduced in \eqref{e1.5}.

We have the following theorem. (The proof is given in Subsection \ref{subsec12.3}.)
\begin{thm} \label{t3.2}
Let $n,d$ be positive integers with $n\meg (4\kappa)^{2d}$ where $\kappa=20 d^3 18^d (2d)!$
is as in Theorem \emph{\ref{t1.4}}, and let $f\colon \rr^n\to\rr$ be as in~\eqref{e3.13}.
Also let $k$ be a positive integer~with
\begin{equation} \label{e3.15}
(2\kappa)\, n^{1-\frac{1}{2d}} \mik k \mik n- (2\kappa)\, n^{1-\frac{1}{2d}}
\end{equation}
and let $\boldsymbol{\xi}=(\xi_1,\dots,\xi_n)$ be a random vector in $\{0,1\}^n$ which is uniformly
distributed\,\footnote{Here, we identify $\binom{[n]}{k}$ with the set of all $x\in\{0,1\}^n$
which have exactly $k$ nonzero coordinates.} on~$\binom{[n]}{k}$. Then, setting $p\coloneqq k/n$
and $\sigma^2\coloneqq \mathrm{Var}\big(f(\boldsymbol{\xi})\big)$, we have
\begin{equation} \label{e3.16}
\Big| \sigma^2-\sum_{s=1}^d p^{2d-s}(1-p)^s \, \seminorm{\boldsymbol{\alpha}}^2_s\Big| \mik
\frac{12d^2 2^d}{(1-p)^d}\cdot \frac{\seminorm{\boldsymbol{\alpha}}_0^2}{n} +
\frac{12d^2 2^d}{p(1-p)^d}\cdot \frac{\sigma^2}{n}
\end{equation}
and, for every $\ee>0$\footnote{We follow the convention that $\frac{1}{0}=\infty$.},
\begin{align} \label{e3.17}
\mathcal{L}_{f(\boldsymbol{\xi})}(\ee) & \mik \frac{\ee}{\sqrt{2\pi}\,\sigma} +
\frac{16\kappa}{\sqrt{p}}\cdot \frac{1}{\sqrt{n}} +
\frac{12d^2}{(1-p)\seminorm{\boldsymbol{\alpha}}^2_1}\cdot \frac{\seminorm{\boldsymbol{\alpha}}^2_0}{n}\ + \\
& \hspace{1.1cm} + \frac{2^{39}p^{3/2}}{p^{3d}(1-p)^{3/2} \seminorm{\boldsymbol{\alpha}}_1^3}\cdot
\sum_{j=1}^n \Big| \sum_{j\in F\in \binom{[n]}{d}} a_F\Big|^3 \ + \nonumber \\
& \hspace{1.9cm} +  \frac{16\kappa p^{1/2}}{p^d(1-p)^{1/2} \seminorm{\boldsymbol{\alpha}}_1}\cdot
\sum_{s=2}^d \sqrt{p^{2d-s}(1-p)^s}\, \seminorm{\boldsymbol{\alpha}}_s. \nonumber
\end{align}
\end{thm}
\begin{rem} \label{r3.3}
Using Proposition \ref{p1.2}, it is easy to see that the left-hand-side of \eqref{e3.16} is equal to
$\big|\mathrm{Var}\big(f(\boldsymbol{\xi})\big)-\mathrm{Var}\big(f(\boldsymbol{b})\big)\big|$
where $\boldsymbol{\xi}$ is as in Theorem \ref{t3.2} and $\boldsymbol{b}$ is a random vector
in $\{0,1\}^n$ with i.i.d. Bernoulli entries with expectation $p$.
\end{rem}
In order to put Theorem \ref{t3.2} into context, fix $f$ as in \eqref{e3.13}, and recall that
if $\boldsymbol{G}$ is a random vector in $\rr^n$ with independent standard normal entries
such that $\mathrm{Var}\big(f(\boldsymbol{G})\big)=1$, then the Carbery--Wright inequality \cite{CW01}
yields that for every $\ee>0$,
\begin{equation} \label{e3.18}
\mathcal{L}_{f(\boldsymbol{G})}(\ee) = O(d\ee^{1/d}).
\end{equation}
Using their invariance principle, Mossel, O’Donnell and Oleszkiewicz \cite{MOO10} extended
this estimate to general random vectors with i.i.d. entries with an extra term on the right-hand-side
of \eqref{e3.18} depending on the influence of the coefficients of $f$; more recently, Filmus, Kindler,
Mossel and Wimmer \cite{FKMW18} also covered the case of random vectors $\boldsymbol{\xi}$ as in
Theorem \ref{t3.2} under an extra harmonicity assumption on the polynomial $f$.
(See, also, \cite{FM19} for related results).

On the other hand, if $\boldsymbol{\xi}$ is as in Theorem \ref{t3.2} with, say,
$\frac{n}{10}\mik k \mik \frac{9n}{10}$ and $\mathrm{Var}\big(f(\boldsymbol{\xi})\big)=1$,
then, by \eqref{e3.16} and~\eqref{e3.17}, for every $0<\ee\mik 1/2$ we have
\begin{equation} \label{e3.19}
\mathcal{L}_{f(\boldsymbol{\xi})}(\ee) = O(\ee)
\end{equation}
provided that $n$ is sufficiently large and
\[ \sum_{j=1}^n \Big|\!\sum_{j\in F\in \binom{[n]}{d}} \! a_F\Big|^3,
\frac{\seminorm{\boldsymbol{a}}_0^2}{n},
\seminorm{\boldsymbol{a}}_2,\dots,\seminorm{\boldsymbol{a}}_d\mik c_d\ee\]
where $c_d$ is a positive constant that depends only on $d$. In other words,
we have an improvement over~\eqref{e3.18} as long as the coefficients of $f$ are balanced
and not too dense, and the contribution of the seminorms $\seminorm{\boldsymbol{a}}_2,
\dots,\seminorm{\boldsymbol{a}}_d$ in $\mathrm{Var}\big(f(\boldsymbol{\xi})\big)$
is relatively small.

%
%


\section{Notation} \label{sec4}

\numberwithin{equation}{section}

In this section we collect all pieces of notation that are used throughout this paper.
Working with tensors (and, in general, with high-dimensional objects) necessitates some notation
which, unfortunately, is not completely standardized. Most of this notation will be used
in Sections \ref{sec5}, \ref{sec7} and \ref{sec9}. The reader should have in mind this remark,
and consult this particular section for any possible clarification while reading the paper.

As we have mentioned, for every integer $n\meg 1$ we set $[n]\coloneqq\{1,\dots,n\}$;
moreover, for every $k\in \{0,\dots,n\}$ by $\binom{[n]}{k}$ we denote the collection
of all subsets of $[n]$ of cardinality~$k$.

\subsection{\!} \label{subsec4.1}

For every positive integer $n$ and every (possibly empty) set $S$ by $[n]^S$ we denote the set of
all functions from $S$ into $[n]$, and by $[n]^S_{\mathrm{Inj}}$ the set of all one-to-one functions from
$S$ into~$[n]$. In particular, if $S$ is empty, then $[n]^S$ contains only the empty function
which shall be denoted by $\emptyset$; on the other hand, if $S=[d]$ for some positive integer $d$,
then the corresponding sets of functions shall be denoted simply by $[n]^d$ and~$[n]^d_{\mathrm{Inj}}$.
Finally, if $S$ is nonempty, then by $\mathrm{PartInj}(S,[n])$ we denote the set of all \textit{nonempty},
partial, one-to-one maps from $S$ into $[n]$. (We emphasize, again, the nonemptyness.)

\subsection{\!} \label{subsec4.2}

Given a function $f$, by $\mathrm{dom}(f)$ we shall denote its domain and by $\mathrm{Im}(f)$
its image. For every (possibly empty) subset $F$ of $\mathrm{dom}(f)$, by
$f\upharpoonright F$ we denote the restriction of $f$~on~$F$. If $g$ is another function, then
we write $f\sqsubseteq g$ provided that $\mathrm{dom}(f)\subseteq\mathrm{dom}(g)$ and
$g\upharpoonright\mathrm{dom}(f) = f$. Finally, if $\mathrm{Im}(f)\subseteq\mathrm{dom}(g)$,
then by $g\circ f$ we denote the composition of $g$ with $f$, that is, the map
$\mathrm{dom}(f)\ni i \mapsto g(f(i))\in \mathrm{Im}(g)$.

In particular, if
$n,d$ are positive integers, then for every $i=(i_1,\dots,i_d)\in [n]^d$,
every permutation $\pi$ of $[n]$ and every permutation~$\tau$ of $[d]$ we have
\begin{equation} \label{e4.1}
\pi\circ i\coloneqq\big(\pi(i_1),\dots,\pi(i_d)\big)\in [n]^d \ \ \ \text{ and } \ \ \
i\circ \tau\coloneqq \big(i_{\tau(1)},\dots,i_{\tau(d)}\big)\in [n]^d;
\end{equation}
moreover, $i\upharpoonright F\in [n]^F$ for every subset $F$ of $[d]$.

\subsection{\!} \label{subsec4.3}

Let $n,d$ be positive integers, and let $i,j,p,q\in [n]^d$. We say that the pairs $(i,j)$ and
$(p,q)$ are \textit{equivalent} if there exists a permutation $\pi$ of $[n]$ such that $p=\pi\circ i$
and $q=\pi\circ j$. We shall write $(i,j)\sim(p,q)$ to denote the fact that $(i,j)$ and
$(p,q)$ are equivalent.

\subsection*{Acknowledgments}

We would like to thank the anonymous referees for their comments, remarks and suggestions.

The research was supported by the Hellenic Foundation for Research and Innovation
(H.F.R.I.) under the “2nd Call for H.F.R.I. Research Projects to support
Faculty Members \& Researchers” (Project Number: HFRI-FM20-02717).


\part{Combinatorial CLT for high-dimensional tensors} \label{part2}


\section{A tool for computing the variance} \label{sec5}

\numberwithin{equation}{section}

The following proposition is the first step of the proof of Theorem \ref{t2.2}.
It is the main tool that enables us to compute the variance of several random
variables associated with $W$-statistics.
\begin{prop} \label{p5.1}
Let $n\meg s$ be positive integers, and let $\boldsymbol{\xi}\colon [n]^s \times [n]^s \to \rr$
be a Hoeffding tensor. Also let\,\footnote{That is, $t$ is either the empty map, or a nonempty,
partial, one-to-one map from $[s]$ into $[n]$.} $t\in \{\emptyset\}\cup\mathrm{PartInj}([s],[n])$
and $r\in \big\{|\mathrm{Im}(t)|,\dots,s\big\}$. Then,
\begin{equation} \label{e5.1}
\sum_{(i,j,p,q)\in \mathcal{M}_{t,r}} \!\!\! \boldsymbol{\xi}(i,p)\,\boldsymbol{\xi}(j,q) \mik
e^{2s}\,\big((2s)!\big)^2 n^{s-r} \sum_{\substack{i,p\in[n]^s\\t\sqsubseteq i}} \boldsymbol{\xi}(i,p)^2
\end{equation}
where $\mathcal{M}_{t,r}=\big\{(i,j,p,q)\in \big([n]^s_{\mathrm{Inj}}\big)^4: t\sqsubseteq i,\,
t\sqsubseteq j,\, (i,j)\sim(p,q) \text{ and } |\mathrm{Im}(i)\cap\mathrm{Im}(j)| = r\big\}$.
\end{prop}
(Recall that, for a quadruple $(i,j,p,q)$ of elements of $[n]^s$, we write $(i,j)\sim (p,q)$
to denote the fact $(i,j)$ and $(p,q)$ are equivalent; see Subsection \ref{subsec4.3}.)
For the proof of Proposition \ref{p5.1} we need to do some preparatory work. In what follows,
let $n,s,\boldsymbol{\xi},t,r$ be as in Proposition \ref{p5.1}.

\subsection{Partitions} \label{subsec5.1}

Consider the set $\{0,1\}\times [s]$ which we view as a partially ordered set equipped with the
lexicographical order $<_{\mathrm{lex}}$ defined by setting $(\mu,a)<_{\mathrm{lex}}(\nu,b)$
if either $\mu<\nu$, or $\mu=\nu$ and $a<b$. We denote by $\Pi$ the set of all partitions
(with nonempty~parts) of $\{0,1\}\times [s]$; we shall use the letters $P,Q,R,S$ to denote partitions.
Moreover, with every $(i,j)\in [n]^s\times [n]^s$ we associate a partition $P(i,j)\in \Pi$ defined by
\begin{equation} \label{e5.2}
\Big\{ \big\{(0,u) \in \{0\}\times[s] :  i(u)=x\big\} \cup
\big\{(1,v) \in \{1\}\times[s] : j(v)=x\big\}: x\in\mathrm{Im}(i)\cup\mathrm{Im}(j) \Big\}.
\end{equation}
We isolate, for future use, the following elementary fact.
\begin{observation} \label{o5.2}
If $i,j,p,q\in [n]^s$, then $(i,j)\sim(p,q)$ if and only if  $P(i,j)=P(p,q)$.
\end{observation}
Next, we set
\begin{equation} \label{e5.3}
\mathcal{I}\coloneqq \big\{P(i,j): (i,j) \in [n]^s_{\mathrm{Inj}}\times [n]^s_{\mathrm{Inj}},\,
|\mathrm{Im}(i)\cap\mathrm{Im}(j)|=r, \, t\sqsubseteq i \text{ and } t\sqsubseteq j\big\};
\end{equation}
notice that for every $P\in\mathcal{I}$, its cells have cardinality 1 or 2, and it has exactly
$r$ cells with cardinality 2. We also set
\begin{equation} \label{e5.4}
\mathcal{F} \coloneqq \big\{ P\in\Pi: \text{every cell of $P$ has cardinality at least } 2\big\}.
\end{equation}
We view $\mathcal{I}$ as the set of ``initial" partitions, and $\mathcal{F}$ as the set of
``final" partitions.

\subsection{Ordering partitions} \label{subsec5.2}

We about to define a partial order $\preccurlyeq$ on the set~$\Pi$. To this end,
first, for every $P\in\Pi\setminus\mathcal{F}$ let $X(P)\in P$ denote the $<_{\mathrm{lex}}$-minimal
element of $\{X\in P: |X|=1\}$; namely, among all cells of $P$ which are singletons, $X(P)$
is the lexicographically least. Next, for every $P\in\Pi\setminus\mathcal{F}$ and
every $Q\in\Pi$ we write $Q \prec_1 P$ if there exist $X_1\in Q$ and $X_2\in P$ such that
\begin{enumerate}
\item[($\mathcal{P}$1)] \label{P1} $X_2\neq X(P)$ and $X_1 = X_2 \cup X(P)$, and
\item[($\mathcal{P}$2)] \label{P2} for every $X\in Q$ with $X\neq X_1$ we have $X\in P$.
\end{enumerate}
Notice, in particular, that if $Q\prec_1 P$, then $|P| = |Q|+1$.

Finally, we define $\preccurlyeq$ to be the transitive closure of $\prec_1$; more precisely,
given $P, Q\in \Pi$, we write $Q \preccurlyeq P$ if either $P=Q$, or there exists a finite sequence
$(P_0,P_1,\dots,P_\ell)$ in $\Pi$ such that $Q=P_\ell\prec_1 \dots \prec_1 P_0=P$.
We shall refer to such a sequence $(P_0,P_1,\dots,P_\ell)$ as a \emph{path $($from $P$ to $Q$$)$}
and we call the positive integer $\ell$ as the \emph{length of the path}; moreover,
by $\alpha(P,Q)$ we shall denote the number of paths from $P$ to $Q$. We will also need
the following observation. (Its proof is straightforward.)
\begin{observation} \label{o5.3}
If\, $Q\preccurlyeq P$, then all paths from $P$ to $Q$ have the same length.
\end{observation}

\subsection{Successors} \label{subsec5.3}

For every $P\in \Pi$ we set
\begin{equation} \label{e5.5}
\mathrm{Succ}(P)\coloneqq \{Q\in\Pi:Q \preccurlyeq P\} \ \ \  \text{ and } \ \ \
\mathrm{SuccFin}(P)\coloneqq \mathrm{Succ}(P)\cap\mathcal{F}.
\end{equation}
Moreover, for every $P\in\Pi\setminus\mathcal{F}$ set
$\mathrm{Succ}_1(P)\coloneqq \{Q\in\Pi:Q \prec_1 P\}$. In what follows, we will be mainly interested
in the set
\begin{equation} \label{e5.6}
\mathcal{A}\coloneqq \bigcap_{P\in\mathcal{I}}\mathrm{Succ}(P)\subseteq \Pi.
\end{equation}

\subsection{Coding pairs} \label{subsec5.4}

For every $(i,j)\in [n]^s\times [n]^s$ let $P(i,j)$ be as in \eqref{e5.2} and define
$f_{i,j}\colon P(i,j)\to [n]$ by
\begin{equation} \label{e5.7}
f_{i,j}(X)\coloneqq
\begin{cases}
i(u) &  \text{if } (0,u)\in X, \\
j(v) & \text{if } (1,v)\in X.
\end{cases}
\end{equation}
To spell it out in more detail, for every $u\in [s]$ we have that $f_{i,j}(X)=i(u)$ where $X$
is the unique cell of the partition $P(i,j)$ that contains $(0,u)$; respectively, for every
$v\in [s]$ we have that $f_{i,j}(X)=j(v)$ where $X$ is the unique cell of $P(i,j)$ that contains $(1,v)$.
Observe that, by definition, we have that $f_{i,j}\in [n]^{P(i,j)}_{\mathrm{Inj}}$.

Our next goal is to show that the process of obtaining $P(i,j)$ and $f_{i,j}$ from the pair
$(i,j)$, can be reversed. More precisely, for every $P\in \Pi$ and every $f\in [n]^P_{\mathrm{Inj}}$ we define
$i_f,j_f\colon [s]\to[n]$ as follows. For every $u\in[s]$ we set $i_f(u) = f(X)$ where $X$ is the unique
cell of $P$ that contains $(0,u)$; respectively, for every $v\in[s]$ we set $j_f(v) = f(X)$ where $X$
is the unique cell of $P$ that contains $(1,v)$. It is clear that for every $P\in \Pi$ and
$f\in [n]^P_{\mathrm{Inj}}$ we have $P(i_f,j_f) = P$ and $f_{i_f,j_f}=f$. Moreover, for every
$(i,j)\in [n]^s\times [n]^s$~the~map
\begin{equation} \label{e5.8}
[(i,j)]_\sim \ni (p,q) \to f_{p,q} \in [n]^{P(i,j)}_{\mathrm{Inj}}
\end{equation}
is a bijection, and its inverse is given by
\begin{equation} \label{e5.9}
[n]^{P(i,j)}_{\mathrm{Inj}} \ni f \to (i_f,j_f) \in [(i,j)]_\sim
\end{equation}
where $[(i,j)]_\sim=\big\{(p,q)\in [n]^s\times [n]^s: (p,q)\sim(i,j)\big\}$ denotes the equivalence
class~of~$(i,j)$.

\subsection{The restriction of $t$} \label{subsec5.5}

Recall that in the left-hand-side of \eqref{e5.1} the sum is over all $(i,j)$ which
both extend the given partial map $t$. Having this constrain in mind,
for every $P\in \mathcal{A}$ we set
\begin{equation} \label{e5.10}
\mathcal{M}(P)\coloneqq \big\{f\in[n]^P_{\mathrm{Inj}}:t\sqsubseteq i_f \text{ and } t\sqsubseteq j_f\big\}.
\end{equation}
The set $\mathcal{M}(P)$ can be alternatively described as follows. Notice that
the map which sends each $k\in\mathrm{dom}(t)$ to the unique cell $X_k\in P$
that contains $(0,k)$, is one-to-one. Then $\mathcal{M}(P)$ consists of all
$f\in [n]^{P}_{\mathrm{Inj}}$ such that $f(X_k) = t(k)$ for every $k\in \mathrm{dom}(t)$.

\subsection{Using the fact that $\boldsymbol{\xi}$ is a Hoeffding tensor} \label{subsec5.6}

For every $P,Q\in \Pi$ with $Q\prec_1 P$ and every $g\in [n]^{P\setminus\{X(P)\}}_{\mathrm{Inj}}$
we define $T^{P\to Q}(g)\in [n]^{Q}_{\mathrm{Inj}}$ as follows. Let $X_1\in Q$ and $X_2\in P$
be the unique cells which satisfy (\hyperref[P1]{$\mathcal{P}$1}) and
(\hyperref[P2]{$\mathcal{P}$2}). Then we set
\begin{equation} \label{e5.11}
T^{P\to Q}(g)(X_1) = g(X_2)
\end{equation}
and $T^{P\to Q}(g)(X)=g(X)$ for every $X\in Q\setminus\{X_1\}$.

Using the material introduced so far and invoking the fact that $\boldsymbol{\xi}$
is a Hoeffding tensor, we see that for every $P\in \Pi\setminus\mathcal{F}$, every
$g\in [n]^{P\setminus\{X(P)\}}_{\mathrm{Inj}}$ and every $(p,q)\in [n]^s\times [n]^s$,
\begin{equation} \label{e5.12}
\sum_{\substack{f\in[n]^{P}_{\mathrm{Inj}} \\ f\upharpoonright P\setminus\{X(P)\}= g}}
\!\!\!\!\! \boldsymbol{\xi}(i_f,p)\, \boldsymbol{\xi}(j_f,q) = \ - \!\!\!
\sum_{Q\in\mathrm{Succ}_1(P)} \!\! \boldsymbol{\xi}(i_{T^{P\to Q}(g)},p)\,
\boldsymbol{\xi}(j_{T^{P\to Q}(g)},q)
\end{equation}
\begin{equation} \label{e5.13}
\sum_{\substack{f\in[n]^{P}_{\mathrm{Inj}} \\ f\upharpoonright P\setminus\{X(P)\} = g}}
\!\!\!\!\! \boldsymbol{\xi}(p,i_f)\, \boldsymbol{\xi}(q,j_f) = \ - \!\!\!
\sum_{Q\in\mathrm{Succ}_1(P)} \!\! \boldsymbol{\xi}(p,i_{T^{P\to Q}(g)})\,
\boldsymbol{\xi}(q,j_{T^{P\to Q}(g)}).
\end{equation}
Therefore, by \eqref{e5.10} and \eqref{e5.12}, for every $P\in \mathcal{A}\setminus\mathcal{F}$ and
every $(p,q)\in [n]^s\times [n]^s$ we have
\begin{equation} \label{e5.14}
\sum_{f\in\mathcal{M}(P)} \! \boldsymbol{\xi}(i_f,p)\, \boldsymbol{\xi}(j_f,q) = \ -
\!\! \sum_{Q\in\mathrm{Succ}_1(P)} \sum_{g\in\mathcal{M}(Q)} \!
\boldsymbol{\xi}(i_g,p)\, \boldsymbol{\xi}(j_g,q);
\end{equation}
respectively, by \eqref{e5.10} and \eqref{e5.13}, for every $R\in \Pi\setminus\mathcal{F}$
and every $(p,q)\in [n]^s\times [n]^s$,
\begin{equation} \label{e5.15}
\sum_{f\in[n]^R_{\mathrm{Inj}}} \! \boldsymbol{\xi}(p,i_f)\, \boldsymbol{\xi}(q,j_f) = \ -
\!\! \sum_{S\in\mathrm{Succ}_1(R)} \sum_{g\in[n]^S_{\mathrm{Inj}}} \!
\boldsymbol{\xi}(p,i_g)\, \boldsymbol{\xi}(q,j_g).
\end{equation}
We have the following claim. (Recall that for every $P,Q\in \Pi$ by $\alpha(P,Q)$ we denote the
number of paths from $P$ to $Q$.)
\begin{claim} \label{c5.4}
For every $P\in \mathcal{A}$ and every $(p,q)\in [n]^s\times [n]^s$ we have
\begin{equation} \label{e5.16}
\sum_{f\in\mathcal{M}(P)} \! \boldsymbol{\xi}(i_f,p)\, \boldsymbol{\xi}(j_f,q)
=\!\! \sum_{Q\in\mathrm{SuccFin}(P)} \!\!\!\!\! (-1)^{|P|-|Q|}\, \alpha(P,Q) \!
\sum_{g\in\mathcal{M}(Q)} \! \boldsymbol{\xi}(i_g,p)\, \boldsymbol{\xi}(j_g,q);
\end{equation}
respectively, for every $R\in \Pi$ and every $(p,q)\in [n]^s\times [n]^s$ we have
\begin{equation} \label{e5.17}
\sum_{f\in[n]^{R}_{\mathrm{Inj}}} \! \boldsymbol{\xi}(p,i_f)\, \boldsymbol{\xi}(q,j_f)
=\!\! \sum_{S\in\mathrm{SuccFin}(R)} \!\!\!\!\! (-1)^{|R|-|S|}\, \alpha(R,S) \!
\sum_{g\in[n]^{S}_{\mathrm{Inj}}} \! \boldsymbol{\xi}(p,i_g)\, \boldsymbol{\xi}(q,j_g).
\end{equation}
\end{claim}
\begin{proof}
We will only give the proof of \eqref{e5.16}; equality \eqref{e5.17} follows with identical arguments.
Fix $P\in\mathcal{A}$, and for every $Q\in \mathrm{Succ}(P)$, we set $\mathrm{dist}(P,Q)=|P|-|Q|$;
by Observation \ref{o5.3}, this quantity coincides with the length of an arbitrary path from $P$
to~$Q$. Moreover, for every positive integer $\ell$ we set
\begin{equation} \label{e5.18}
\mathrm{Succ}_\ell(P)\coloneqq \big\{Q\in \Pi:Q\preccurlyeq P \text{ and } \mathrm{dist}(P,Q)=\ell\big\}.
\end{equation}
Clearly, we have $Q\in \mathrm{Succ}_1(P)$ if and only if $Q\prec_1 P$. On the other hand,
using Observation \ref{o5.3} once again, we see that for every positive integer $\ell$ and every
$Q\in\mathrm{Succ}_{\ell+1}(P)$,
\begin{equation} \label{e5.19}
\alpha(P,Q)= \!\! \sum_{Q\prec_1 R\in\mathrm{Succ}_\ell(P)} \!\!\! \alpha(P,R).
\end{equation}
Finally, using \eqref{e5.14}, \eqref{e5.19} and proceeding by induction on $\ell$, we obtain that
\begin{align}
\label{e5.20} \sum_{f\in\mathcal{M}(P)} \! \boldsymbol{\xi}(i_f,p)\, \boldsymbol{\xi}(j_f,q)
& =\!\! \sum_{Q\in\mathrm{Succ}_\ell(P)}  \!\!\!\!\! (-1)^{|P|-|Q|}\, \alpha(P,Q) \!
\sum_{g\in\mathcal{M}(Q)} \! \boldsymbol{\xi}(i_g,p)\, \boldsymbol{\xi}(j_g,q) \ + \\
& \ \ \ \ \ \ + \!\!\sum_{\substack{Q\in\mathrm{SuccFin}(P)\\\mathrm{dist}(P,Q)<\ell}} \!\!\!\!\!
(-1)^{|P|-|Q|}\, \alpha(P,Q) \! \sum_{g\in\mathcal{M}(Q)} \!
\boldsymbol{\xi}(i_g,p)\, \boldsymbol{\xi}(j_g,q). \nonumber
\end{align}
Equality \eqref{e5.16} follows from \eqref{e5.20} for $\ell=|P|$.
\end{proof}

\subsection{Estimating the number of paths} \label{subsec5.7}

Notice, first, that the number of partitions of $\{0,1\}\times[s]$---known as the Bell number $B_{2s}$---is
at most $(2s)^{2s}$; also observe that every path has length at most $2s$. Thus, we immediately obtain that
$\alpha(P,Q)\mik (2s)^{4s^2}$ for every $P,Q\in \Pi$ with $Q\preccurlyeq P$. We can improve this estimate as follows.
\begin{fact} \label{f5.5}
For every $P\in \Pi$ and every $Q\in \mathrm{SuccFin}(P)$ we have $\alpha(P, Q)\mik e^{s}$.
\end{fact}
\begin{proof}
The result is a consequence of some rigidity properties of the partial order $\preccurlyeq$.
Indeed, fix $P\in \Pi$ and $Q\in \mathrm{SuccFin}(P)$, and notice that the cells of $Q$
can be categorized according to whether they contain ``large" cells of $P$. More precisely,
let $Y\in Q$ be arbitrary. First, assume that $Y$ contains a cell $X\in P$ with $|X|\meg 2$;
then the definitions of $\prec_1$ and $\preccurlyeq$ in Section \ref{subsec5.2} imply that,
in any possible path from $P$ to $Q$, there exists a unique way to glue the cells
$\{X\in P: X\subseteq Y\}$ and arrive at the cell $Y$. Next, suppose that every cell of $P$
which is contained in $Y$ is a singleton; in this case, observe that there are precisely $|Y|-1$
ways to glue the cells $\{X\in P: X\subseteq Y\}$. Taking these remarks into account,
we see that the number of paths from $P$ to $Q$ is at most
\begin{equation} \label{e5.21}
\max\Big\{ \prod_{r=1}^{\ell} a_r:\ell\in [s], \, a_1,\dots,a_\ell\in [2s] \text{ and }
a_1+\dots+a_\ell=2s-\ell\Big\}.
\end{equation}
A straightforward computation shows that this number is at most $\exp\big(\frac{2s}{e}\big)\mik e^s$.
\end{proof}
		
\subsection{Proof of Proposition \ref{p5.1}} \label{subsec5.8}		
		
By \eqref{e5.16}, \eqref{e5.17}, Fact \ref{f5.5} and the Cauchy--Schwarz inequality,
for every $P\in \mathcal{I}$ and every $R\in \Pi$ we have
\begin{align}
\label{e5.22} \sum_{\substack{f\in\mathcal{M}(P)\\g\in[n]^R_{\mathrm{Inj}}}} \!\!
\boldsymbol{\xi}(i_f,&i_g)\,  \boldsymbol{\xi}(j_f,j_g)  =
\sum_{\substack{Q\in\mathrm{SuccFin}(P)\\S\in\mathrm{SuccFin}(R)}} \!\!\!
(-1)^{|P|+|R|-|Q|-|S|} \, \alpha(P,Q)\, \alpha(R,S) \ \times \\
& \hspace{3.6cm} \times \!\! \sum_{\substack{f'\in\mathcal{M}(Q)\\g'\in[n]^S_{\mathrm{Inj}}}} \!\!
\boldsymbol{\xi}(i_{f'},i_{g'}) \, \boldsymbol{\xi}(j_{f'},j_{g'}) \nonumber \\
& \mik e^{2s} \!\! \sum_{\substack{Q\in\mathrm{SuccFin}(P)\\S\in\mathrm{SuccFin}(R)}}
\Big( \! \sum_{\substack{f'\in\mathcal{M}(Q)\\g'\in[n]^{S}_{\mathrm{Inj}}}}
\boldsymbol{\xi}(i_{f'},i_{g'})^2 \Big)^{1/2}
 \Big( \! \sum_{\substack{f'\in\mathcal{M}(Q)\\g'\in[n]^S_{\mathrm{Inj}}}}
\boldsymbol{\xi}(j_{f'},j_{g'})^2 \Big)^{1/2}. \nonumber
\end{align}
For every $P\in \Pi$ set $\mathcal{X}_0(P)\coloneqq \big\{X \in P: X\subseteq\{0\}\times[s]\big\}$,
$\mathcal{X}_1(P)\coloneqq \big\{X \in P: X\subseteq\{1\}\times[s]\big\}$ and
$\mathcal{X}_{0,1}(P)\coloneqq P\setminus \big(\mathcal{X}_0(P) \cup \mathcal{X}_1(P)\big)$;
notice that if $Q\in \mathrm{SuccFin}(P)$, then
\begin{enumerate}
\item[($\mathcal{P}$3)] \label{P3} $|\mathcal{X}_0(Q)|\mik |\mathcal{X}_0(P)|/2$ and
$|\mathcal{X}_1(Q)|\mik |\mathcal{X}_1(P)|/2$,
\item[($\mathcal{P}$4)] \label{P4} for every $X\in\mathcal{X}_{0,1}(P)$
there exists $Y\in\mathcal{X}_{0,1}(Q)$ such that $X\subseteq Y$.
\end{enumerate}
Also observe that for every $(i,j)\in [n]^s_{\mathrm{Inj}}\times [n]^s_{\mathrm{Inj}}$ such that
$|\mathrm{Im}(i)\cap\mathrm{Im}(j)|=r$ we have
$|\mathcal{X}_0(P(i,j))| = |\mathcal{X}_1(P(i,j))|=s-r$. Therefore, using the bijection
described in \eqref{e5.8} and~\eqref{e5.9}, we obtain that
\begin{align}
\label{e5.23} & \sum_{(i,j,p,q)\in\mathcal{F}_{t,r}} \boldsymbol{\xi}(i,p)\, \boldsymbol{\xi}(j,q)
= \sum_{P\in\mathcal{I}} \sum_{\substack{f\in\mathcal{M}(P)\\g\in[n]^P_{\mathrm{Inj}}}}
\boldsymbol{\xi}(i_f,i_g)\, \boldsymbol{\xi}(j_f,j_g) \\
& \hspace{0.4cm} \stackrel{\eqref{e5.22}}{\mik} e^{2s} \sum_{P\in\mathcal{I}} \sum_{Q,S\in\mathrm{SuccFin}(P)}
\!\! \Big(\! \sum_{\substack{f'\in\mathcal{M}(Q)\\g'\in[n]^S_{\mathrm{Inj}}}}
\boldsymbol{\xi}(i_{f'},i_{g'})^2 \Big)^{1/2}
\Big(\! \sum_{\substack{f'\in\mathcal{M}(Q)\\g'\in[n]^S_{\mathrm{Inj}}}}
\boldsymbol{\xi}(j_{f'},j_{g'})^2 \Big)^{1/2} \nonumber \\
& \hspace{0.4cm} \mik e^{2s} \sum_{P\in\mathcal{I}} \sum_{Q,S\in\mathrm{SuccFin}(P)}
\Big( n^{|\mathcal{X}_1(Q)|+|\mathcal{X}_1(S)|} \sum_{\substack{i,p\in[n]^s\\t\sqsubseteq i}}
\boldsymbol{\xi}(i,p)^2\Big)^{1/2} \ \times \nonumber \\
& \hspace{5.4cm} \times  \Big( n^{|\mathcal{X}_0(Q)|+|\mathcal{X}_0(S)|}
\sum_{\substack{j,q\in[n]^s\\t\sqsubseteq j}} \boldsymbol{\xi}(j,q)^2\Big)^{1/2} \nonumber \\
& \hspace{0.4cm} \mik e^{2s}\, n^{s-r} \sum_{P\in\mathcal{I}} \sum_{Q,S\in\mathrm{SuccFin}(P)}
\sum_{\substack{i,p\in[n]^s\\t\sqsubseteq i}} \boldsymbol{\xi}(i,p)^2 \ \mik \
e^{2s} \big((2s)!\big)^2 n^{s-r} \!\!
\sum_{\substack{i,p\in[n]^s\\t\sqsubseteq i}} \boldsymbol{\xi}(i,p)^2 \nonumber
\end{align}
where in the last inequality we have used the fact that
\begin{equation} \label{e5.24}
|\mathcal{I}|\mik \Big(\frac{s!}{(s-r)!}\Big)^2 \ \ \ \ \text{ and } \ \ \ \
|\mathrm{SuccFin}(P)|\mik(2s-r-1)!
\end{equation}
for every $P\in \Pi$. The proof of Proposition \ref{p5.1} is completed.


\section{An exchangeable pair of random permutations} \label{sec6}

\numberwithin{equation}{section}

Let $X,Y$ be random variables which are defined on a common probability space and
take values in a common measurable space, and recall that the pair $(X,Y)$ is called
\textit{exchangeable} if $(X,Y)$ and $(Y,X)$ have the same distribution. This notion
was introduced by Stein \cite{St86}, and it is a key concept in his method for normal
approximation as well as in several related developments (see, \textit{e.g.}, \cite{Ch14}).

An important ingredient of the proof of Theorem \ref{t2.2} is an exchangeable pair of random permutations
which also appears in the works of Bolthausen \cite{Bo84} and Barbour/Chen~\cite{BC05}.
Specifically, given an integer $n\meg 2$, we fix a triple $I_1,I_2,\pi_1$ of independent
random variables such that $I_1, I_2$ are uniformly distributed on $[n]$ and $\pi_1$
is uniformly distributed on $\mathbb{S}_n$, and we set
\begin{equation} \label{e6.1}
J_1\coloneqq\pi_1(I_1), \ \ \ \ J_2\coloneqq\pi_1(I_2) \ \ \text{ and } \ \
\pi_2\coloneqq \pi_1\circ t(I_1,I_2)
\end{equation}
where for every $i_1,i_2\in [n]$ by $t(i_1,i_2)\in\mathbb{S}_n$ we denote the transposition\footnote{That is,
$t(i_1,i_2)$ is the unique permutation which maps $i_1$ to $i_2$ and $i_2$ to $i_1$, and it is the identity
on $[n]\setminus \{i_1,i_2\}$.} of $i_1$ and $i_2$. We will need the following basic (and well-known)
properties of this construction.
\begin{enumerate}
\item[($\mathcal{E}$1)] \label{E1} $\pi_1$ and $(I_1,I_2)$ are independent.
\item[($\mathcal{E}$2)] \label{E2} $\pi_2$ and $(I_1,I_2)$ are independent.
\item[($\mathcal{E}$3)] \label{E3} $\pi_1$ and $\pi_2$ are uniformly distributed on $\mathbb{S}_n$.
\item[($\mathcal{E}$4)] \label{E4} $(\pi_1,\pi_2)$ is an exchangeable pair.
\end{enumerate}
Property (\hyperref[E1]{$\mathcal{E}$1}) is an immediate consequence of the relevant definitions.
Properties (\hyperref[E2]{$\mathcal{E}$2}) and (\hyperref[E3]{$\mathcal{E}$3}) follow from the fact
that for every $\ell_1,\ell_2\in [n]$ and every $\pi\in \mathbb{S}_n$ we have that
$\mathbb{P}\big(\pi_2 = \pi\, |\, (I_1,I_2) = (\ell_1,\ell_2)\big) = \frac{1}{n!}$; indeed, observe that
\begin{align} \label{e6.2}
\mathbb{P}\big(\pi_2 = \pi\, |\, (I_1,I_2) = (\ell_1,\ell_2)\big) & =
\mathbb{P}\big(\pi_1 \circ t(\ell_1,\ell_2) = \pi\, |\, (I_1,I_2) = (\ell_1,\ell_2)\big) \\
& = \mathbb{P}\big(\pi_1 = \pi \circ t(\ell_1,\ell_2)^{-1}\, |\, (I_1,I_2) = (\ell_1,\ell_2)\big)
= \frac{1}{n!}. \nonumber
\end{align}
Finally, for property (\hyperref[E4]{$\mathcal{E}$4}) notice that, by (\hyperref[E1]{$\mathcal{E}$1}),
(\hyperref[E2]{$\mathcal{E}$2}) and (\hyperref[E3]{$\mathcal{E}$3}), the pairs
$\big(\pi_1, (I_1,I_2)\big)$ and $\big(\pi_2, (I_1,I_2)\big)$ have the same distribution; consequently,
the pairs $\big(\pi_1,\pi_1\circ t(I_1,I_2)\big)$ and $\big(\pi_2, \pi_2\circ t(I_1,I_2)\big)$ also
have the same distribution. By \eqref{e6.1} we have $\pi_2=\pi_1\circ t(I_1,I_2)$ which implies, in particular,
that $\pi_2\circ t(I_1,I_2)=\pi_1 \circ t(I_1,I_2)\circ t(I_1,I_2)=\pi_1$. Combining the previous
observations, we conclude that $(\pi_1,\pi_2)$ is an exchangeable pair, as desired.


\section{Proof of Theorem \ref*{t2.2}} \label{sec7}

\numberwithin{equation}{section}

We have already pointed out that the one-dimensional case of Theorem \ref{t2.2} is due to Bolthausen,
while the two-dimensional case is due to Barbour/Chen. That said, the proof for all high-dimensional
cases is uniform, and so, in what follows we fix two positive integers $n,d$ with $d\meg 2$ and $n\meg 4d^2$.
Let $\boldsymbol{\xi}_1,\dots,\boldsymbol{\xi}_d$ be as in Theorem \ref{t2.2} and recall that
$\beta_1=n-1$, where $\beta_s$ is as in \eqref{e2.6} for every $s\in [d]$; moreover,
let $\pi_1, \pi_2$ be the random permutations described in Section \ref{sec6} and set
\begin{equation} \label{e7.1}
\Xi_s=\sum_{i\in [n]^s_{\mathrm{Inj}}} \! \boldsymbol{\xi}_s(i,\pi_1\circ i) \ \ \
\text{ and } \ \ \ \ \Xi'_s=\sum_{i\in [n]^s_{\mathrm{Inj}}} \! \boldsymbol{\xi}_s(i,\pi_2\circ i).
\end{equation}
Note that, by property (\hyperref[E4]{$\mathcal{E}$4}), the pair $(\Xi_s,\Xi'_s)$ is exchangeable
for every $s\in [d]$. We~also~set
\begin{equation} \label{e7.2}
\Lambda\coloneqq \sum_{i,j=1}^n |\boldsymbol{\xi}_1(i,j)|^3.
\end{equation}
In the following lemma we collect some known properties of the exchangeable pair $(\Xi_1,\Xi_1')$;
for a proof, see \cite{BC05,Bo84,ZBCL97}.
\begin{lem} \label{l7.1}
We have
\begin{align}
\label{e7.3} & \ave[\Xi_1^2] = 1 \\
\label{e7.4} & \ave[\Xi_1 - \Xi_1'\, |\, \pi_1] = \frac{2}{n}\, \Xi_1 \\
\label{e7.5} & \ave\big[(\Xi_1-\Xi_1')^2\big] = \frac{4}{n} \\
\label{e7.6} & \ave\big[|\Xi_1-\Xi_1'|^3\big] \mik 64\, \frac{\Lambda}{n^2}.
\end{align}
\end{lem}

The next lemma complements Lemma \ref{l7.1} and provides analogous estimates
for the rest of the exchangeable pairs $(\Xi_s,\Xi_s')$. We note that it is precisely
in the proof of this lemma where Proposition \ref{p5.1} is applied.
\begin{lem} \label{l7.2}
For every $s\in \{2,\dots,d\}$ we have
\begin{align}
\label{e7.7} & \ave[\Xi_s^2] \mik 2 e^{2s}\big((2s)!\big)^2\, \frac{\beta_s}{n^s} \\
\label{e7.8} & \ave\big[(\Xi_s - \Xi_s')^2\big] \mik 24s^4 e^{2s}\big((2s)!\big)^2\,
\frac{\beta_s}{n^{s+1}}.
\end{align}
\end{lem}
\begin{proof}
Fix $s\in\{2,\dots,d\}$. For every $r\in \{0,\dots,s\}$ set
\begin{enumerate}
\item[$\bullet$] $\mathcal{F}_r^1 \coloneqq \big\{(i,j)\in [n]^s_{\mathrm{Inj}}\times [n]^s_{\mathrm{Inj}}:
|\mathrm{Im}(i)\cap\mathrm{Im}(j)|=r\big\}$ and
\item[$\bullet$] $\mathcal{F}_r^2 \coloneqq \big\{ (i,j,p,q)\in \big([n]^s_{\mathrm{Inj}}\big)^4:
|\mathrm{Im}(i)\cap\mathrm{Im}(j)|=r \text{ and } (i,j)\sim(p,q)\big\}$.
\end{enumerate}
Then observe that
\begin{align}
\label{e7.9} & \ave[\Xi_s^2] = \ave\Big[ \sum_{i,j\in[n]^s_{\mathrm{Inj}}}
\boldsymbol{\xi}_s(i,\pi_1\circ i)\, \boldsymbol{\xi}_s(j,\pi_1\circ j)\Big] = \\
= \sum_{r=0}^{s} \sum_{(i,j)\in\mathcal{F}_r^1} & \ave\big[\boldsymbol{\xi}_s(i,\pi_1\circ i)\,
\boldsymbol{\xi}_s(j,\pi_1\circ j)\big] =
\sum_{r=0}^{s} \sum_{(i,j,p,q)\in\mathcal{F}_r^2} \!\! \frac{(n-2s+r)!}{n!}\,
\boldsymbol{\xi}_s(i,p)\, \boldsymbol{\xi}_s(j,q). \nonumber
\end{align}
Also notice that for every $k\mik \sqrt{n}$ we have
\begin{equation} \label{e7.10}
\frac{(n-k)!}{n!}\mik\frac{1}{n^k}\cdot\frac{1}{1-\binom{k}{2}/n}\mik\frac{2}{n^k};
\end{equation}
since $2s\mik 2d\mik \sqrt{n}$, by \eqref{e7.9} we obtain that
\begin{equation} \label{e7.11}
\ave[\Xi_s^2] \mik 2\,\sum_{r=0}^{s} \frac{1}{n^{2s-r}} \,
\Big|\sum_{(i,j,p,q)\in\mathcal{F}_r^2} \boldsymbol{\xi}_s(i,p)\, \boldsymbol{\xi}_s(j,q) \Big|.
\end{equation}
Finally, plugging in the right-hand-side of \eqref{e7.11} the estimate in \eqref{e5.1}
applied for ``$t=\emptyset$", we conclude that \eqref{e7.7} is satisfied.

We proceed to the proof of inequality \eqref{e7.8}. For every $\ell_1,\ell_2\in [n]$ set
\begin{enumerate}
\item[$\bullet$] $\mathcal{F}_{\ell_1,\ell_2}^3 \coloneqq \big\{ i\in[n]^s_{\mathrm{Inj}}:
\mathrm{Im}(i)\cap\{\ell_1,\ell_2\}\neq\emptyset\big\}$.
\end{enumerate}
We start by observing that
\begin{align}
\label{e7.12} & \ \ \ave\big[(\Xi_s - \Xi_s')^2\big]
 = \ave\Big[ \mathbf{1}_{[I_1\neq I_2]}\, \Big( \sum_{i\in\mathcal{F}^3_{I_1,I_2}}
\boldsymbol{\xi}_s(i,\pi_1\circ i) - \boldsymbol{\xi}_s(i,\pi_2\circ i) \Big)^2\Big] \mik \\
\mik 2 \ave\Big[ & \mathbf{1}_{[I_1\neq I_2]}  \, \Big(\!\!\!
\sum_{\substack{i\in[n]^s_{\mathrm{Inj}}\\\mathrm{Im}(i)\cap\{I_1,I_2\}\neq\emptyset}}
\!\!\!\!\! \boldsymbol{\xi}_s(i,\pi_1\circ i)\Big)^2\Big] + 2 \ave\Big[ \mathbf{1}_{[I_1\neq I_2]} \,
\Big(\!\!\! \sum_{\substack{i\in[n]^s_{\mathrm{Inj}}\\\mathrm{Im}(i)\cap\{I_1,I_2\}\neq\emptyset}}
\!\!\!\!\! \boldsymbol{\xi}_s(i,\pi_2\circ i)  \Big)^2\Big]. \nonumber
\end{align}
By (\hyperref[E1]{$\mathcal{E}$1}), (\hyperref[E2]{$\mathcal{E}$2}) and
(\hyperref[E3]{$\mathcal{E}$3}), we see that the pairs
$\big(\pi_1,(I_1,I_2)\big)$ and $\big(\pi_2,(I_1,I_2)\big)$ have the same distribution.
Therefore, by \eqref{e7.12}, we have
\begin{equation} \label{e7.13}
\ave\big[(\Xi_s - \Xi_s')^2\big] \mik 4 \ave\Big[ \mathbf{1}_{[I_1\neq I_2]}\,
\Big( \sum_{\substack{i\in[n]^s_{\mathrm{Inj}}\\\mathrm{Im}(i)\cap\{I_1,I_2\}\neq\emptyset}}
\!\!\!\! \boldsymbol{\xi}_s(i,\pi_1\circ i)\Big)^2\Big]
\end{equation}
which implies, by (\hyperref[E1]{$\mathcal{E}$1}) and the fact that $(I_1,I_2)$ is uniformly
distributed on $[n]^2$, that
\begin{align}
\label{e7.14} \ave\big[(\Xi_s & - \Xi_s')^2\big] \mik
\frac{4}{n^2} \sum_{\substack{\ell_1,\ell_2\in[n]\\\ell_1\neq\ell_2}}
\ave\Big[ \Big( \sum_{\substack{i\in[n]^s_{\mathrm{Inj}}\\\mathrm{Im}(i)\cap\{\ell_1,\ell_2\}\neq\emptyset}}
\!\!\!\!\!\!\! \boldsymbol{\xi}_s(i,\pi_1\circ i) \Big)^2\Big] \\
& = \frac{4}{n^2} \sum_{\substack{\ell_1,\ell_2\in[n]\\\ell_1\neq\ell_2}}
\ave\Big[ \Big( \sum_{t\in\mathrm{PartInj}([s],\{\ell_1,\ell_2\})}
\sum_{\substack{t\sqsubseteq i\in[n]^s_{\mathrm{Inj}}\\ (\mathrm{Im}(i)\setminus\mathrm{Im}(t))
\cap\{\ell_1,\ell_2\}=\emptyset}} \!\!\!\!\!\!\!\!\!\!\!\!\!\!
\boldsymbol{\xi}_s(i,\pi_1\circ i) \Big)^2\Big] \nonumber \\
& = \frac{4}{n^2} \sum_{\substack{\ell_1,\ell_2\in[n]\\\ell_1\neq\ell_2}}
\ave\Big[\Big( \sum_{t\in\mathrm{PartInj}([s],\{\ell_1,\ell_2\})}
(-1)^{|t|+1} \sum_{t\sqsubseteq i\in[n]^s_{\mathrm{Inj}}}
\boldsymbol{\xi}_s(i,\pi_1\circ i)\Big)^2\Big]. \nonumber
\end{align}
Now observe that
\begin{equation} \label{e7.15}
\big|\mathrm{PartInj}([s],\{\ell_1,\ell_2\})\big|=2s+2\, \binom{s}{2}= s(s+1).
\end{equation}
Hence, by \eqref{e7.14}, \eqref{e7.15} and the Cauchy--Schwarz inequality,
\begin{align}
\label{e7.16} & \ave\big[ (\Xi_s - \Xi_s')^2\big] \mik
\frac{4s(s+1)}{n^2} \sum_{\substack{\ell_1,\ell_2\in[n]\\\ell_1\neq\ell_2}}
\ \sum_{t\in\mathrm{PartInj}([s],\{\ell_1,\ell_2\})} \!\!
\ave\Big[\Big( \sum_{t\sqsubseteq i\in[n]^s_{\mathrm{Inj}}}
\boldsymbol{\xi}_s(i,\pi_1\circ i)\Big)^2\Big] \\
= \ \ & \frac{4s(s+1)}{n^2} \sum_{\substack{\ell_1,\ell_2\in[n]\\\ell_1\neq\ell_2}} \
\sum_{t\in\mathrm{PartInj}([s],\{\ell_1,\ell_2\})} \
\sum_{\substack{i,j\in[n]^s_{\mathrm{Inj}}\\t\sqsubseteq i,j}}
\ave\big[\boldsymbol{\xi}_s(i,\pi_1\circ i)\, \boldsymbol{\xi}_s(j,\pi_1\circ j)\big] \nonumber \\
= \ \ & \frac{4s(s+1)}{n^2}\!\! \sum_{\substack{\ell_1,\ell_2\in[n]\\\ell_1\neq\ell_2}}
\ \sum_{t\in\mathrm{PartInj}([s],\{\ell_1,\ell_2\})} \
\sum_{r=|\mathrm{Im}(t)|}^{s} \!\!\!\! \frac{(n-(2s-r))!}{n!} \!\!\!\!\!\!\!\!
\sum_{\substack{i,j,p,q\in[n]^s_{\mathrm{Inj}}\\t\sqsubseteq i,j\\
|\mathrm{Im}(i) \cap \mathrm{Im}(j)|=r \\ (i,j)\sim(p,q)}}
\!\!\!\!\!\!\! \boldsymbol{\xi}_s(i,p)\, \boldsymbol{\xi}_s(j,q) \nonumber \\
\stackrel{\eqref{e7.10}}{\mik} & \frac{8s(s+1)}{n^2}\!\! \sum_{\substack{\ell_1,\ell_2\in[n]\\\ell_1\neq\ell_2}}
\ \sum_{t\in\mathrm{PartInj}([s],\{\ell_1,\ell_2\})} \
\sum_{r=|\mathrm{Im}(t)|}^{s} \frac{1}{n^{2s-r}}\,
\Big|\sum_{\substack{i,j,p,q\in[n]^s_{\mathrm{Inj}}\\t\sqsubseteq i,j\\
|\mathrm{Im}(i) \cap \mathrm{Im}(j)|=r \\ (i,j)\sim(p,q)}}
\!\!\!\!\!\!\! \boldsymbol{\xi}_s(i,p)\, \boldsymbol{\xi}_s(j,q)\Big|. \nonumber
\end{align}
Setting $K_d\coloneqq 8s(s+1)\,e^{2s}\,\big((2s)!\big)^2$ and applying \eqref{e5.1}
for every $t\in\mathrm{PartInj}([s],\{\ell_1,\ell_2\})$ and every $r\in\big\{|\mathrm{Im}(t)|,\dots,s\big\}$,
by \eqref{e7.16} we have
\begin{equation} \label{e7.17}
\ave\big[ (\Xi_s - \Xi_s')^2\big] \mik \frac{K_d}{n^{s+2}}
\sum_{\substack{\ell_1,\ell_2\in[n]\\\ell_1\neq\ell_2}}
\ \sum_{t\in\mathrm{PartInj}([s],\{\ell_1,\ell_2\})} \
\sum_{r=|\mathrm{Im}(t)|}^{s} \,
\sum_{\substack{i,p\in[n]^s\\t\sqsubseteq i}} \boldsymbol{\xi}_s(i,p)^2
\end{equation}
By double-counting, we see that the right-hand-side of \eqref{e7.17} is equal to
\[ \frac{K_d}{n^{s+2}} \Big( 2(n-1)s(s-1) \sum_{F\in \binom{[s]}{1}} \sum_{x\in [n]^F}
\sum_{\substack{i,p\in[n]^s\\x\sqsubseteq i}} \boldsymbol{\xi}_s(i,p)^2 +
2(s-2) \sum_{H\in \binom{[s]}{2}} \sum_{y\in [n]^H_{\mathrm{Inj}}}
\sum_{\substack{i,p\in[n]^s\\y\sqsubseteq i}} \boldsymbol{\xi}_s(i,p)^2 \Big) \]
which is at most, since $s\mik d\mik n$,
\begin{equation} \label{e7.18}
2\,K_d\, s(s-1)\,\frac{\beta_s}{n^{s+1}}  + K_d\, s(s-1)(s-2)\, \frac{\beta_s}{n^{s+2}}
\mik 3\, K_d\, s(s-1) \frac{\beta_s}{n^{s+1}}.
\end{equation}
By \eqref{e7.17}, \eqref{e7.18} and the choice of $K_d$, we conclude that \eqref{e7.8} is satisfied.
\end{proof}
Now set
\begin{equation} \label{e7.19}
\Theta \coloneqq \sum_{s=2}^{d}\Xi_s \ \ \ \text{ and } \ \ \ \Theta' \coloneqq \sum_{s=2}^{d}\Xi_s'.
\end{equation}
The last ingredient of the proof of Theorem \ref{t2.2} is the following lemma which is taken
from \cite[Lemma 2.2]{BC05}.
\begin{lem}[Barbour/Chen] \label{l7.3}
For every $z\in \rr$ we have
\begin{align} \label{e7.20}
\max\big\{ \mathbb{P}(z- |\Theta|\mik \Xi_1\mik z),&\,\mathbb{P}(z\mik \Xi_1 \mik z+|\Theta|)\big\} \mik \\
& \mik 2^5\frac{\Lambda}{n}+ 2^{17} \frac{\Lambda^2}{n^2} +
5d^2 e^d (2d)! \sum_{s=2}^{d}\sqrt{\frac{\beta_s}{n^s}} \nonumber
\end{align}
where $\Lambda$ is as in \eqref{e7.2} and $\beta_s$ is as in \eqref{e2.6} for every $s\in \{2,\dots,d\}$.
\end{lem}
Since the setting in \cite{BC05} is not identical with ours, for completeness we give the
proof of Lemma \ref{l7.3} in Appendix \ref{appendix}.

We are now in a position to complete the proof of the theorem. Notice that
\eqref{e2.7} is automatically satisfied if $\frac{\Lambda}{n}\meg 1$. Thus, in what follows,
we may assume that $\frac{\Lambda}{n}<1$. By Bolthausen's theorem \cite{Bo84}, we have 		
\begin{equation} \label{e7.21}
\sup_{z\in\rr} \big| \mathbb{P}(\Xi_1\mik z) - \mathbb{P}\big(\mathcal{N}(0,1)\mik z\big) \big|
\mik C_1\frac{\Lambda}{n}
\end{equation}
where $C_1\meg 1$ is as in \eqref{e2.4}. Since $\Theta+|\Theta|\meg 0$, for every $z\in\rr$ we have
\begin{align}
\label{e7.22} & \mathbb{P}(W\mik z) = \mathbb{P}\big(\Xi_1+\Theta+|\Theta|\mik z+|\Theta|\big)
\mik \mathbb{P}\big(\Xi_1\mik z+|\Theta|\big) \\
\label{e7.23} & \mathbb{P}(W\mik z) = \mathbb{P}\big(\Xi_1+\Theta-|\Theta|\mik z-|\Theta|\big)
\meg \mathbb{P}\big(\Xi_1\mik z-|\Theta|\big)
\end{align}
and so,
\begin{align}
\label{e7.24} & \mathbb{P}(W\mik z) - \mathbb{P}(\Xi_1\mik z)
 \mik \mathbb{P}\big(z< \Xi_1\mik z+|\Theta|\big) \\
\label{e7.25} & \mathbb{P}(W\mik z) - \mathbb{P}(\Xi_1\mik z)
 \meg -\mathbb{P}(z-|\Theta|< \Xi_1\mik z).
\end{align}
Hence, by Lemma \ref{l7.3} and the fact that $\frac{\Lambda}{n}<1$, we obtain that
\begin{align}
\label{e7.26}
\sup_{z\in\rr}|\mathbb{P}(W\mik z) & - \mathbb{P}(\Xi_1\mik z)|  \\
& \mik \sup_{z\in\rr}\max\big\{ \mathbb{P}\big(z< \Xi_1\mik z+|\Theta|\big),
\mathbb{P}(z-|\Theta|< \Xi_1\mik z)\big\} \nonumber \\
& \mik (2^5+2^{17})\frac{\Lambda}{n} + 5d^2 e^d (2d)! \sum_{s=2}^{d}\sqrt{\frac{\beta_s}{n^s}}. \nonumber
\end{align}
Inequality \eqref{e2.7} follows from \eqref{e7.21}, \eqref{e7.26}, the triangle inequality
and the choice of the constant $C_d$ in Theorem \ref{t2.2}. The proof is completed.
\begin{rem}[Towards a group-theoretic combinatorial CLT] \label{r7.4}
Let $I$ be a nonempty finite set, and let $G$ be a subgroup of $\mathbb{S}_I$ which acts
transitively on $I$. With every matrix $\boldsymbol{\zeta}\colon I\times I\to \rr$ we
associate the statistic
\begin{equation} \label{e7.27}
Z_G = \sum_{i\in I} \boldsymbol{\zeta}\big(i,\pi(i)\big)
\end{equation}
where $\pi$ is a random permutation which is uniformly distributed on $G$. The case $G=\mathbb{S}_I$
is, of course, the setting of the classical combinatorial CLT; on the other hand, it is easy to see
that the statistics covered by the work of Barbour/Chen \cite{BC05} and Theorem~\ref{t2.2} are
also included\footnote{More precisely, in these cases the index set $I$ can be identified with
the set $\binom{[n]}{d}$ and the group $G$ consists of those permutations $\overline{\pi}$ of
$\binom{[n]}{d}$ which are of the form $\overline{\pi}\big(\{j_1,\dots,j_d\}\big)=
\{\pi(j_1),\dots,\pi(j_d)\}$ for some permutation $\pi$ of $[n]$.} in this general
framework. It is thus natural to ask\footnote{In fact, it is quite likely that this
question has already been asked, but we could not find something relevant in the
literature.} for which finite transitive permutation groups $G$ we have quantitative
normal approximation of the statistics $Z_G$ under suitable conditions on the
matrix $\boldsymbol{\zeta}$. The case of alternating groups is already interesting.
\end{rem}
\begin{rem}[Finite population statistics] \label{r7.8}
This remark has been kindly communicated to us by the referee. Let $n\meg d\meg 2$ be positive integers,
and let $t\colon [n]^d\to\rr$ be a (deterministic) symmetric real tensor. With $t$ we associate the statistic
\begin{equation} \label{e7.28}
T \coloneqq t\big(\pi(1),\dots,\pi(d)\big)
\end{equation}
where $\pi$ is a random permutation which is uniformly distributed on $\mathbb{S}_n$. These statistics
have been studied by several authors; see, \textit{e.g.}, \cite{BG93,BoG02,ZC90} and the references therein.
It was shown by Bloznelis and G\"{o}tze \cite{BoG01} that $T-\ave[T]$ can be written~as
\begin{equation} \label{e7.29}
T-\ave[T]= \sum_{s=1}^d \ \, \sum_{1\mik i_1<\dots <i_s\mik d} \!\!g_s\big(\pi(i_1),\dots,\pi(i_d)\big)
\end{equation}
where, for every $s\in [d]$, $g_s\colon [n]^s\to\rr$ is a symmetric, Hoeffding\footnote{That is,
$\sum_{i_0\sqsubseteq i\in [n]^s} g_s(i)=0$ for every $i_0\in [n]^{s-1}$.} tensor whose diagonal
terms vanish. Using this representation, it is easy to see that the random variable $T-\ave[T]$ can also
be expressed as a $W$-statistic of order $d$. More precisely, for every $s\in \{2,\dots,d\}$ define
$\boldsymbol{\xi}_s\colon [n]^s\times [n]^s\to \rr$ by setting
$\boldsymbol{\xi}_s(i,j)=\alpha^s_{|\mathrm{Im}(i)|} \cdot \mathbf{1}_{[d]^s}(i) \cdot g_s(j)$ where
$\alpha^s_s\coloneqq 1/s!$ and $\alpha^s_l\coloneqq (1/s!) \prod_{k=1}^{s-l} \big( 1-\frac{d}{s-k}\big)$
if $l\in \{1,\dots, s-1\}$; moreover, define $\boldsymbol{\xi}_1\colon [n]\times [n]\to\rr$ by
$\boldsymbol{\xi}_1(i,j)=\big( \mathbf{1}_{[d]}(i)-\frac{d}{n}\big)\, g_1(j)$. Then observe that the
tensors $\boldsymbol{\xi}_1,\dots,\boldsymbol{\xi}_d$ are Hoeffding in the sense of Definition \ref{d2.1},
and the $W$-statistic associated with them via~\eqref{e2.3} coincides with $T-\ave[T]$.
In particular, if $n\meg 4d^2$, then, by Theorem \ref{t2.2}, we have\footnote{Here, for every
tensor $g\colon [n]^s\to\rr$ and every $p>1$ we set $\|g\|_{\ell_p}\coloneqq \big( \sum_{i\in [n]^s}
|g(i)|^p\big)^{1/p}$.}
\begin{equation} \label{e7.30}
d_K\big( T,\mathcal{N}(\mu,\sigma^2)\big) \mik
\frac{2^{19}C_1}{\sqrt{dn}} \, \Big(\frac{\|g_1\|_{\ell_3}}{\|g_1\|_{\ell_2}}\Big)^3 +
2 C_d \sum_{s=2}^d \frac{1}{s!}\, \Big(\frac{d}{\sqrt{n}}\Big)^{s-1} \,
\frac{\|g_s\|_{\ell_2}}{\|g_1\|_{\ell_2}}
\end{equation}
where $\mu\coloneqq \ave[T]$ and $\sigma^2\coloneqq \frac{d}{n-1}\big(1-\frac{d}{n}\big) \|g_1\|_{\ell_2}^2$.
(Recall that $C_1\meg 1$ is as in \eqref{e2.4}, and $C_d=5d^2 e^d(2d)!$.)
\end{rem}


\part{Proofs of Theorem \ref*{t1.4} and its applications} \label{part3}

\section{Computation of the variance: proof of Proposition \ref*{p1.2}} \label{sec8}

\numberwithin{equation}{section}

This section is devoted to the proof of Proposition \ref{p1.2}. We start by introducing some
auxiliary quantities which will also be used in Section \ref{sec9}.

\subsection{The parameters $\gamma_{s,r}$} \label{subsec8.1}

For every $s,r\in \{0,1,\dots\}$ with $s\mik r$,
\begin{enumerate}
\item[$\bullet$] let $\mathcal{X}_{s,r}$ denote the set of all strictly increasing maps
$u\colon [\ell] \to \mathbb{Z}$ for some $\ell\in\nn$ which satisfy $u(1)=s$ and $u(\ell)= r$
(thus, $\ell\mik r-s+1$),
\end{enumerate}
and set
\begin{align}
\label{e8.1} \gamma_{s,r} & \coloneqq \sum_{u\in\mathcal{X}_{s,r}}
(-1)^{|\mathrm{dom}(u)|+1} \prod_{i=2}^{|\mathrm{dom}(u)|} \binom{u(i)}{u(i-1)} \\
& = \sum_{u\in\mathcal{X}_{s,r}} (-1)^{|\mathrm{dom}(u)|+1}
\binom{r}{s,u(2)-u(1),\dots, u\big(|\mathrm{dom}(u)|\big)-u\big(|\mathrm{dom}(u)|-1\big)} \nonumber
\end{align}
with the convention that the above product is equal to $1$ if $i>|\mathrm{dom}(u)|$.
Observe that
\begin{equation} \label{e8.2}
\gamma_{s,s}=1.
\end{equation}
Also notice that for every $s,r\in \{0,1,\dots\}$ with $s\mik r$ we have
\begin{equation} \label{e8.3}
\mathcal{X}_{s-1,r} = \bigcup_{x=s}^r\{f_{s-1, u}:u\in \mathcal{X}_{x,r}\}
\end{equation}
where the union above is a disjoint union; here, $f_{s-1, u}\colon \big[|\mathrm{dom}(u)|+1\big]\to \mathbb{Z}$
denotes the function defined by $f_{s-1, u}(1)=s-1$ and $f_{s-1, u}(i) = u(i-1)$
if $i\in \{2,\dots,|\mathrm{dom}(u)|+1\}$. By \eqref{e8.1}, it follows in particular that
for every $s,r\in \{1,2,\dots\}$ with $s\mik r$ we have
\begin{equation} \label{e8.4}
\gamma_{s-1,r} = - \sum_{x=s}^{r}\binom{x}{s-1}\, \gamma_{x,r}.
\end{equation}

\subsubsection{\!} \label{subsubsec8.1.1}

We proceed to compute these quantities. More precisely, we claim
that for every $s,r\in\{0,1,\dots\}$ with $s\mik r$ we have
\begin{equation} \label{e8.5}
\gamma_{s,r} = (-1)^{r-s} \binom{r}{s}.
\end{equation}
The case ``$s=r$" follows, of course,  from \eqref{e8.2}. So, let $s,r\in \{0,1,,\dots\}$ with $s<r$.
Notice that for every $u\in \mathcal{X}_{s,r}$ we have $2 \mik|\mathrm{dom}(u)|\mik r-s+1$.
For every $\ell\in [r-s]$ set
\begin{enumerate}
\item[$\bullet$] $\mathcal{X}_{s,r}^\ell \coloneqq \big\{u\in\mathcal{X}_{s,r}:|\mathrm{dom}(u)|=\ell+1\big\}$
and  $\mathcal{Y}_\ell\coloneqq \big\{y\in[r-s]^\ell: \sum_{i=1}^\ell y(i) = r-s\big\}$.
\end{enumerate}
Then we have
\begin{align}
\label{e8.6} & \gamma_{s,r} = \binom{r}{s} \sum_{\ell = 1}^{r-s}\, (-1)^{\ell+2}
\sum_{u\in\mathcal{X}_{s,r}^\ell} \binom{r-s}{u(2)-u(1),\dots,u(\ell+1)-u(\ell)} \\
 = & \binom{r}{s} \sum_{\ell = 1}^{r-s}\, (-1)^{\ell} \sum_{y\in\mathcal{Y}_\ell}
\binom{r-s}{y(1),\dots,y(\ell)} = \binom{r}{s} \sum_{\ell = 1}^{r-s}\, (-1)^{\ell}\,
\big|\big\{f\in[\ell]^{r-s}:f\text{ is onto}\big\}\big| \nonumber \\
= & \binom{r}{s} \sum_{\ell = 1}^{r-s}\, (-1)^{\ell}
\sum_{j=0}^\ell\, (-1)^j\, \binom{\ell}{j}\, (\ell-j)^{r-s}
= \binom{r}{s} \sum_{0\mik j \mik \ell \mik r-s} \,
(-1)^{\ell-j}\, \binom{\ell}{j}\, (\ell-j)^{r-s} \nonumber
\end{align}
which yields, by setting ``$t=\ell-j$", that
\begin{equation} \label{e8.7}
\gamma_{s,r} = \binom{r}{s} \sum_{t=0}^{r-s} (-1)^t\, t^{r-s} \sum_{j=0}^{r-s-t} \binom{t+j}{j}.
\end{equation}
Noticing that $\sum_{j=0}^{r-s-t} \binom{t+j}{j}= \binom{r-s+1}{r-s-t}=\binom{r-s+1}{t+1}$
for every positive integer $t$, by \eqref{e8.7} and an additional change of variable ``$q=t+1$",
we obtain that
\begin{equation} \label{e8.8}
\gamma_{s,r} = \binom{r}{s} \sum_{t=0}^{r-s} (-1)^t t^{r-s}\,\binom{r-s+1}{t+1}
= -\binom{r}{s} \sum_{q=1}^{r-s+1} (-1)^q (q-1)^{r-s} \binom{r-s+1}{q}.
\end{equation}
By a classical formula which goes back to Euler---see, \textit{e.g.}, \cite[Eq. (5.42)]{GKP94}---and the
fact that the degree of the polynomial $(x-1)^{r-s}$ is strictly smaller than $r-s+1$, we also obtain that
$\sum_{q=0}^{r-s+1} (-1)^q (q-1)^{r-s} \binom{r-s+1}{q} = 0$. Therefore, by \eqref{e8.8}, we conclude that
\eqref{e8.5} is satisfied.

\subsection{Proof of Proposition \ref{p1.2}} \label{subsec8.2}

In what follows, let $n,d$ be positive integers with $n\meg 2d$,
and fix $\xtensor$ and $\thtensor$ which satisfy (\hyperref[A1]{$\mathcal{A}$1}),
(\hyperref[A2]{$\mathcal{A}$2}) and (\hyperref[A3]{$\mathcal{A}$3}).
We define $\boldsymbol{a}=\langle a_F: F\in \binom{[n]}{d}\rangle$ by setting
$a_F = \theta_{(i_1,\dots,i_d)}$ for every $F=\{i_1<\dots<i_d\}\in \binom{[n]}{d}$.
By (\hyperref[A2]{$\mathcal{A}$2}), we see that $a_F=\theta_i$ for every $F\in \binom{[n]}{d}$
and every $i\in [n]^d$ with $\mathrm{Im}(i)=F$. Moreover, by (\hyperref[A1]{$\mathcal{A}$1}),
we have $\ave\big[ \langle \bbth,\bbx\rangle\big]=0$. Therefore,
\begin{equation} \label{e8.9}
\mathrm{Var}\big(\langle \bbth,\bbx\rangle\big) \!= \!\!\!\!
\sum_{i,j\in[n]^d} \!\! \theta_i\, \theta_j\,\ave[X_i X_j] \stackrel{(\hyperref[A2]{\mathcal{A}2})}{=}
\sum_{s=0}^{d}\delta_s \!\!\!\!\!\!\!\! \sum_{\substack{i,j\in[n]^d\\|\mathrm{Im}(i)\cap\mathrm{Im}(j)|=s}}
\!\!\!\!\!\!\!\!\!\! \theta_i\, \theta_j
=(d!)^2 \sum_{s=0}^{d}\delta_s \!\!\!\! \sum_{\substack{F,G\in \binom{[n]}{d}\\|F\cap G|=s}}\!\!\!\! a_F a_G.
\end{equation}
For every $s\in \{0,1,\dots,d\}$ set
\begin{equation} \label{e8.10}
\kappa_s \coloneqq \!\! \sum_{\substack{F,G\in \binom{[n]}{d}\\|F\cap G|=s}}\! a_F\, a_G
\ \ \text{ and } \ \
\lambda_s \coloneqq \sum_{H\in \binom{[n]}{s}}
\Big(\sum_{F\in \binom{[n]\setminus H}{d-s}} \! a_{H\cup F}\Big)^2.
\end{equation}
Notice that $\kappa_d = \lambda_d$; moreover, for every $s\in \{0,\dots,d-1\}$ we have
\begin{align} \label{e8.11}
\kappa_s & = \sum_{H\in \binom{[n]}{s}}
\sum_{\substack{K,L\in \binom{[n]\setminus H}{d-s}\\K\cap L = \emptyset}} a_{H\cup K}\, a_{H\cup L} =
\lambda_s-\sum_{t=1}^{d-s} \sum_{H\in \binom{[n]}{s}}
\sum_{\substack{K,L\in \binom{[n]\setminus H}{d-s}\\|K\cap L| = t}} a_{H\cup K}\, a_{H\cup L}\\
& = \lambda_s - \sum_{t=1}^{d-s} \binom{s+t}{s} \kappa_{s+t} =
\lambda_s - \sum_{r=s+1}^d \binom{r}{s} \kappa_r. \nonumber
\end{align}
Proceeding by backwards induction and using the fact that $\kappa_d = \lambda_d$,
\eqref{e8.11}, \eqref{e8.2} and \eqref{e8.4}, we see that for every $s\in \{0,1,\dots,d\}$,
\begin{equation} \label{e8.12}
\kappa_s=\sum_{r=s}^d \gamma_{s,r}\, \lambda_r.
\end{equation}
By \eqref{e8.10}, the description of the variance of $\langle \bbth,\bbx\rangle$ in \eqref{e8.9}
and \eqref{e8.12}, we obtain that
\begin{equation} \label{e8.13}
\mathrm{Var}\big(\langle \bbth,\bbx\rangle\big) =
(d!)^2\sum_{s=0}^{d}\delta_s \sum_{r=s}^d \gamma_{s,r}\, \lambda_r =
(d!)^2\sum_{r=0}^{d}\lambda_r \sum_{s=0}^r \gamma_{s,r}\, \delta_s
\end{equation}
which in turn implies, by \eqref{e8.5}, that
\begin{equation} \label{e8.14}
\mathrm{Var}\big( \langle \bbth,\bbx\rangle\big) =
(d!)^2\sum_{r=0}^{d}\,\lambda_r \Big(\sum_{s=0}^r (-1)^{r-s}\, \binom{r}{s}\, \delta_s\Big).
\end{equation}
Equality \eqref{e1.8} follows from \eqref{e8.14} after observing that
$\seminorm{\boldsymbol{\theta}}^2_s = s!\,\big((d-s)!\big)^2 \lambda_s$
for every $s\in \{0,\dots,d\}$. The proof of Proposition \ref{p1.2} is completed.


\section{From $Z$-statistics to $W$-statistics} \label{sec9}

\numberwithin{equation}{section}

We have pointed out in Section \ref{sec2} that every statistic of the form \eqref{e2.1}
can be written in the form \eqref{e2.3}. In the one-dimensional case, this transformation is
classical; see \cite{Ho51}. The two-dimensional case is also well-known;
see, \textit{e.g.},~\cite{ZBCL97}. These arguments can be extended to the higher-dimensional case
though, to the best of our knowledge, this has not appeared in the literature.
Since this step is needed for the proof of Theorem~\ref{t1.4} and, more important, it affects
its quantitative aspects\footnote{Specifically, it affects the constant $\kappa$.},
in this section we shall describe this transformation for those tensor permutation
statistics which are relevant to Theorem~\ref{t1.4}. That said, we note that
some of the combinatorial constructions that appear in this section are related
to the classical work of de Jong \cite{deJ90} as well as the more recent work of
D\"{o}bler, Kasprzak and Peccati \cite{DKP22}.

We also point out that a slightly different type of decomposition of permutation statistics---which
is more akin to the standard Hoeffding decomposition \cite{Ho48}---has been considered by various authors,
including Zhao/Chen \cite{ZC90} and Bloznelis/G\"{o}tze \cite{BoG01}, and culminating in the work of
Peccati \cite{P04} (see, also, \cite{PP11}).

\subsection{Summing, averaging and Hoeffding operators} \label{subsec9.1}

Let $n\meg s$ be positive integers, and let $\boldsymbol{a}\colon [n]^s\to\rr$
be a real tensor. For every (possibly empty) subset $F$~of~$[s]$ we introduce the tensors
$\mathcal{S}_F[\boldsymbol{a}]\colon [n]^F\to\rr$ and
$\mathcal{A}_F[\boldsymbol{a}]\colon [n]^F\to \rr$ by setting for every~$i\in [n]^F$,
\begin{equation} \label{e9.1}
\mathcal{S}_F[\boldsymbol{a}](i) \coloneqq \sum_{i\sqsubseteq j\in[n]^s} \!\!\boldsymbol{a}(j)
\ \ \text{ and } \ \ \mathcal{A}_F[\boldsymbol{a}](i) \coloneqq
\frac{1}{n^{s-|F|}}\sum_{i\sqsubseteq j\in[n]^s} \!\!\boldsymbol{a}(j).
 \end{equation}
Note that if ``$F=\emptyset$", then $[n]^{\emptyset} = \{\emptyset\}$ and, consequently,
$\mathcal{S}_\emptyset[\boldsymbol{a}]$ and $\mathcal{A}_\emptyset[\boldsymbol{a}]$ are
tensors with a single entry which is the sum and the average of the entries of $\boldsymbol{a}$
respectively. Also observe that
$\mathcal{S}_{[s]}[\boldsymbol{a}]= \mathcal{A}_{[s]}[\boldsymbol{a}] =\boldsymbol{a}$.

Moreover, we define the \textit{Hoeffding tensor} $\mathcal{H}[\boldsymbol{a}]\colon [n]^s\to\rr$
\textit{associated with} $\boldsymbol{a}$ by setting for every~$i\in [n]^s$,
\begin{equation} \label{e9.2}
\mathcal{H}[\boldsymbol{a}](i) \coloneqq \sum_{F\subseteq[s]}(-1)^{s- |F|}\,
\mathcal{A}_F[\boldsymbol{a}](i\upharpoonright F).
\end{equation}

\subsubsection{Doubly-indexed tensors} \label{subsubsec9.1.1}

We will also work with averaging and Hoeffding operators associated with doubly-indexed tensors.
More precisely, let $n\meg s$ be positive integers, and let
$\boldsymbol{\zeta}\colon [n]^s\times [n]^s\to\rr$. For every pair $F,G$ of (possibly empty) subsets
of $[s]$ we define the tensor $\mathcal{A}_{F,G}[\boldsymbol{\zeta}]\colon [n]^F\times[n]^G\to\rr$
by setting for every $i\in [n]^F$ and every $p\in [n]^G$,
\begin{equation} \label{e9.3}
\mathcal{A}_{F,G}[\boldsymbol{\zeta}](i,p) \coloneqq \frac{1}{n^{2s-|F|-|G|}}
\sum_{\substack{i\sqsubseteq j\in[n]^s\\p\sqsubseteq q\in [n]^s}}\!\!\boldsymbol{\zeta}(j,q).
\end{equation}
Again, observe that $\mathcal{A}_{[s],[s]}[\boldsymbol{\zeta}] = \boldsymbol{\zeta}$; moreover,
the tensor $\mathcal{A}_{\emptyset,\emptyset}(\boldsymbol{\zeta})$ contains a single entry which
is equal to the average of the entries of $\boldsymbol{\zeta}$.

Finally, we define the \textit{Hoeffding tensor}
$\mathcal{H}[\boldsymbol{\zeta}]\colon [n]^s\times [n]^s\to\rr$ \textit{associated with} $\boldsymbol{\zeta}$
by setting for every $i,p\in [n]^s$,
\begin{equation} \label{e9.4}
\mathcal{H}[\boldsymbol{\zeta}](i,p) \coloneqq \sum_{F,G\subseteq[s]} (-1)^{|F|+|G|}\,
\mathcal{A}_{F,G}[\boldsymbol{\zeta}](i\upharpoonright F,p\upharpoonright G).
\end{equation}
We have the following elementary---though important---fact.
\begin{fact} \label{f9.1}
For every $\boldsymbol{\zeta}\colon [n]^s\times [n]^s\to\rr$, the tensor
$\mathcal{H}[\boldsymbol{\zeta}]$ associated with $\boldsymbol{\zeta}$
is Hoeffding in the sense of Definition \emph{\ref{d2.1}}.
\end{fact}
\begin{rem} \label{r9.1}
It is clear that every tensor $\boldsymbol{\zeta}\colon [n]^s\times[n]^s\to\rr$ can be
naturally identified with a tensor $\boldsymbol{a}_{\boldsymbol{\zeta}}\colon [n]^{2s}\to\rr$.
Then observe that the Hoeffding tensor $\mathcal{H}[\boldsymbol{\zeta}]$ associated with
$\boldsymbol{\zeta}$ via \eqref{e9.4} coincides with the Hoeffding tensor
$\mathcal{H}[\boldsymbol{a}_{\boldsymbol{\zeta}}]$ associated
with  $\boldsymbol{a}_{\boldsymbol{\zeta}}$ via~\eqref{e9.2}.
\end{rem}

\subsection{Decomposition} \label{subsec9.2}

Let $n,r$ be integers with $n\meg r\meg 2$,
and let $\boldsymbol{\zeta}\colon [n]^r\times[n]^r\to\rr$ be a real tensor.
In what follows, we make the following assumptions on $\boldsymbol{\zeta}$.
\begin{enumerate}
\item[($\mathcal{A}$5)] \label{A5} (Symmetry) We have\footnote{Recall that
for every $i=(i_1,\dots,i_r)\in [n]^r$ and every $\tau\in\mathbb{S}_r$ we set
$i \circ \tau=(i_{\tau(1)},\dots,i_{\tau(r)})\in [n]^r$.}
$\boldsymbol{\zeta}(i,j) = \boldsymbol{\zeta}(i \circ \tau, j \circ \rho)$
for every $i,j\in [n]^r$ and $\tau,\rho\in \mathbb{S}_r$.
\item [($\mathcal{A}$6)] \label{A6} (Vanishing diagonal) We have $\boldsymbol{\zeta}(i,j)=0$
if either $i\notin [n]^r_{\mathrm{Inj}}$ or $j\notin [n]^r_{\mathrm{Inj}}$.
\end{enumerate}
As we shall shortly see, this class includes---among others---all those tensors
which appear in the proof of Theorem \ref{t1.4}.

Next, for every $s\in \{0,1,\dots,r\}$ we define $\boldsymbol{\zeta}_s\colon [n]^s\times [n]^s\to\rr$ by
\begin{equation} \label{e9.5}
\boldsymbol{\zeta}_s \coloneqq \mathcal{A}_{[s],[s]}[\boldsymbol{\zeta}]
\end{equation}
where $\mathcal{A}_{[s],[s]}$ is the averaging operator introduced in \eqref{e9.3}; in particular,
$\boldsymbol{\zeta}_0=\mathcal{A}_{\emptyset,\emptyset}(\boldsymbol{\zeta})$ is the average
of the entries of $\boldsymbol{\zeta}$. Moreover, for every $s\in [r]$ set
\begin{equation} \label{e9.6}
R_s\coloneqq \sum_{i\in [n]^s_{\mathrm{Inj}}} \mathcal{H}[\boldsymbol{\zeta}_s](i,\pi \circ i)
\end{equation}
where $\pi$ is a random permutation uniformly distributed on $\mathbb{S}_n$; here
$\mathcal{H}[\boldsymbol{\zeta}_s]$ denotes the Hoeffding tensor associated with $\boldsymbol{\zeta}_s$
via \eqref{e9.4}.

\subsubsection{The weights $w_{s,r}$} \label{subsubsec9.2.1}

For every pair of integers $h,s$ with $0\mik h< s\mik r$ set
\begin{equation} \label{e9.7}
\mu_{h,s} \coloneqq \binom{s}{h} \sum_{f=0}^{s-h} 2^{s-h-f}\,(-1)^f \binom{s-h}{f}\,
\frac{(n-s+f)!}{(n-s)!\, n^f}.
\end{equation}
Moreover, for every $s\in \{0,\dots,r-1\}$,
\begin{enumerate}
\item[$\bullet$] let $\mathcal{X}_{s,r}$ be as in Subsection \ref{subsec8.1},
namely, $\mathcal{X}_{s,r}$ is the set of all strictly increasing maps
$u\colon [\ell]\to\mathbb{Z}$ for some $\ell\in [\nn]$ such that $u(1)=s$ and $u(\ell)=r$.
\end{enumerate}
Given $u\in \mathcal{X}_{s,r}$ set
\begin{equation} \label{e9.8}
\mu_u \coloneqq \! \prod_{i=1}^{|\mathrm{dom}(u)|-1} \!\!\! \mu_{u(i),u(i+1)}
\end{equation}
and for every $s\in\{0,\dots,r-1\}$ define
\begin{equation} \label{e9.9}
w_{s,r} \coloneqq (-1)^{r-s+1}\sum_{u\in \mathcal{X}_{s,r}} (-1)^{|\mathrm{dom}(u)|}\,\mu_u.
\end{equation}
Also set
\begin{equation} \label{e9.10}
w_{r,r} \coloneqq 1.
\end{equation}
We have the following estimate for these quantities.
\begin{lem} \label{l9.3}
If $n\meg 6^r r^2$, then for every $s\in \{0,1,\dots,r\}$,
\begin{equation} \label{e9.11}
\Big| w_{s,r} - \binom{r}{s} \Big| \mik \frac{r^3 18^r r!}{n}.
\end{equation}
\end{lem}
We postpone the proof of Lemma \ref{l9.3} to the end of this section.

\subsubsection{Main result} \label{subsubsec9.2.2}

We have the following proposition.
\begin{prop} \label{p9.4}
Let  $\boldsymbol{\zeta}\colon [n]^r\times [n]^r\to \rr$ be a tensor which satisfies
\emph{(\hyperref[A5]{$\mathcal{A}$5})} and \emph{(\hyperref[A6]{$\mathcal{A}$6})}, and let
$Z$ be the statistic associated with $\boldsymbol{\zeta}$ via \eqref{e2.1}. Then we have
\begin{equation} \label{e9.12}
Z = \sum_{s=1}^{r} n^{r-s} w_{s,r}\, R_s + n^r w_{0,r}\,\boldsymbol{\zeta}_0
\end{equation}
where $R_s$ and $w_{s,r}$ are as in \eqref{e9.6} and \eqref{e9.9} respectively,
and $\boldsymbol{\zeta}_0$ is the average of the entries of $\boldsymbol{\zeta}$.
\end{prop}
\begin{proof}
For every $s\in \{0,\dots, r\}$ let $Z_s$ denote the $Z$-statistic associated with
$\boldsymbol{\zeta}_s$ via \eqref{e2.1}. (In particular, $Z_0$ is constant,
and takes the value $\boldsymbol{\zeta}_0$.) Clearly, we have $Z_r=Z$.
Next, for every $s,p\in \{0,\dots, r\}$ with $s<p$,
\begin{enumerate}
\item[$\bullet$] let $\mathcal{X}_{s,p,r}\coloneqq \big\{u\in\mathcal{X}_{s,r}:
u\big(|\mathrm{dom}(u)|-1\big)\meg p\big\}$ where $\mathcal{X}_{s,r}$ is as in Subsection \ref{subsec8.1},
\end{enumerate}
and set
\begin{equation} \label{e9.13}
D_{s,p} \coloneqq \sum_{u\in \mathcal{X}_{s,p,r}} (-1)^{|\mathrm{dom}(u)|}\,\mu_u.
\end{equation}
Since $\mathcal{X}_{s,r}=\mathcal{X}_{s,s+1,r}$ for every $s\in \{0,\dots,r-1\}$, by \eqref{e9.9} we have
\begin{equation} \label{e9.14}
w_{s,r}= (-1)^{r-s+1}\, D_{s,s+1}.
\end{equation}
Moreover, using a decomposition as in \eqref{e8.3}, we see that for every $s,p\in \{0,\dots, r-1\}$ with $s<p$,
\begin{equation} \label{e9.15}
D_{s,p} = D_{s,p+1} - D_{p,p+1}\, \mu_{s,p}.
\end{equation}

Next, we claim that for every $s\in [r]$ we have
\begin{equation} \label{e9.16}
Z_s = R_s + \sum_{h=0}^{s-1} \,\mu_{h,s}\, n^{s-h}(-1)^{s-h-1}Z_h.
\end{equation}
Indeed, fix $s\in [r]$. Then, upon conditioning on the value of $\pi$,
\begin{align} \label{e9.17}
Z_s(\pi) & \stackrel{\eqref{e2.1},(\hyperref[A6]{\mathcal{A}6})}{=} \sum_{i\in [n]^s_{\mathrm{Inj}}}
\boldsymbol{\zeta}_s(i,\pi\circ i) \stackrel{\eqref{e9.6}}{=} R_s(\pi) + \sum_{i\in [n]^s_{\mathrm{Inj}}}
\big(\boldsymbol{\zeta}_s(i,\pi\circ i)- \mathcal{H}[\boldsymbol{\zeta}_s](i,\pi\circ i)\big) \\
& \hspace{0.35cm} \stackrel{\eqref{e9.4}}{=} R_s(\pi) \, - \!\!\!\! \sum_{\substack{F,G\subseteq [s]\\ |F|+|G|< 2s}}
(-1)^{|F|+|G|}  \sum_{i\in[n]^s_{\mathrm{Inj}}}
\mathcal{A}_{F,G}[\boldsymbol{\zeta}_s](i\upharpoonright F, \pi\circ i\upharpoonright G) \nonumber \\
& \hspace{0.5cm} = \, R_s(\pi) \, - \!\!\!\! \sum_{\substack{F,G\subseteq[s]\\ |F|+|G|<2s}} (-1)^{|F|+|G|} \,
n^{|(F\cup G)^{\complement}|}\, \frac{\big(n-|F\cup G|\big)!}{(n-s)!\, n^{|(F\cup G)^{\complement}|}}
\ \times \nonumber \\
& \hspace{3.5cm} \times \, \sum_{i\in [n]^{F\cup G}_{\mathrm{Inj}}}
\mathcal{A}_{F,G}[\boldsymbol{\zeta}_s](i\upharpoonright F, \pi\circ i\upharpoonright G) \nonumber
\end{align}
where the complements above are taken with respect to $[s]$. Therefore, after observing that
$|(F\cup G)^{\complement}|+|F\setminus G|+|G\setminus F|=|(F\cap G)^{\complement}|$, we obtain that
\begin{align} \label{e9.18}
Z_s(\pi) & \stackrel{(\hyperref[A6]{\mathcal{A}6}),\eqref{e9.3}}{=} R_s(\pi) -
\sum_{\substack{F,G\subseteq [s]\\ |F|+|G|< 2s}}
(-1)^{|F|+|G|}\, n^{|(F\cap G)^{\complement}|} \,
\frac{\big(n-|F\cup G|\big)!}{(n-s)!\, n^{|(F\cup G)^{\complement}|}}
\ \times \\
& \hspace{4cm} \times \, \sum_{i\in [n]^{F\cap G}_{\mathrm{Inj}}}
\mathcal{A}_{F\cap G,F\cap G}[\boldsymbol{\zeta}_s](i, \pi\circ i) \nonumber \\
& \hspace{0.56cm} = R_s(\pi) - \sum_{H \varsubsetneq [s]} n^{s-|H|} \sum_{i\in [n]^H_{\mathrm{Inj}}}
\mathcal{A}_{H,H}[\boldsymbol{\zeta}_s](i, \pi\circ i) \ \times \nonumber \\
& \hspace{4cm} \times \sum_{\substack{F,G\subseteq [s]\\ F\cup G=H}} (-1)^{|F\cup G|+|F\cap G|} \,
\frac{\big(n-|F\cup G|\big)!}{(n-s)!\, n^{|(F\cup G)^{\complement}|}}. \nonumber
\end{align}
Fix a proper subset $H$ of $[s]$ and set $h=|H|$. Then we have
\begin{align} \label{e9.19}
& \sum_{\substack{F,G\subseteq [s]\\ F\cup G=H}} (-1)^{|F\cup G|+|F\cap G|} \,
\frac{\big(n-|F\cup G|\big)!}{(n-s)!\, n^{|(F\cup G)^{\complement}|}} \\
= & \sum_{K\subseteq [s]\setminus H} 2^{|K|} (-1)^{|K|}
\frac{\big(n-h-|K|\big)!}{(n-s)!\, n^{s-h-|K|}}
= \sum_{k=0}^{s-h} \binom{s-h}{k} 2^k (-1)^k
\frac{\big(n-h-k\big)!}{(n-s)!\, n^{s-h-k}} \nonumber \\
= & \sum_{f=0}^{s-h} \binom{s-h}{f} 2^{s-h-f} (-1)^{s-h-f}
\frac{\big(n-s+f\big)!}{(n-s)!\, n^f} \nonumber
\end{align}
where the last equality follows from the change of variable ``$f=s-h-k$". Thus,
by \eqref{e9.18} and \eqref{e9.19}, we conclude that
\begin{align} \label{e9.20}
Z_s(\pi) & = R_s(\pi) - \sum_{h=0}^{s-1}n^{s-h}\sum_{H \in \binom{[s]}{h}} \sum_{i\in [n]^H_{\mathrm{Inj}}}
\mathcal{A}_{H,H}[\boldsymbol{\zeta}_s](i, \pi\circ i) \ \times \\
 & \hspace{4cm} \times \, \sum_{f=0}^{s-h} \binom{s-h}{f}\, 2^{s-h-f} (-1)^{s-h-f}
\frac{\big(n-s+f\big)!}{(n-s)!\, n^f} \nonumber.
\end{align}
By the definition of $\boldsymbol{\zeta}_s$ in \eqref{e9.5} and (\hyperref[A5]{$\mathcal{A}$5}),
for every $H\varsubsetneq [s]$, setting $h\coloneqq |H|$, we have
$\sum_{i\in [n]^H_{\mathrm{Inj}}}\mathcal{A}_{H,H}[\boldsymbol{\zeta}_s](i, \pi\circ i) =
\sum_{i\in [n]^h_{\mathrm{Inj}}}\boldsymbol{\zeta}_h(i, \pi\circ i) = Z_h(\pi)$;
hence, \eqref{e9.20} yields that
\begin{align} \label{e9.21}
Z_s(\pi) & =  R_s(\pi) - \sum_{h=0}^{s-1}n^{s-h} \binom{s}{h}  Z_h(\pi) \ \times \\
& \hspace{3cm} \times \, \sum_{f=0}^{s-h} \binom{s-h}{f}\, 2^{s-h-f} (-1)^{s-h-f}
\frac{\big(n-s+f\big)!}{(n-s)!\, n^f}.\nonumber
\end{align}
Thus, \eqref{e9.16} follows from \eqref{e9.21} and \eqref{e9.7}.

Next, proceeding by backwards induction, we will show that for every $p\in [r]$ we have
\begin{equation} \label{e9.22}
Z_r = R_r + \sum_{q=p}^{r-1}n^{r-q} w_{q,r} R_q
 + \sum_{q=0}^{p-1}(-1)^{r-q-1}n^{r-q} D_{q,p} Z_q
\end{equation}
with the convection that a sum of the form $\sum_{i=a}^{b}$ with $b<a$ is equal to zero.
For the case ``$p=r$", notice that for every $q\in \{0,\dots,r-1\}$ the set
$\mathcal{X}_{q,r,r}$ is a singleton which contains the map $u\colon [2]\to\mathbb{Z}$
with $u(1)=q$ and $u(2)=r$; in particular, for every $q\in \{0,\dots,r-1\}$
we have $D_{q,r} = \mu_{q,r}$, and so \eqref{e9.22} for ``$p=r$" follows from
\eqref{e9.16} for ``$s=r$".  Next, let $p\in [r-1]$ and assume that equality
\eqref{e9.22} holds true for $p+1$, that is,
\begin{equation} \label{e9.23}
Z_r  = R_r + \sum_{q=p+1}^{r-1}n^{r-q} w_{q,r} R_q
 + \sum_{q=0}^{p}(-1)^{r-q-1}n^{r-q} D_{q,p+1} Z_q .
\end{equation}
Plugging in \eqref{e9.23} equality \eqref{e9.16} for ``$s=p$", we see that
\begin{align} \label{e9.24}
Z_r & = R_r + \sum_{q=p+1}^{r-1}n^{r-q} w_{q,r} R_q
+ (-1)^{r-p-1}n^{r-p} D_{p,p+1} R_p \ + \\
& \hspace{3.5cm}  +\, \sum_{q=0}^{p-1}(-1)^{r-q-1}n^{r-q} (D_{q,p+1} - D_{p,p+1} \mu_{q,p}) Z_q. \nonumber
\end{align}
Using \eqref{e9.14} for ``$s=p$", \eqref{e9.15} for ``$s=q$" and \eqref{e9.24}, we conclude that
\eqref{e9.22} also holds true for $p$. This completes the proof of \eqref{e9.22}.

Finally, after recalling that $w_{r,r}=1$, the desired equality \eqref{e9.12} follows
from \eqref{e9.22} applied for ``$p=1$" and \eqref{e9.14} applied for ``$s=0$". The proof
of Proposition \ref{p9.4} is completed.
\end{proof}

\subsection{Proof of Lemma \ref{l9.3}} \label{subsubsec9.3}

Our first goal is to approximate the quantities $\mu_{q,p}$. To this end,
fix $q,p\in \{0,1,\dots,r\}$ with $q<p$. Then, by \eqref{e9.7}, we have
\begin{equation} \label{e9.25}
\mu_{q,p} = \binom{p}{q} \sum_{f=0}^{p-q} 2^{p-q-f}\, (-1)^f\, \binom{p-q}{f}
\prod_{u=1}^{f} \Big(1-\frac{p-u}{n}\Big);
\end{equation}
moreover,
\begin{equation} \label{e9.26}
\sum_{f=0}^{p-q} 2^{p-q-f}\, (-1)^f\, \binom{p-q}{f} = 1.
\end{equation}
Also observe that for every $f\in \{0,\dots, p-q\}$ we have
\begin{equation} \label{e9.27}
1\meg \prod_{u=1}^{f}\Big(1-\frac{p-u}{n}\Big)\meg \prod_{v=1}^{r} \Big(1-\frac{v}{n}\Big)
\meg 1- \sum_{v=1}^{r} \frac{v}{n}\meg 1-\frac{r^2}{n}.
\end{equation}
On the other hand, we have
\begin{equation} \label{e9.28}
\sum_{f=0}^{p-q} \Big|2^{p-q-f}(-1)^f \binom{p-q}{f} \Big| =
\sum_{f=0}^{p-q} 2^{p-q-f} \binom{p-q}{f} = 3^{p-q} \mik 3^r
\end{equation}
By \eqref{e9.25}--\eqref{e9.28}, we obtain that
\begin{equation} \label{e9.29}
\Big|\mu_{q,p} - \binom{p}{q} \Big| \mik \binom{p}{q} \frac{3^r r^2}{n} \mik \frac{6^r r^2}{n}.
\end{equation}

Next, we shall approximate the quantities $\mu_u$. So, fix $u\in \bigcup_{s=0}^{r-1}\mathcal{X}_{s,r}$
and, for notational simplicity, set $\ell\coloneqq |\mathrm{dom}(u)|$. Then,
\begin{align} \label{e9.30}
\Big| \mu_u - \prod_{i=2}^{\ell} \binom{u(i)}{u(i-1)} \Big|
= \Big| & \prod_{i=2}^{\ell}\mu_{u(i-1),u(i)} - \prod_{i=2}^{\ell} \binom{u(i)}{u(i-1)} \Big| \\
& \stackrel{\eqref{e9.29}}{\mik} \sum_{j=1}^{\ell-1}
\Big|\prod_{i=1}^{j-1}\mu_{u(i),u(i+1)}\Big| \cdot
\frac{6^r r^2}{n} \cdot \prod_{i=j+1}^{\ell-1} \binom{u(i+1)}{u(i)} \nonumber
\end{align}
with the usual convention that a product of the form $\prod_{i=a}^{b}$ with $b<a$ is equal to $1$.
By \eqref{e9.29} and the fact that $n\meg 6^r r^2$, we see that
$|\mu_{u(i),u(i+1)}| \mik \frac{3}{2} \binom{u(i+1)}{u(i)}$ for every $i\in\{2,\dots,\ell\}$.
Hence, by \eqref{e9.30},
\begin{align} \label{e9.31}
\Big| \mu_u - \prod_{i=2}^{\ell} & \binom{u(i)}{u(i-1)} \Big| \mik
\frac{9^r r^3}{n}\, \prod_{i=2}^{\ell} \binom{u(i)}{u(i-1)} \\
& = \frac{9^r r^3}{n}\, \binom{r}{u(1), u(2)-u(1),\dots,u(\ell)-u(\ell-1)}
\mik \frac{12^r r^3 r!}{n}. \nonumber
\end{align}

Finally, let $s\in \{0,\dots, r-1\}$. Since $|\mathcal{X}_{s,r}| = 2^{r-s-1} \mik 2^r$,
by \eqref{e9.9}, \eqref{e8.1} and \eqref{e9.31},
\begin{equation} \label{e9.32}
|w_{s,r} - (-1)^{r-s}\gamma_{s,r}| \mik
\sum_{u\in\mathcal{X}_{s,r}} \Big|\mu_u - \prod_{i=2}^{|\mathrm{dom}(u)|}
\binom{u(i)}{u(i+1)} \Big| \mik \frac{18^r r^3 r!}{n}.
\end{equation}
By \eqref{e8.5} and \eqref{e9.32}, we conclude that \eqref{e9.11} is satisfied, as desired.


\section{Preparatory lemmas} \label{sec10}

\numberwithin{equation}{section}

Our goal in this section is to present some preparatory lemmas which are needed
for the proof of Theorem \ref{t1.4} but are not related to the main argument.

The first lemma is a simple comparison result of gaussians.
\begin{lem} \label{l10.1}
Let $\mu_1,\mu_2\in\rr$, and let $\sigma_1,\sigma_2\meg 0$ such that $\max\{\sigma_1,\sigma_2\}>0$. Then,
\begin{equation} \label{e10.1}
d_K\big(\mathcal{N}(\mu_1,\sigma_1^2), \mathcal{N}(\mu_2,\sigma_2^2)\big) \mik
\frac{1}{\max\{\sigma_1,\sigma_2\}}\, \big(|\mu_1-\mu_2|+ |\sigma_1 - \sigma_2|\big).
\end{equation}
\end{lem}
\begin{proof}
Clearly, we may assume that $\sigma_1\meg \sigma_2$. (In particular, we have $\sigma_1>0$.)
Set $\sigma\coloneqq \sigma_2/\sigma_1$ and notice that $0\mik \sigma \mik 1$.
We first observe that
\begin{align} \label{e10.2}
d_K\big(\mathcal{N}(\mu_1,\sigma_1^2),&\, \mathcal{N}(\mu_2,\sigma_2^2)\big) =
d_K\Big(\mathcal{N}(0,1), \mathcal{N}\Big( \frac{\mu_2-\mu_1}{\sigma_1}, \sigma^2\Big)\Big) \\
& \mik d_K\Big(\mathcal{N}(0,1), \mathcal{N}\Big( \frac{\mu_2-\mu_1}{\sigma_1}, 1\Big)\Big)+
d_K\big(\mathcal{N}(0,1), \mathcal{N}(0,\sigma^2)\big). \nonumber
\end{align}
The first distance in the right-hand-side of \eqref{e10.2} can be estimated by
\begin{align} \label{e10.3}
d_K\Big(\mathcal{N}(0,1),& \, \mathcal{N}\Big( \frac{\mu_2-\mu_1}{\sigma_1}, 1\Big)\Big) \\
& = \sup_{x\in\rr} \Big|\mathbb{P}\big(\mathcal{N}(0,1)\mik x\big) -
\mathbb{P}\Big(\mathcal{N}(0,1)\mik x-\frac{\mu_2-\mu_1}{\sigma_1} \Big) \Big| \nonumber \\
& \mik \mathbb{P}\Big( -\frac{|\mu_1-\mu_2|}{2\sigma_1}\mik \mathcal{N}(0,1)\mik \frac{|\mu_1-\mu_2|}{2\sigma_1}\Big)
\mik \frac{|\mu_1-\mu_2|}{\sigma_1}. \nonumber
\end{align}
Next, notice that if $\sigma\mik 1/2$, then \eqref{e10.1} follows immediately by \eqref{e10.2} and \eqref{e10.3}.
Thus, in order to estimate the second distance in the right-hand-side of \eqref{e10.2}, we may assume
that $\sigma>1/2$. Then, upon recalling that $\sigma\mik 1$, we obtain that
\begin{align} \label{e10.4}
d_K& \big(\mathcal{N}(0,1),\mathcal{N}(0,\sigma^2)\big)
 = \sup_{x\in\rr} \big|\mathbb{P}\big(\mathcal{N}(0,1)\mik x\big) -
\mathbb{P}\big(\mathcal{N}(0,\sigma^2)\mik x\big)\big| \\
 & = \sup_{x\in\rr} \big|\mathbb{P}\big(\mathcal{N}(0,1)\mik x\big) -
\mathbb{P}\big(\mathcal{N}(0,1)\mik x/\sigma\big)\big|  =
\sup_{x>0} \mathbb{P}\big(x< \mathcal{N}(0,1)\mik x/\sigma\big) \nonumber \\
& =\sup_{x>0}\, \frac{1}{\sqrt{2\pi}}\,\int_{x}^{x+ (\frac{1}{\sigma}-1)x} e^{-\frac{t^2}{2}}\, dt
\mik \sup_{x>0}\, \frac{2|\sigma-1|}{\sqrt{2\pi}}\, x e^{-\frac{x^2}{2}}
\mik |\sigma-1|= \frac{|\sigma_1-\sigma_2|}{\sigma_1}. \nonumber
\end{align}
By \eqref{e10.2}--\eqref{e10.4}, we conclude that \eqref{e10.1} is satisfied.
\end{proof}
The second lemma is a variant of Lemma \ref{l10.1} and refers to mixtures of gaussians.
\begin{lem} \label{l10.2}
Let $\sigma>0$, and let $Y$ be a real-valued random variable defined on some probability space
$(\Omega,\mathcal{F},\mathbf{P})$ with zero mean and finite second moment. Let $Y_\omega$ denote
a realization of\, $Y$, and let $M$ be the mixture with respect to $(\Omega,\mathcal{F},\mathbf{P})$
of $\mathcal{N}(Y_\omega,\sigma^2)$, that~is, for every $x\in\mathbb{R}$ we have
\begin{equation} \label{e10.5}
\mathbb{P}\big(M\mik x)= \mathbf{E}\big[\mathbb{P}\big(\mathcal{N}(Y,\sigma^2)\mik x\big)\big]
\end{equation}
where $\mathbf{E}$ denotes expectation with respect to\, $\mathbf{P}$. Then we have
\begin{equation} \label{e10.6}
d_K\big(M,\mathcal{N}(0,\sigma^2)\big) \mik \frac{\mathbf{E}[Y^2]}{2\sqrt{2\pi e}\,\sigma^2}.
\end{equation}
\end{lem}
\begin{proof}
Let $x\in\mathbb{R}$ be arbitrary. Then, setting $P=[Y\meg 0]$ and $N=[Y<0]$, we have
\begin{align} \label{e10.7}
\big|\prob(M& \mik x)  -\prob\big(\mathcal{N}(0,\sigma^2)\mik x\big)\big| =
\big| \mathbf{E}\big[ \mathbb{P}\big(\mathcal{N}(Y,\sigma^2)\mik x\big)\big] -
\prob\big(\mathcal{N}(0,\sigma^2)\mik x\big)\big| \\
& = \big| \mathbf{E}\big[ \mathbb{P}\big(\mathcal{N}(0,\sigma^2)\mik x-Y\big)\big] -
\prob\big(\mathcal{N}(0,\sigma^2)\mik x\big)\big| \nonumber \\
& = \big| \mathbf{E}\big[ \mathbf{1}_{P} \mathbb{P}\big(x-Y<\mathcal{N}(0,\sigma^2)\mik x\big)
- \mathbf{1}_{N} \prob\big(x<\mathcal{N}(0,\sigma^2)\mik x-Y\big)\big]\big|. \nonumber
\end{align}
Define $F_{\sigma,x}\colon \rr\to\rr$ by
\begin{equation} \label{e10.8}
F_{\sigma,x}(s) \coloneqq \frac{1}{\sqrt{2\pi}\, \sigma} \int_{x-s}^x e^{-\frac{t^2}{2\sigma^2}}\, dt
\end{equation}
with the usual convention that $\int_a^b g(t)\, dt= -\int_b^a g(t)\, dt$ if $a>b$. Using Taylor approximation,
it is easy to see that for every $s\in\mathbb{R}$ we have
\begin{equation} \label{e10.9}
\Big| F_{\sigma,x}(s)- \frac{1}{\sqrt{2\pi}\,\sigma}\, e^{-\frac{x^2}{2\sigma^2}}\, s\Big| \mik
\frac{s^2}{2\sqrt{2\pi e}\,\sigma^2}.
\end{equation}
Moreover, by \eqref{e10.7}, \eqref{e10.8} and the fact that $\mathbf{E}[Y]=0$,
\begin{align} \label{e10.10}
\big|\prob(M \mik x)-\prob\big(\mathcal{N}(0,\sigma^2)\mik x\big)\big| & =
\big| \mathbf{E}\big[ \mathbf{1}_{P} F_{\sigma,x}(Y) +
\mathbf{1}_{N} F_{\sigma,x}(Y)\big]\big| \\
& =  \big| \mathbf{E}\big[F_{\sigma,x}(Y)\big]\big|
\stackrel{\eqref{e10.9}}{\mik}  \nonumber \frac{\mathbf{E}[Y^2]}{2\sqrt{2\pi e}\, \sigma^2}. \qedhere
\end{align}
\end{proof}
\begin{rem} \label{r10.3}
It is easy to see that if the random variable $Y$ in Lemma \ref{l10.2} does not have zero mean,
then we have $d_K\big(M,\mathcal{N}(0,\sigma^2)\big)\mik \frac{|\mathbf{E}[Y]|}{\sigma}+
\frac{\mathrm{Var}(Y)}{2\sqrt{2\pi e}\,\sigma^2}$ where $M$ denotes the mixture associated
with $Y$ via \eqref{e10.6}.
\end{rem}
The third lemma is a coherence property of Hoeffding's operators, and it concerns doubly-indexed tensors
whose entries are products of entries of tensors.
\begin{lem} \label{l10.4}
Let $n\meg s$ be positive integers, let $\boldsymbol{a},\boldsymbol{b}\colon [n]^s\to\rr$ be tensors,
and define $\boldsymbol{\zeta}\colon [n]^s\times[n]^s\to\rr$ by setting
$\boldsymbol{\zeta}(i,p)\coloneqq \boldsymbol{a}(i)\cdot \boldsymbol{b}(p)$. Then for every $i,p\in[n]^s$,
\begin{equation} \label{e10.11}
\mathcal{H}[\boldsymbol{\zeta}](i,p) = \mathcal{H}[\boldsymbol{a}](i)\cdot \mathcal{H}[\boldsymbol{b}](p)
\end{equation}
where $\mathcal{H}[\boldsymbol{\zeta}]$ is the Hoeffding tensor associated with $\boldsymbol{\zeta}$
via \eqref{e9.4}, and $\mathcal{H}[\boldsymbol{a}]$ and $\mathcal{H}[\boldsymbol{b}]$ are the
Hoeffding tensors associated with $\boldsymbol{a}$ and $\boldsymbol{b}$, respectively, via \eqref{e9.2}.
\end{lem}
\begin{proof}
By \eqref{e9.1}, \eqref{e9.3} and the definition of $\boldsymbol{\zeta}$, for every $F,G\subseteq [s]$
and $i,p\in [n]^s$ we have
\begin{equation} \label{e10.12}
\mathcal{A}_{F,G}[\boldsymbol{\zeta}](i\upharpoonright F,p\upharpoonright G) =
\mathcal{A}_{F}[\boldsymbol{a}](i\upharpoonright F) \cdot
\mathcal{A}_{G}[\boldsymbol{b}](p\upharpoonright G).
\end{equation}
Therefore,
\begin{align} \label{e10.13}
& \ \ \mathcal{H}[\boldsymbol{\zeta}](i,p) = \sum_{F,G\subseteq[s]}(-1)^{|F|+|G|}\,
\mathcal{A}_{F,G}[\boldsymbol{\zeta}](i\upharpoonright F,p\upharpoonright G) = \\
= \sum_{F,G\subseteq[s]} & (-1)^{s-|F|}(-1)^{s-|G|}\, \mathcal{A}_{F}[\boldsymbol{a}](i\upharpoonright F)
\cdot \mathcal{A}_{G}[\boldsymbol{b}](p\upharpoonright G) \nonumber \\
= \Big(\sum_{F\subseteq[s]} & (-1)^{s-|F|} \mathcal{A}_{F}[\boldsymbol{a}](i\upharpoonright F)\Big)
\!\cdot\! \Big(\sum_{G\subseteq[s]}(-1)^{s-|G|} \mathcal{A}_{G}[\boldsymbol{b}](p\upharpoonright G)\Big)
\!= \!\mathcal{H}[\boldsymbol{a}](i)\!\cdot\!\mathcal{H}[\boldsymbol{b}](p).  \nonumber \qedhere
\end{align}
\end{proof}
The fourth, and last, result in this section is a representation of the parameters $\Sigma_s$
defined in~\eqref{e1.9}. Specifically, let $\bbx$ be a random tensor which satisfies (\hyperref[A1]{$\mathcal{A}$1})
and (\hyperref[A2]{$\mathcal{A}$2}), and let $(\Omega,\mathcal{F},\prob)$ denote the underlying
probability space. Every realization $\bbx_{\omega}$ of $\bbx$ is a (deterministic) real tensor
$\bbx_\omega\colon [n]^s\to\rr$ and, consequently, we may associate with~$\bbx_{\omega}$
the average and Hoeffding tensors defined in Subsection \ref{subsec9.1}. The following
lemma relates their statistical behavior with the parameters $\Sigma_s$.
\begin{lem} \label{l10.5}
Let $\bbx$ which satisfies \emph{(\hyperref[A1]{$\mathcal{A}$1})}
and \emph{(\hyperref[A2]{$\mathcal{A}$2})}. Then for every $s\in \{0,\dots,d\}$,
\begin{equation} \label{e10.14}
\Big| \sum_{p\in [n]^s} \ave\big[ \mathcal{H}\big[\mathcal{A}_{[s]}[\bbx_\omega]\big](p)^2\big] -
n^s \, \Sigma_s\Big| \mik n^s\, \frac{8d^2 2^d}{n}
\end{equation}
where $\Sigma_s$ is as in \eqref{e1.9}. In particular, we have $\Sigma_s\meg - \frac{8d^2 2^d}{n}$.
\end{lem}
\begin{proof}
Assume that $s\in [d]$, and notice that
\begin{align}
\label{e10.15} & \sum_{p\in [n]^s}
\ave\big[ \mathcal{H}\big[\mathcal{A}_{[s]}[\bbx_\omega]\big](p)^2\big]
\stackrel{\eqref{e9.1},\eqref{e9.2}}{=} \\
= & \ \ave\Big[ \sum_{p\in [n]^s} \sum_{F,G\subseteq [s]} (-1)^{|F|+|G|}\,
\mathcal{A}_{F}[\bbx_\omega](p\upharpoonright F)\,
\mathcal{A}_{G}[\bbx_\omega](p\upharpoonright G) \Big] \nonumber \\
= & \sum_{F,G\subseteq [s]} (-1)^{|F|+|G|} \
\ave\Big[ n^{s-|F\cup G|} \sum_{p\in [n]^{F\cup G}}
\mathcal{A}_{F}[\bbx_\omega](p\upharpoonright F)\,
\mathcal{A}_{G}[\bbx_\omega](p\upharpoonright G) \Big] \nonumber \\
= & \sum_{F,G\subseteq [s]} (-1)^{|F|+|G|} \
\ave\Big[ n^{s-|F\cap G|} \sum_{p\in [n]^{F\cap G}}
\mathcal{A}_{F\cap G}[\bbx_\omega](p)\, \mathcal{A}_{F\cap G}[\bbx_\omega](p) \Big] \nonumber \\
= & \ n^s \sum_{F,G\subseteq [s]} (-1)^{|F|+|G|} \,
\frac{1}{n^{|F\cap G|}} \, \sum_{p\in [n]^{F\cap G}}
\ave\big[\mathcal{A}_{F\cap G}[\bbx_\omega](p)^2\big]. \nonumber
\end{align}
Thus, after observing that for every $H\subseteq [s]$ we have
\begin{equation} \label{e10.16}
(-1)^{s-|H|}= \sum_{\substack{K,L\subseteq [s]\setminus H\\K\cap L=\emptyset}} (-1)^{|K|} \, (-1)^{|L|},
\end{equation}
by \eqref{e10.15} and (\hyperref[A2]{$\mathcal{A}$2}) we obtain that
\begin{equation} \label{e10.17}
\sum_{p\in [n]^s} \ave\big[ \mathcal{H}\big[\mathcal{A}_{[s]}[\bbx_\omega]\big](p)^2\big] =
n^s (-1)^s \sum_{h=0}^s \binom{s}{h} \frac{(-1)^h}{n^h} \sum_{p\in [n]^{h}}
\ave\big[\mathcal{A}_{[h]}[\bbx_\omega](p)^2\big]
\end{equation}
for every $s\in [d]$. On the other hand, if ``$s=0$", then \eqref{e10.17} is trivially valid.
Hence, in what follows, we have at our disposal identity \eqref{e10.17} for every $s\in \{0,1,\dots,d\}$.

We proceed to estimate $\ave\big[\mathcal{A}_{[h]}[\bbx_\omega](p)^2\big]$.
By (\hyperref[A2]{$\mathcal{A}$2}), this quantity is $0$ if $p\notin [n]^{h}_{\mathrm{Inj}}$.
So,~let $p\in [n]^{h}_{\mathrm{Inj}}$ for some $h\in \{0,1,\dots,s\}$ and notice that
\begin{align} \label{e10.18}
\ave\big[ \mathcal{A}_{[h]}&[\bbx_\omega](p)^2\big]  =
\frac{1}{n^{2(d-h)}} \sum_{\substack{q_1,q_2 \in [n]^d_{\mathrm{Inj}}\\p\extension q_1,q_2}} \ave[X_{q_1}X_{q_2}] \\
& = \frac{1}{n^{2(d-h)}} \sum_{t=0}^{d-h} \frac{(n-h)!}{(n-2d+h+t)!} \cdot \frac{1}{t!} \cdot
\Big( \frac{(d-h)!}{(d-h-t)!}\Big)^2 \, \delta_{h+t} \nonumber
\end{align}
where we have used (\hyperref[A2]{$\mathcal{A}$2}) and the definition of $\delta_0,\delta_1,\dots,\delta_d$
in Paragraph \ref{subsubsec1.2.1}. We will need the following elementary fact.
\begin{fact} \label{f10.6}
Let $n,d$ be positive integers with $n\meg 2d$. Then for every $k,\ell\in \{0,\dots,d\}$,
\begin{equation} \label{e10.19}
\Big| \frac{1}{n^\ell}\cdot \frac{(n-k)!}{n-k-\ell)!} - 1\Big| \mik \frac{(k+\ell)^2}{n}.
\end{equation}
\end{fact}
By (\hyperref[A1]{$\mathcal{A}$1}) and the Cauchy--Schwarz inequality, we see that $\delta_t\mik 1$
for every $t\in \{0,\dots,d\}$. Using this observation, \eqref{e10.18} and Fact \ref{f10.6},
we obtain that
\begin{equation} \label{e10.20}
\Big| \ave\big[ \mathcal{A}_{[h]}[\bbx_\omega](p)^2\big] -\delta_h \Big| \mik
\frac{4d^2}{n} + \sum_{t=1}^{d-h} \frac{1}{n^t}\cdot \frac{1}{t!} \cdot
\Big( \frac{(d-h)!}{(d-h-t)!}\Big)^2 \mik \frac{7d^2}{n}.
\end{equation}
This, in turn, implies that for every $h\in \{0,\dots,d\}$ we have
\begin{equation} \label{e10.21}
\Big| \frac{1}{n^h} \sum_{p\in [n]^h}
\ave\big[ \mathcal{A}_{[h]}[\bbx_\omega](p)^2\big] -\delta_h \Big|
\mik \frac{8d^2}{n}.
\end{equation}
Plugging inequality \eqref{e10.21} in the right-hand-side of \eqref{e10.17}, we conclude that
\begin{align} \label{e10.22}
\Big| \sum_{p\in [n]^s} \ave\big[ &\mathcal{H}\big[\mathcal{A}_{[s]}[\bbx_\omega]\big](p)^2\big] -
n^s \, \Sigma_s \Big| \stackrel{\eqref{e1.9}}{\mik} \\
& n^s \sum_{h=0}^s \binom{s}{h}\, \Big| \frac{1}{n^h} \sum_{p\in [n]^h}
\ave\big[ \mathcal{A}_{[h]}[\bbx_\omega](p)^2\big] -\delta_h\Big|
\mik n^s \, \frac{8d^2 2^d}{n}. \nonumber
\end{align}
for every $s\in \{0,1,\dots, d\}$. The proof of Lemma \ref{l10.5} is thus completed.
\end{proof}


\section{Proof of Theorem \ref*{t1.4}} \label{sec11}

\numberwithin{equation}{section}

\subsection{Overview of the argument} \label{subsec11.1}

As we have noted in Section \ref{sec2}, Theorem \ref{t1.4} follows by expressing the distribution of
$\langle \bbth,\bbx\rangle$ as a mixture of distributions of simpler random variables,
and then integrating the Kolmogorov distance of each component of the mixture to an
appropriately chosen gaussian. Of course, in order to implement this strategy,
the selection of the mixture is quite important. Usually, in arguments of this sort,
the components satisfy some independence-type property; however, these are not available
in our context since we only have at our disposal the fact that the distribution
of the random tensor $\bbx$ is invariant under certain symmetries.

The key observation in the proof of Theorem \ref{t1.4} is that the distribution of
$\langle \bbth,\bbx\rangle$ can be expressed as a mixture of $Z$-statistics which can
then be handled effectively by Theorem \ref{t2.2}. In retrospect, this maneuver is
very natural, so much so that sometimes we tend to view Theorems \ref{t1.4} and \ref{t2.2}
as a single result.

\subsection{A preliminary estimate} \label{subsec11.2}

Set $K\coloneqq \ave\big[|X_{(1,\dots,d)}|^3\big]$ and recall that $\kappa\!=\!20 d^3 18^d (2d)!$
and $B= \big\| \frac{1}{n^d}\sum_{i\in [n]^d} X_i\big\|^2_{L_2}$;
in what follows, we will use these constants without further notice.
The following proposition is the main estimate of the proof of Theorem \ref{t1.4}.
\begin{prop} \label{p11.1}
Let the notation and assumptions be as in Theorem \emph{\ref{t1.4}}.
Assume that $n\meg \kappa$ and $\seminorm{\bbth}_0 < \sqrt{n}$, and set
\begin{equation} \label{e11.1}
\sigma_1^2 \coloneqq d^2\Big(1-\frac{\seminorm{\bbth}_0^2}{n} \Big)\,\delta_1
\stackrel{\eqref{e1.11}}{>} 0
\end{equation}
Then
\begin{equation} \label{e11.2}
d_K\big(\langle\bbth,\bbx\rangle,\mathcal{N}(0,\sigma_1^2)\big) \mik
E_1'+ E'_2 + E'_3
\end{equation}
where, setting $\bbth_+\coloneqq \sum_{i\in[n]^d} \theta_i$, we have
\begin{align}
\label{e11.3} E'_1 & \coloneqq 5\,\mathrm{osc}(\bbx)^{1-\alpha}+
5B^{1-\alpha} + 5\Big(\frac{\kappa}{n}\Big)^{1-\alpha} +
\frac{\seminorm{\bbth}_0^2\, B}{\sqrt{2\pi e}\, \sigma_1^2} \\
\label{e11.4} E_2' & \coloneqq 2^{22}C_1d^3\, \frac{\ave\big[|X_{(1,\dots,d)}|^3\big]}{\sigma_1^3}
\Big( \sum_{j=1}^n \Big|\Big(\! \sum_{\substack{i\in[n]^d\\i(1) = j}}\!\! \theta_i\Big)-\frac{\bbth_+}{n}\Big|^3\Big) \\
\label{e11.5} E_3' & \coloneqq \frac{2\sqrt{2} C_d}{\sigma_1}\, \sum_{s=2}^d \binom{d}{s} \,
\sqrt{\Sigma_s+ \frac{8d^2 2^d}{n}} \, \seminorm{\mathcal{H}\big[\mathcal{S}_{[s]}[\bbth]\big]}_s.
\end{align}
Here, $\mathcal{H}\big[\mathcal{S}_{[s]}[\bbth]\big]$ denotes the Hoeffding tensor associated with
the tensor $\mathcal{S}_{[s]}[\bbth]$ for every $s\in \{2,\dots,d\}$. $($See Subsection \emph{\ref{subsec9.1}}.$)$
Moreover, the constant $C_1\meg 1$ is as in \eqref{e2.4}, and $C_d=5d^2e^d(2d)!$ is as in Theorem \emph{\ref{t2.2}}.
\end{prop}
\begin{proof}
Let $Z$ denote the random variable $\sum_{i\in[n]^d} \bbth(i)\, \bbx(\pi\circ i)$ where $\pi$
is a random permutation, independent of $\bbx$, which is uniformly distributed on $\mathbb{S}_n$.
(Here, $\bbx(\pi\circ i)$ denotes the $(\pi\circ i)$-entry of $\bbx$.) Since $\bbx$ is exchangeable,
we see that $\langle \bbth,\bbx\rangle$ and $Z$ have the same distribution. Consequently, we have
\begin{equation} \label{e11.6}
d_K\big(\langle \bbth,\bbx\rangle, \mathcal{N}(0,\sigma_1^2)\big) =
d_K\big(Z, \mathcal{N}(0,\sigma_1^2)\big).
\end{equation}
Let $(\Omega,\mathcal{F},\prob)$ denote the underlying probability space on which the random
tensor $\bbx$ is defined; as we have already pointed out, every realization $\bbx_\omega$
of $\bbx$ is a deterministic real tensor. It follows from its definition, that $Z$ is
a mixture of the $Z$-statistics $Z_\omega$ associated with the tensors
$\boldsymbol{\zeta}_\omega\colon [n]^d\times [n]^d\to\rr$ defined~by
\begin{equation} \label{e11.7}
\boldsymbol{\zeta}_\omega(i,p)\coloneqq \bbth(i) \cdot \bbx_\omega(p).
\end{equation}
Notice that $\boldsymbol{\zeta}_\omega$ satisfies
(\hyperref[A5]{$\mathcal{A}$5}) and (\hyperref[A6]{$\mathcal{A}$6}). Moreover, for every
$s\in [d]$ and every $i,p\in [n]^s$,
\begin{equation} \label{e11.8}
\mathcal{A}_{[s],[s]}[\boldsymbol{\zeta}_\omega](i,p) =
\mathcal{A}_{[s]}[\bbth](i) \cdot \mathcal{A}_{[s]}[\bbx_\omega](p)
\end{equation}
and so, by Lemma \ref{l10.4},
\begin{equation} \label{e11.9}
\mathcal{H}\big[\mathcal{A}_{[s],[s]}[\boldsymbol{\zeta}_\omega]\big](i,p) =
\mathcal{H}\big[\mathcal{A}_{[s]}[\bbth]\big](i) \cdot \mathcal{H}\big[\mathcal{A}_{[s]}[\bbx_\omega]\big](p).
\end{equation}
Therefore, setting
\begin{equation}\label{e11.10}
\overline{\boldsymbol{X}}_\omega \coloneqq \frac{1}{n^d}\sum_{i\in[n]^d} \bbx_\omega(i),
\end{equation}
by Proposition \ref{p9.4} and Lemma \ref{l9.3}, we obtain that
\begin{align} \label{e11.11}
& Z_\omega = \sum_{s=1}^{d}n^{d-s} \binom{d}{s} \Big(1+\frac{\lambda_s}{n}\Big)
\sum_{i\in[n]^s_{\mathrm{Inj}}} \mathcal{H}\big[\mathcal{A}_{[s]}[\bbth]\big](i)\cdot
\mathcal{H}\big[\mathcal{A}_{[s]}[\bbx_\omega]\big](\pi\circ i) \ + \\
& \hspace{5.5cm} + \, \Big(1+\frac{\lambda_0}{n}\Big)\, \bbth_+\, \overline{\boldsymbol{X}}_\omega \nonumber \\
\stackrel{\eqref{e9.1},\eqref{e9.2}}{=} & \sum_{s=1}^{d} \binom{d}{s} \Big(1+\frac{\lambda_s}{n}\Big)
\sum_{i\in[n]^s_{\mathrm{Inj}}} \mathcal{H}\big[\mathcal{S}_{[s]}[\bbth]\big](i)\cdot
\mathcal{H}\big[\mathcal{A}_{[s]}[\bbx_\omega]\big](\pi\circ i)
+ \Big(1+\frac{\lambda_0}{n}\Big)\, \bbth_+\, \overline{\boldsymbol{X}}_\omega  \nonumber
\end{align}
where $|\lambda_s|\mik d^3 18^d d! \mik \kappa/20$ for every $s\in \{0,\dots,d\}$;
in particular, since $n\meg \kappa$,
\begin{equation} \label{e11.12}
\frac{19}{20}\mik 1+\frac{\lambda_s}{n}\mik \frac{21}{20} \ \ \ \ \text{ for every } s\in \{0,1,\dots,d\}.
\end{equation}
Set $Y\coloneqq \big(1+\frac{\lambda_0}{n}\big)\, \bbth_+ \overline{\bbx}$, and let $M$ denote the mixture
in \eqref{e10.5} associated with $Y$ and ``$\sigma=\sigma_1$". By \eqref{e11.6}, Lemma \ref{l10.2}
and the triangle inequality for the Kolmogorov distance, we have
\begin{align} \label{e11.13}
& d_K\big(\langle \bbth,\bbx\rangle,\mathcal{N}(0,\sigma_1^2)\big) \mik
d_K(Z,M)+ d_K\big(M,\mathcal{N}(0,\sigma_1^2)\big) \mik \\
\mik  \ave\big[ d_K\big(Z,& \mathcal{N}(Y,\sigma_1^2)\big)\big]
+ \frac{1}{2\sqrt{2\pi e}\, \sigma_1^2}\, \ave[Y^2]
\stackrel{\eqref{e11.12}}{\mik} \ave\big[ d_K\big(Z-Y,\mathcal{N}(0,\sigma_1^2)\big)\big]
+ \frac{\seminorm{\bbth}_0^2\, B}{\sqrt{2\pi e}\, \sigma_1^2}. \nonumber
\end{align}
Let $\sigma^2_\omega$ denote the variance of the one dimensional component,
\[ d\, \Big(1+\frac{\lambda_1}{n}\Big) \sum_{i\in [n]}
\mathcal{H} \big[\mathcal{S}_{[1]}[\bbth]\big](i) \cdot
\mathcal{H}\big[\mathcal{A}_{[1]}[\bbx_\omega]\big](\pi\circ i), \]
of $Z_\omega-Y_\omega$. Since $\seminorm{\bbth}_1=1$, we have
\begin{align} \label{e11.14}
\sigma^2_\omega & = d^2\Big(1+\frac{\lambda_1}{n}\Big)^2 \frac{1}{n-1}
  \Big(\sum_{i\in [n]} \mathcal{H}\big[\mathcal{S}_{[1]}[\bbth]\big](i) ^2\Big)
  \Big(\sum_{p\in [n]} \mathcal{H}\big[\mathcal{A}_{[1]}[\bbx_\omega]\big](p)^2\Big)\\
& = d^2\Big(1+\frac{\lambda_1}{n}\Big)^2 \frac{1}{n-1}
  \Big(1 - \frac{\boldsymbol{\theta}_+^2}{n}\Big)
	\Big(\sum_{p\in [n]} \mathcal{H}\big[\mathcal{A}_{[1]}[\bbx_\omega]\big](p)^2\Big) \nonumber \\
& = d^2\Big(1+\frac{\lambda_1}{n}\Big)^2 \frac{1}{n}
  \Big(1 - \frac{\boldsymbol{\theta}_+^2}{n}\Big)
	\Big(\sum_{p\in [n]} \mathcal{H}\big[\mathcal{A}_{[1]}[\bbx_\omega]\big](p)^2\Big)\, + \nonumber \\
& \hspace{1cm}  +\, d^2\Big(1+\frac{\lambda_1}{n}\Big)^2 \frac{1}{n(n-1)}
  \Big(1 - \frac{\boldsymbol{\theta}_+^2}{n}\Big)
	\Big(\sum_{p\in [n]} \mathcal{H}\big[\mathcal{A}_{[1]}[\bbx_\omega]\big](p)^2\Big) \nonumber \\
& = d^2 \Big(1 - \frac{\boldsymbol{\theta}_+^2}{n}\Big)
  \Big(\frac{1}{n}\sum_{p\in[n]}\! \mathcal{A}_{[1]}[\bbx_\omega](p)^2\Big)
  + d^2\Big(1 - \frac{\boldsymbol{\theta}_+^2}{n}\Big)\frac{1}{n}\, E_\omega \nonumber
\end{align}
where
\begin{equation} \label{e11.15}
E_\omega \coloneqq
\Big(\frac{1}{n} \sum_{p\in [n]} \mathcal{H}\big[\mathcal{A}_{[1]}[\bbx_\omega]\big](p)^2\Big)\,
\Big[ \frac{n}{n-1}\Big(1+\frac{\lambda_1}{n}\Big)^2 + 2\lambda_1+\frac{\lambda_1^2}{n}\Big]
-n\overline{\boldsymbol{X}}_\omega^2.
\end{equation}
\begin{claim} \label{c11.2}
We have
\begin{equation} \label{e11.16}
\ave\big[ |E_\omega|\big] \mik n B + \kappa.
\end{equation}
\end{claim}
\begin{proof}[Proof of Claim \emph{\ref{c11.2}}]
By Lemma \ref{l10.5} and \eqref{e1.9},
\begin{equation} \label{e11.17}
\frac{1}{n}\sum_{p\in [n]} \ave\big[\mathcal{H}\big[\mathcal{A}_{[1]}[\bbx_\omega]\big](p)^2\big] \mik
\delta_1-\delta_0 +\frac{8d^2 2^d}{n} \stackrel{(\hyperref[A1]{\mathcal{A}1})}{\mik}
2+\frac{8d^2 2^d}{n}.
\end{equation}
Since $n\meg \kappa$ and $|\lambda_1|\mik \kappa/20$, \eqref{e11.16} follows from \eqref{e11.17},
the definition of $\overline{\bbx}_\omega$ in \eqref{e11.10} and the choice of the constant $B$.
\end{proof}
By Claim \ref{c11.2}, \eqref{e11.14}, the definition of oscillation in \eqref{e1.10}
and the choice of $\sigma_1^2$ in~\eqref{e11.1}, we obtain that
\begin{equation} \label{e11.18}
\ave \big[|\sigma^2_\omega-\sigma_1^2|\big] \mik
d^2 \Big(1 - \frac{\seminorm{\bbth}_0^2}{n}\Big) \Big[\mathrm{osc}(\boldsymbol{X}) +
B + \frac{\kappa}{n}\Big].
\end{equation}
Define the event
\begin{equation} \label{e11.19}
\Gamma \coloneqq \Big\{|\sigma_\omega^2-\sigma_1^2| \mik \frac{d^2}{4}
\Big(1 - \frac{\seminorm{\bbth}_0^2}{n}\Big) \, \max\Big\{
\mathrm{osc}(\boldsymbol{X})^\alpha,B^\alpha,
\Big(\frac{\kappa}{n}\Big)^{\alpha}\Big\}\Big\}.
\end{equation}
By Markov's inequality and \eqref{e11.18}, we see that
\begin{equation} \label{e11.20}
\mathbb{P}(\Gamma^{\complement}) \mik
4\mathrm{osc}(\boldsymbol{X})^{1-\alpha} + 4B^{1-\alpha} +
4\Big(\frac{\kappa}{n}\Big)^{1-\alpha}.
\end{equation}
Moreover, by \eqref{e1.11} and \eqref{e11.1}, for every $\omega\in \Gamma$ we have
\begin{equation} \label{e11.21}
\sigma_\omega\meg\frac{\sigma_1}{2}.
\end{equation}
On the other hand, setting for every  $\omega\in \Gamma$ and every $s\in \{2,\dots,d\}$
\begin{align}
\label{e11.22} \Lambda_\omega & \coloneqq
d^3 \Big(1+\frac{\lambda_1}{n}\Big)^3 \Big(\sum_{i\in [n]} \big|\mathcal{H}\big[\mathcal{S}_{[1]}[\bbth]\big](i)\big|^3\Big)
\Big(\sum_{p\in [n]} \big|\mathcal{H}[\mathcal{A}_{[1]}[\bbx_\omega]\big](p)\big|^3\Big) \\
\label{e11.23} \beta_{s,\omega} & \coloneqq
\binom{d}{s}^2 \Big(1+\frac{\lambda_s}{n}\Big)^2
\Big(\sum_{i\in [n]^s} \mathcal{H}\big[\mathcal{S}_{[s]}[\bbth]\big](i)^2\Big)
\Big(\sum_{p\in [n]^s} \mathcal{H}[\mathcal{A}_{[s]}[\bbx_\omega]\big](p)^2\Big),
\end{align}
by \eqref{e11.11}, the choice of $Y$, the fact that $n\meg \kappa\meg 4d^2$,
Theorem \ref{t2.2} and \eqref{e11.21}, we have
\begin{equation} \label{e11.24}
d_K\big(Z_\omega-Y_\omega, \, \mathcal{N}(0,\sigma_\omega^2)\big)
\mik \frac{2^{21}C_1\Lambda_\omega}{\sigma_1^{3}n}+ \frac{2C_d}{\sigma_1}
\sum_{s=2}^{d} \sqrt{ \frac{\beta_{s,\omega}}{n^s}}
\end{equation}
where $C_1\meg 1$ is as in \eqref{e2.4} and $C_d=5d^2 e^d (2d)!$.
\begin{claim} \label{c11.3}
We have
\begin{equation} \label{e11.25}
\ave[\mathbf{1}_{\Gamma} \Lambda_\omega] \mik 2 d^3 K n \Big( \sum_{j=1}^n
\Big|\Big(\!\sum_{\substack{i\in [n]^d\\i(1)=j}}\!\! \theta_i\Big)-\frac{\bbth_+}{n}\Big|^3\Big).
\end{equation}
\end{claim}
\begin{proof}[Proof of Claim \emph{\ref{c11.3}}]
By the definition of $\Lambda_\omega$ in \eqref{e11.22} and using
\eqref{e9.1}, \eqref{e9.2}, (\hyperref[A1]{$\mathcal{A}$1}) and the triangle inequality, we obtain that
\begin{equation} \label{e11.26}
\ave[\mathbf{1}_{\Gamma} \Lambda_\omega] \mik d^3 \Big(1+\frac{\lambda_1}{n}\Big)^3 K n
\Big( \sum_{j=1}^n \Big|\Big(\!\sum_{\substack{i\in [n]^d\\i(1)=j}}\!\!\theta_i\Big)-
\frac{\bbth_+}{n}\Big|^3\Big).
\end{equation}
The claim follows from this estimate and \eqref{e11.12}.
\end{proof}
\begin{claim} \label{c11.4}
For every $s\in \{2,\dots,d\}$ we have
\begin{equation} \label{e11.27}
\ave[\mathbf{1}_{\Gamma} \beta_{s,\omega}] \mik 2 n^s \binom{d}{s}^2
\Big(\Sigma_s+ \frac{8d^2 2^d}{n}\Big)\, \seminorm{\mathcal{H}\big[\mathcal{S}_{[s]}[\bbth]\big]}_s^2.
\end{equation}
\end{claim}
\begin{proof}[Proof of Claim \emph{\ref{c11.4}}]
We first observe that, by \eqref{e11.12} and Lemma \ref{l10.5}, we have
\begin{equation} \label{e11.28}
\binom{d}{s}^2 \Big(1+\frac{\lambda_s}{n}\Big)^2 \ave\Big[ \Big(\sum_{p\in [n]^s}
\mathcal{H}[\mathcal{A}_{[s]}[\bbx_\omega]\big](p)^2\Big) \Big] \mik
2 n^s \binom{d}{s}^2 \Big(\Sigma_s+ \frac{8d^2 2^d}{n}\Big).
\end{equation}
Moreover, by the definition of the seminorm $\seminorm{\cdot}_s$ in \eqref{e1.5},
\begin{equation} \label{e11.29}
\sum_{i\in [n]^s} \mathcal{H}\big[\mathcal{S}_{[s]}[\bbth]\big](i)^2 =
\seminorm{\mathcal{H}\big[\mathcal{S}_{[s]}[\bbth]\big]}_s^2.
\end{equation}
Inequality \eqref{e11.27} follows by combining \eqref{e11.28} and \eqref{e11.29}.
\end{proof}
We are ready for the last step of the argument. By \eqref{e11.13}, we have
\begin{equation} \label{e11.30}
d_K\big( \langle\bbth,\bbx\rangle, \mathcal{N}(0,\sigma_1^2)\big) \mik
\prob(\Gamma^{\complement}) +  \ave\big[ \mathbf{1}_{\Gamma} d_K\big(Z-Y,\mathcal{N}(0,\sigma_1^2)\big)\big]
+ \frac{\seminorm{\bbth}_0^2\, B}{\sqrt{2\pi e}\, \sigma_1^2}.
\end{equation}
Moreover, by the triangle inequality for the Kolmogorov distance and Lemma \ref{l10.1},
\begin{align} \label{e11.31}
\ave\big[\mathbf{1}_{\Gamma} & d_K\big( Z-Y, \mathcal{N}(0,\sigma_1^2)\big)\big] \mik
\ave\big[\mathbf{1}_{\Gamma} d_K\big( Z-Y, \mathcal{N}(0,\sigma_\omega^2)\big)\big]+
\ave\Big[ \mathbf{1}_\Gamma\,\frac{|\sigma_1-\sigma_\omega|}{\max\{\sigma_1,\sigma_\omega\}}\Big] \\
& \mik \ave\big[\mathbf{1}_{\Gamma} d_K\big( Z-Y, \mathcal{N}(0,\sigma_\omega^2)\big)\big]+
\frac{1}{\sigma_1^2}\, \ave\big[ |\sigma_1^2-\sigma_\omega^2|\big] \nonumber \\
& \stackrel{\eqref{e11.1},\eqref{e11.18}}{\mik}
\ave\big[\mathbf{1}_{\Gamma} d_K\big( Z-Y, \mathcal{N}(0,\sigma_\omega^2)\big)\big] +
\frac{1}{\delta_1} \Big(\mathrm{osc}(\boldsymbol{X}) + B + \frac{\kappa}{n}\Big) \nonumber \\
& \hspace{1.15cm} \stackrel{\eqref{e1.11}}{\mik}
\ave\big[\mathbf{1}_{\Gamma} d_K\big( Z-Y, \mathcal{N}(0,\sigma_\omega^2)\big)\big] +
\mathrm{osc}(\boldsymbol{X})^{1-\alpha} + B^{1-\alpha} + \Big(\frac{\kappa}{n}\Big)^{1-\alpha}. \nonumber
\end{align}
By \eqref{e11.24}, Claims \ref{c11.2}--\ref{c11.4} and
the Cauchy--Schwarz inequality,
\begin{align} \label{e11.32}
\ave\big[\mathbf{1}_{\Gamma} d_K\big( Z-Y, \mathcal{N}(0,\sigma_\omega^2)\big) & \big]
\mik \frac{2^{22} C_1 d^3 K}{\sigma_1^3} \Big( \sum_{j=1}^n
\Big|\Big(\!\sum_{\substack{i\in [n]^d\\i(1)=j}}\!\! \theta_i\Big)-\frac{\bbth_+}{n}\Big|^3\Big) \, + \\
& \hspace{0.5cm} + \frac{2\sqrt{2} C_d}{\sigma_1} \sum_{s=2}^d \binom{d}{s}
\sqrt{\Sigma_s+ \frac{8d^2 2^d}{n}}\, \seminorm{\mathcal{H}\big[\mathcal{S}_{[s]}[\bbth]\big]}_s^2. \nonumber
\end{align}
The desired estimate \eqref{e11.2} follows by combining \eqref{e11.30}, \eqref{e11.20},
\eqref{e11.31} and \eqref{e11.32}, and invoking the definition of $E'_1, E'_2, E'_3$.
The proof of Proposition \ref{p11.1} is completed.
\end{proof}

\subsection{Completion of the proof} \label{subsec11.3}

We are finally in a position to complete the proof of Theorem \ref{t1.4}. To this end, we first
observe that we may assume that $n\meg \kappa$ and
$\seminorm{\bbth}_0 \mik \sqrt{n}/2$. (Otherwise, \eqref{e1.12} is straightforward.)
By \eqref{e11.1}, the last assumption implies, in particular, that
\begin{equation} \label{e11.33}
\sigma_1^2 \meg \frac{3d^2\delta_1}{4}.
\end{equation}
Next observe that, by the triangle inequality for the Kolmogorov distance, Lemma \ref{l10.1},
\eqref{e11.33} and Proposition \ref{p1.2}, we have
\begin{align} \label{e11.34}
d_K\big(\mathcal{N}(0,\sigma_1^2),\, &\mathcal{N}(0,\sigma^2)\big) \mik
d_K\big(\mathcal{N}(0,d^2\delta_1),\mathcal{N}(0,\sigma^2)\big)+ \frac{\seminorm{\bbth}_0^2}{n} \\
&\mik \Big| \frac{\delta_0\, (\seminorm{\bbth}_0^2-1)}{d^2\delta_1}\Big| +
\frac{1}{d^2\delta_1}\, \sum_{s=2}^d \binom{d}{s}^2\, s! \,
|\Sigma_s| \cdot \seminorm{\bbth}_s^2 + \frac{\seminorm{\bbth}_0^2}{n}. \nonumber
\end{align}
(Recall that $\sigma^2$ denotes the variance of $\langle\bbth,\bbx\rangle$.) On the other hand,
by (\hyperref[A1]{$\mathcal{A}$1}), (\hyperref[A2]{$\mathcal{A}$2}), \eqref{e1.7},
Fact \ref{f10.6} and the definition of $B$, we have
\begin{equation} \label{e11.35}
B =\frac{1}{n^{2d}} \sum_{i,j\in [n]^d} \ave[X_i X_j] \mik |\delta_0| + \frac{4d^2}{n}.
\end{equation}
Moreover, by H\"{o}lder's inequality,
\begin{equation} \label{e11.36}
\sum_{j=1}^n \Big|\Big(\!\sum_{\substack{i\in [n]^d\\i(1)=j}} \!\!\theta_i\Big)-\frac{\bbth_+}{n}\Big|^3
\mik 2^4 \sum_{j=1}^n \Big|\!\sum_{\substack{i\in [n]^d\\i(1)=j}} \theta_i\Big|^3
\end{equation}
Finally, by the Cauchy--Schwarz inequality, for every $s\in \{2,\dots, d\}$ we have
\begin{align} \label{e11.37}
 \seminorm{\mathcal{H}\big[\mathcal{S}_{[s]}[\bbth]\big]}_s^2 & \ \, =
\sum_{i\in [n]^s} \mathcal{H}\big[\mathcal{S}_{[s]}[\bbth]\big](i)^2 = n^{2d-2s}
\sum_{i\in [n]^s} \mathcal{H}\big[\mathcal{A}_{[s]}[\bbth]\big](i)^2 \\
& \stackrel{\eqref{e9.2}}{=} n^{2d-2s} \sum_{i\in [n]^s} \Big( \sum_{F\subseteq [s]}
(-1)^{s-|F|} \mathcal{A}_F[\bbth](i\upharpoonright F)\Big)^2 \nonumber \\
& \ \mik n^{2d-2s}\, 2^s \sum_{F\subseteq [s]} \sum_{i\in [n]^s}
\mathcal{A}_F[\bbth](i\upharpoonright F)^2 \nonumber \\
& \ \mik n^{2d-2s}\, 2^s \sum_{F\subseteq [s]} \sum_{i\in [n]^s} \Big(
\frac{1}{n^{s-|F|}} \sum_{i\upharpoonright F\extension j\in [n]^s}
\mathcal{A}_{[s]}[\bbth](j)^2\Big) \nonumber \\
& \ = n^{2d-2s}\, 2^s \sum_{F\subseteq [s]} \sum_{j\in [n]^s}
\mathcal{A}_{[s]}[\bbth](j)^2 \stackrel{\eqref{e1.5}}{=} 2^{2s}
\seminorm{\bbth}_s^2 \mik 12 s!\, \seminorm{\bbth}_s^2 \nonumber
\end{align}
where in the last inequality we have used the fact that $2^{2s}\mik 12 s!$ for every integer $s\meg 2$.

Inequality \eqref{e1.12} now follows from \eqref{e11.2}, \eqref{e11.33}--\eqref{e11.37},
Fact \ref{f1.3} and the estimate $C_1\mik 451$ obtained in \cite{CF15}.
The proof of Theorem \ref{t1.4} is completed.


\section{Proofs of applications} \label{sec12}

\numberwithin{equation}{section}

\subsection{Exchangeable random vectors} \label{subsec12.1}

We start by giving the proof of Corollary \ref{c3.1}.
\begin{proof}[Proof of Corollary \emph{\ref{c3.1}}]
First observe that we may assume that $\mathrm{osc}(\bbx)\mik 1$ and $n\meg \kappa_1$.
(Otherwise, \eqref{e3.2} is straightforward.) Therefore, by \eqref{e3.1}, the non-degenericity
condition \eqref{e1.11} is satisfied for every $\alpha\in (0,1)$. Taking into account this remark,
the result follows by Theorem~\ref{t1.4} albeit with a different absolute constant. The fact that
in \eqref{e3.2} we can select the constant $\kappa_1=4320$ follows after noticing that,
for the one-dimensional case, in the proof of Proposition \ref{p11.1} it suffices to
apply \eqref{e2.4} with $C_1=451$.
\end{proof}
\begin{rem} \label{r12.1}
It is instructive to compare Corollary \ref{c3.1} with the classical work of Diaconis and
Freedman~\cite{DF80} which estimates the total variation distance between the marginals of the law
of an exchangeable random vector, and the corresponding marginals of a mixture of product measures.
The results in~\cite{DF80} are very informative, but unfortunately they become quantitatively
less useful for exchangeable random vectors whose entries are not finite-valued. On the other hand, by considering
a statistically less rigid quantity, Corollary \ref{c3.1} provides estimates of equal
quantitative strength for all permissible distributions.
\end{rem}
We proceed to estimate the oscillation of exchangeable random vectors. We first treat
the case of random vectors whose entries have finite fourth moment.
\begin{fact} \label{f12.2}
Let $\bbx$ be an exchangeable random vector in $\rr^n$ $(n\meg 2)$ whose entries have zero mean,
unit variance and finite fourth moment. Then we have
\begin{equation} \label{e12.1}
\mathrm{osc}(\bbx) \mik \sqrt{ \big|\ave[X_1^2 X_2^2]-1\big|} + \frac{\ave[X_1^4]^{1/2}}{\sqrt{n}}.
\end{equation}
\end{fact}
\begin{proof}
By the Cauchy--Schwarz inequality and our assumptions, we have
\begin{align} \label{e12.2}
& \mathrm{osc}(\bbx) \stackrel{\eqref{e1.10}}{=}  \Big\|\frac{1}{n}\sum_{i=1}^n X_i^2-1\Big\|_{L_1}
\mik  \Big\|\frac{1}{n}\sum_{i=1}^n X_i^2-1\Big\|_{L_2} \\
= & \ \Big( \frac{1}{n^2} \sum_{i,j=1}^n \ave[X_i^2X_j^2] + 1 -
\frac{2}{n}\sum_{i=1}^n \ave[X_i^2]\Big)^{1/2}  \nonumber \\
= & \ \Big( \frac{\ave[X_1^4]-\ave[X_1^2X_2^2]}{n}+ \ave[X_1^2X_2^2]-1\Big)^{1/2}
\mik \sqrt{ \big|\ave[X_1^2 X_2^2]-1\big|} + \frac{\ave[X_1^4]^{1/2}}{\sqrt{n}}. \nonumber \qedhere
\end{align}
\end{proof}
The following proposition supplements Fact \ref{f12.2}, and it estimates the oscillation of exchangeable
random vectors whose entries have finite third moment.
\begin{prop} \label{p12.3}
Let $\bbx$ be an exchangeable random vector in $\rr^n$ $(n\meg 2)$ whose entries have zero mean,
unit variance and finite third moment. Then we have
\begin{equation} \label{e12.3}
\mathrm{osc}(\bbx) \mik \sqrt{ \big|\ave[X_1^2 X_2^2]-1\big|} +
\frac{4\ave\big[|X_1|^3\big]}{\sqrt[4]{n}}.
\end{equation}
\end{prop}
\begin{proof}
We will use a standard truncation argument. Specifically, set $K\coloneqq \ave\big[|X_1|^3\big]$,
and let $\lambda>0$ be a cut-off parameter to be specified later.
Moreover, for every $i\in [n]$ set $G_i\coloneqq X_i^2\, \mathbf{1}_{[|X_i|\mik\lambda]}$
and $B_i\coloneqq X_i^2\,\mathbf{1}_{[|X_i|>\lambda]}$, and notice that
\begin{equation} \label{e12.4}
\ave[G_i^2] = \ave[X_i^3 X_i\,\mathbf{1}_{[|X_i|\mik\lambda]}]\mik
\lambda K;
\end{equation}
on the other hand, by H\"{o}lder's inequality and Markov's inequality,
\begin{equation} \label{e12.5}
\|B_i\|_{L_1} \mik \ave\big[ |X_i|^3\big]^{\frac{2}{3}}\,
\mathbb{P}\big(|X_i|>\lambda\big)^{\frac{1}{3}}\mik \frac{K}{\lambda}.
\end{equation}
Since the random vectors $(G_1,\dots,G_n)$ and $(B_1,\dots,B_n)$ are exchangeable, we obtain that
\begin{align} \label{e12.6}
\mathrm{osc}(\bbx) & \stackrel{\eqref{e1.10}}{=}  \Big\|\frac{1}{n}\sum_{i=1}^n X_i^2-1\Big\|_{L_1}
\mik  \Big\|\frac{1}{n}\sum_{i=1}^n G_i-1\Big\|_{L_1} + \|B_1\|_{L_1} \\
& \stackrel{\eqref{e12.5}}{\mik}
\Big\|\frac{1}{n}\sum_{i=1}^n G_i-1\Big\|_{L_1} + \frac{K}{\lambda}
\mik \Big\|\frac{1}{n}\sum_{i=1}^n G_i-1\Big\|_{L_2} + \frac{K}{\lambda} \nonumber \\
& \ \ =  \Big( \frac{\ave[G_1^2]-\ave[G_1 G_2]}{n}+ \ave[G_1 G_2]-1+ 2\ave[B_1]\Big)^{1/2}
+ \frac{K}{\lambda} \nonumber \\
& \ \ \mik \Big( \frac{\ave[G_1^2]}{n}+ \ave[X_1^2X_2^2]-1 +2 \ave[B_1]\Big)^{1/2}
+ \frac{K}{\lambda} \nonumber \\
& \, \stackrel{\eqref{e12.4},\eqref{e12.5}}{\mik} \frac{\sqrt{\lambda K}}{\sqrt{n}} +
\sqrt{ \big|\ave[X_1^2 X_2^2]-1\big|} + \sqrt{2} \frac{\sqrt{K}}{\sqrt{\lambda}}+
\frac{K}{\lambda} \nonumber
\end{align}
where we have used the Cauchy--Schwarz inequality, and that $0\mik \ave[G_1G_2]\mik \ave[G_1^2]$ and
$\ave[G_1G_2]\mik \ave[X_1^2X_2^2]$. The result follows by applying \eqref{e12.6}
for $\lambda=\sqrt[4]{n}$.
\end{proof}

\subsubsection{Optimality} \label{subsubsec12.1.1}

It follows from \eqref{e3.2} that if $\bbx$ is an exchangeable and isotropic random vector in~$\rr^n$ which satisfies
\begin{equation} \label{e12.7}
\ave\big[|X_1|^3]=O(1),
\end{equation}
then for every positive integer $k\mik n$ we have
\begin{equation} \label{e12.8}
d_K\Big( \frac{X_1+\dots+X_k}{\sqrt{k}}, \mathcal{N}(0,1)\Big) =
O\Big( \mathrm{osc}(\bbx)+ \frac{k}{n} + \frac{1}{\sqrt{k}}\Big).
\end{equation}
Our goal in this subsection is to show that this estimate is optimal, in the sense that
each of the three error terms that appear in the right-hand-side of \eqref{e12.8} can be dominated
by the Kolmogorov distance in \eqref{e12.8} for an appropriately selected exchangeable and
isotropic random vector which satisfies \eqref{e12.7}.

The necessity of the error term $O(1/\sqrt{k})$ is a classical observation, and
it can be attained by considering a random vector with independent Rademacher entries.
The necessity of the first error term---the oscillation---is shown in the next example.
\begin{examp} \label{ex12.4}
Let $n\meg 2$ be an arbitrary integer. Fix $0<\ee\mik 1/2$, and let $\bbx$ be a random vector
in $\rr^n$ whose distribution satisfies
\begin{equation} \label{e12.9}
\prob(\bbx\in A)= \ee\, \prob(\boldsymbol{Z}\in A)+ (1-\ee)\, \prob\Big(\frac{\boldsymbol{R}}{\sqrt{1-\ee}}\in A\Big)
\end{equation}
for every Borel subset $A$ of $\rr^n$, where $\boldsymbol{Z}$ is a random vector in $\rr^n$ with
constant zero entries and $\boldsymbol{R}$ is a random vector in $\rr^n$ with independent Rademacher entries.
Notice that $\bbx$ is exchangeable, isotropic and bounded; in particular, it satisfies \eqref{e12.7}. Moreover,
\begin{equation} \label{e12.10}
\mathrm{osc}(\bbx)=2\ee.
\end{equation}
Finally, for every positive integer $k\mik n$ we have
\begin{equation} \label{e12.11}
d_K\Big( \frac{X_1+\dots+X_k}{\sqrt{k}}, \mathcal{N}(0,1)\Big) \meg
\frac{1}{2}\,\prob\Big( \frac{X_1+\dots+X_k}{\sqrt{k}}=0\Big)
\meg \frac{\ee}{2} \stackrel{\eqref{e12.10}}{\meg} \frac{\mathrm{osc}(\bbx)}{4}.
\end{equation}
\end{examp}
The necessity of the error term $O(k/n)$ is the content of the following example.
\begin{examp} \label{ex12.5}
It is a slight modification of the results of Diaconis and Freedman in \cite[Section 4]{DF80}.
(See also \cite{DF87} for closely related results.)

Specifically, let $n$ be a sufficiently large positive even integer, and
let $\boldsymbol{\xi}=(\xi_1,\dots,\xi_n)$ be a random vector which is uniformly
distributed on the set of all $x\in \{-1,1\}^n$ whose coordinates sum-up to $0$.
Notice that $\boldsymbol{\xi}$~is~exchangeable. Using a non-asymptotic version
of Stirling's approximation---see \cite{Ro55}---and elementary computations,
it is easy to verify that for every integer $k\in \{16,\dots,\frac{n}{14}\}$ and every integer
$a\in \big[-\frac{\sqrt{k}}{2},\frac{\sqrt{k}}{2}\big]$ we have
\begin{equation} \label{e12.12}
\prob(S_k=a)\meg \prob(R_k=a)\cdot \Big(1+\frac{1}{32}\, \frac{k}{n}\Big)
\end{equation}
where $S_k\coloneqq \xi_1+\dots+\xi_k$ and $R_k\coloneqq r_1+\dots+r_k$ with $r_1,\dots,r_k$
independent Rademacher random variables. This estimate and the classical
Berry--Esseen theorem in turn imply that for every\footnote{Here, given two positive quantities
$a$ and $b$, we write $a\gtrsim b$ to denote the fact that $a\meg C b$ for some positive
universal constant $C$.} $n \gtrsim k\gtrsim n^{2/3}$ we have
\begin{equation} \label{e12.13}
d_K\Big( \frac{\xi_1+\dots+\xi_k}{\sqrt{k}}, \mathcal{N}(0,1)\Big) \gtrsim \frac{k}{n};
\end{equation}
on the other hand, by anticoncentration considerations, if $k=O(n^{2/3})$, then we have
\begin{equation} \label{e12.14}
d_K\Big( \frac{\xi_1+\dots+\xi_k}{\sqrt{k}}, \mathcal{N}(0,1)\Big) \gtrsim \frac{1}{\sqrt{k}}
\gtrsim \frac{k}{n}.
\end{equation}
In other words, we have the desired lower bound in the regime $k=O(n)$. Notice, however,
that the random vector $\boldsymbol{\xi}$ is not isotropic, but this is not a serious
problem and it can be easily fixed by appropriately perturbing $\boldsymbol{\xi}$.

Specifically, let $A_0,A_1,\dots,A_n$ be pairwise disjoint events which are independent
of $\boldsymbol{\xi}$ and satisfy $\prob(A_0)=\frac{1}{n-1}$ and
$\prob(A_1)=\dots=\prob(A_n)=\frac{n-2}{n(n-1)}$. For every $i\in [n]$ set
\begin{equation} \label{e12.15}
E_i\coloneqq \mathbf{1}_{A_0}-\frac{n}{n-2}\mathbf{1}_{A_i} \ \ \text{ and } \ \
X_i\coloneqq \frac{1}{\sigma} (\xi_i+E_i)
\end{equation}
where $\sigma^2\coloneqq \ave\big[(\xi_i+E_i)^2\big]=1+\frac{2}{n-2}$. Using the fact
that $\ave[\xi_1\xi_2]=-\frac{1}{n-1}$, it~is straightforward to check that the random
vector $\bbx=(X_1,\dots,X_n)$ is exchangeable, isotropic and bounded---hence,
it satisfies \eqref{e12.7}---and, moreover,
\begin{equation} \label{e12.16}
\mathrm{osc}(\bbx)=O\Big(\frac{1}{n}\Big) \ \ \text{ and } \ \
\big|\ave[X_1^2X_2^2]-1\big|=O\Big(\frac{1}{n}\Big).
\end{equation}
Next let $k\mik n$ be a positive integer. Note that the random variable
$\frac{1}{\sqrt{k}}(E_1+\dots+E_k)$ is independent of $\frac{1}{\sqrt{k}}(\xi_1+\dots+\xi_k)$
and it is highly concentrated around zero: it takes the value $\sqrt{k}$ with probability
$\frac{1}{n-1}$, the value $-\frac{1}{\sqrt{k}}\cdot\frac{n}{n-2}$ with probability $\frac{k(n-2)}{n(n-1)}$
and the value $0$ everywhere else. Taking into account these remarks and using
\eqref{e12.12} and the classical Berry--Esseen theorem, we obtain that for every
$n \gtrsim k\gtrsim n^{2/3}$,
\begin{equation} \label{e12.17}
d_K\Big( \frac{X_1+\dots+X_k}{\sqrt{k}}, \mathcal{N}(0,1)\Big) \gtrsim \frac{k}{n}.
\end{equation}
After observing that the analogue of \eqref{e12.14} for the random vector $\bbx$ is also
valid if $k=O(n^{2/3})$, we see that for every $k=O(n)$,
\begin{equation} \label{e12.18}
d_K\Big( \frac{X_1+\dots+X_k}{\sqrt{k}}, \mathcal{N}(0,1)\Big) \gtrsim \frac{k}{n}.
\end{equation}
Finally, by replacing $\bbx$ with a subvector of appropriate length\footnote{More precisely,
of length $cn$ where $c>0$ is a sufficiently small constant.} if necessary, we conclude that \eqref{e12.18}
holds true for every positive integer $k\mik n$; in particular, by \eqref{e12.8}, \eqref{e12.16} and
\eqref{e12.18}, it follows that the Berry--Esseen bound
\begin{equation} \label{e12.19}
d_K\Big( \frac{X_1+\dots+X_k}{\sqrt{k}}, \mathcal{N}(0,1)\Big) =O\Big(\frac{1}{\sqrt{k}}\Big)
\end{equation}
is satisfied only in the regime $k=O(n^{2/3})$.
\end{examp}

\subsection{High-dimensional examples} \label{subsec12.2}

Our goal in this subsection is twofold: to estimate the oscillation of high-dimensional exchangeable
random tensors, and to present related applications. We first deal with random tensors whose entries
have finite fourth moment. Specifically, we have the following analogue of Fact \ref{f12.2}.
(The proof is left to the reader.)
\begin{fact} \label{f12.6}
Let $\bbx$ which satisfy \emph{(\hyperref[A1]{$\mathcal{A}$1})} and
\emph{(\hyperref[A2]{$\mathcal{A}$2})}. Assume that the entries of $\bbx$ have
finite fourth moment. Then we have
\begin{equation} \label{e12.20}
\mathrm{osc}(\bbx) \mik \sqrt{\mathrm{pc}(\bbx)}+ \frac{5d}{\sqrt{n}} \big(|\delta_1|+\ave[X_{(1,\dots,d)}^4]\big)^{1/2}.
\end{equation}
\end{fact}
The following proposition estimates the oscillation of dissociated, exchangeable and symmetric random
tensors, as well as the oscillation of their mixtures.
\begin{prop} \label{p12.7}
Let $n,d$ be positive integers with $n\meg 4d$, and let $\bbx=\langle X_i: i\in [n]^d\rangle$ be an exchangeable
and symmetric random tensor whose entries have finite third moment.
\begin{enumerate}
\item[(1)] If $\bbx$ is dissociated, then
\begin{equation} \label{e12.21}
\mathrm{osc}(\bbx) \mik \frac{16d\, \big(\ave\big[|X_{(1,\dots,d)}|^3\big]+1\big)}{\sqrt[4]{n}}.
\end{equation}
\item[(2)] If $\bbx$ is a mixture of dissociated, exchangeable and symmetric random tensors,~then
\begin{equation} \label{e12.22}
\mathrm{osc}(\bbx) \mik \sqrt{\mathrm{pc}(\bbx)} +
\frac{16d\, \big(\ave\big[|X_{(1,\dots,d)}|^3\big]+1\big)}{\sqrt[4]{n}}.
\end{equation}
\end{enumerate}
\end{prop}
\begin{proof}
We start with the proof of part (1). Set $K\coloneqq \ave\big[ |X_{(1,\dots,d)}|^3\big]$.
By H\"{o}lder's inequality, we have $\ave\big[|X_iX_p|\big]\mik K^{2/3}$
for every $i,p\in [n]^d$; in particular, $\delta_1\mik K^{2/3}$.

As in the proof of Proposition \ref{p12.3}, we will use a truncation argument. Let
$\lambda>0$ be a cut-off parameter which will be specified later. For every $i\in [n]^d$ set
\begin{equation} \label{e12.23}
G_i\coloneqq X_i\,\mathbf{1}_{[|X_i|\mik \lambda]} \ \ \ \text{ and } \ \ \
B_i\coloneqq X_i\,\mathbf{1}_{[|X_i|>\lambda]}
\end{equation}
and observe that
\begin{equation} \label{e12.24}
\|B_i\|_{L_1}\mik \frac{K}{\lambda^2} \ \ \ \text{ and } \ \ \ \|B_i\|_{L_2}\mik \sqrt{\frac{K}{\lambda}}.
\end{equation}
Moreover, for every $j\in [n]$ set
\begin{equation} \label{e12.25}
Y_j\coloneqq \frac{1}{n^{d-1}} \sum_{\substack{i\in [n]^d\\i(1) = j}} X_i \ \ \ \text{ and } \ \ \
Z_j\coloneqq \frac{1}{n^{d-1}} \sum_{\substack{i\in [n]^d\\i(1) = j}} G_i.
\end{equation}
By \eqref{e12.24} and the Cauchy--Schwarz inequality, we have
\begin{equation} \label{e12.26}
\ave\big[|Y_j^2-Z^2_j|\big] \mik \frac{1}{n^{2d-2}} \!\!\!\sum_{\substack{i,p\in [n]^d\\i(1)=p(1)= j}}\!\!\!
\big(\ave\big[|G_iB_p|\big]+\ave\big[|B_iG_p|\big]+\ave\big[|B_iB_p|\big]\Big) \mik \frac{3K}{\lambda}
\end{equation}
which yields that
\begin{align} \label{e12.27}
& \mathrm{osc}(\bbx) \stackrel{\eqref{e1.10}}{=} \Big\| \frac{1}{n}\sum_{j=1}^n Y_j^2-\delta_1\Big\|_{L_1}
\mik \Big\| \frac{1}{n}\sum_{j=1}^n Z_j^2-\delta_1\Big\|_{L_1} + \frac{3K}{\lambda} \mik \\
\mik \ \Big\| \frac{1}{n} & \sum_{j=1}^n Z_j^2-\delta_1\Big\|_{L_2} + \frac{3K}{\lambda} =
\Big( \frac{1}{n^2} \sum_{j_1,j_2=1}^n Z^2_{j_1}Z^2_{j_2} + \delta_1^2 -
2\delta_1\frac{1}{n} \sum_{j=1}^n Z_j^2\Big)^{1/2}  + \frac{3K}{\lambda}. \nonumber
\end{align}
Using \eqref{e12.24} and H\"{o}lder's inequality, it is easy to see that for every $j\in [n]$,
\begin{equation} \label{e12.28}
\big|\ave[Z_j^2]-\delta_1\big|\mik \frac{3K}{\lambda}+\frac{4d^2K^{2/3}}{n} \ \ \ \ \text{ and } \ \ \ \
\ave[Z_j^4]\mik K\lambda.
\end{equation}
Next, set $i_1\coloneqq (1,2,\dots,d)$, $i_2\coloneqq (1,d+1,\dots,2d-1)$,
$i_3\coloneqq (2d,2d+1,\dots,3d-1)$ and $i_4\coloneqq (2d,3d,\dots,4d-2)$ if $d\meg 2$;
otherwise, set $i_1=i_2=1$ and $i_3=i_4=2$. By the exchangeability of $\bbx$
and H\"{o}lder's inequality, for every $j_1,j_2\in [n]$ with $j_1\neq j_2$~we~have
\begin{equation} \label{e12.29}
\big|\ave[Z^2_{j_1}Z^2_{j_2}]-\delta_1^2\big| \mik
\big|\ave[G_{i_1}G_{i_2}G_{i_3}G_{i_4}]-\delta_1^2\big| + \frac{8d^2 K^{4/3}}{n}+ K\lambda\,\frac{8d^2}{n};
\end{equation}
on the other hand, by the dissociativity and exchangeability of $\bbx$,
\begin{equation} \label{e12.30}
\ave[G_{i_1}G_{i_2}G_{i_3}G_{i_4}]=\ave[G_{i_1}G_{i_2}]\cdot \ave[G_{i_3}G_{i_4}]=
\ave[G_{i_1}G_{i_2}]^2.
\end{equation}
Since $X_{i_1}X_{i_2}=G_{i_1}G_{i_2}+G_{i_1}B_{i_2}+B_{i_1}G_{i_2}+B_{i_1}B_{i_2}$
and $\ave[X_{i_1}X_{i_2}]=\delta_1$, by \eqref{e12.24} and H\"{o}lder's inequality,
we obtain that
\begin{equation} \label{e12.31}
\big|\ave[G_{i_1}G_{i_2}G_{i_3}G_{i_4}]-\delta_1^2\big| \mik \frac{6K^{5/3}}{\lambda}.
\end{equation}
By \eqref{e12.27}--\eqref{e12.31} and setting $\lambda=\sqrt{n}$,
we conclude that \eqref{e12.21} is satisfied.

We proceed to the proof of part (2). Let $(\Omega,\mathcal{F},\mathbf{P})$ be a probability
space, and assume that $\bbx$ is the mixture with respect to $(\Omega,\mathcal{F},\mathbf{P})$
of the dissociated, exchangeable and symmetric random tensors\footnote{Here, we assume that this
process is sufficiently measurable.} $\langle \bbx_\omega: \omega\in \Omega\rangle$, that is,
\begin{equation} \label{e12.32}
\prob(\bbx\in A)= \mathbf{E}\big[ \prob(\bbx_\omega\in A)\big]
\end{equation}
for every Borel subset $A$ of $\rr^{[n]^d}$, where $\mathbf{E}$ denotes expectation with
respect to $\mathbf{P}$. Since the entries of $\bbx$ have finite third moment,
the entries of $\bbx_\omega$ also have finite third moment $\mathbf{P}$-almost surely.
For every $i\in [n]^d$ let $\bbx_\omega(i)$ denote the $i$-entry of $\bbx_{\omega}$. Set
\begin{equation} \label{e12.33}
Z_\omega\coloneqq \frac{1}{n}\sum_{j=1}^n  \Big(\frac{1}{n^{d-1}}
\sum_{\substack{i\in [n]^d\\i(1) = j}} \bbx_{\omega}(i)\Big)^2
\end{equation}
and observe that
\begin{equation} \label{e12.34}
\mathrm{osc}(\bbx)= \mathbf{E}\big[\|Z_{\omega}-\delta_1\|_{L_1}\big].
\end{equation}
If $\delta_{1,\omega}\coloneqq \delta_1(\bbx_\omega)$ is the parameter in~\eqref{e1.6} associated
with the random tensor $\bbx_\omega$ and $K_\omega$ denotes the absolute third moment of its entries,
then, by \eqref{e12.34}, the triangle inequality for the $L_1$-norm and the Cauchy--Schwarz inequality,
\begin{align} \label{e12.35}
\mathrm{osc}(\bbx) & \mik \mathbf{E}\big[\|Z_\omega-\delta_{1,\omega}\|_{L_1}\big]
+\mathbf{E}\big[|\delta_{1,\omega}-\delta_1|\big] \\
& = \mathbf{E}\big[\mathrm{osc}(\bbx_\omega)\big] +
\mathbf{E}\big[|\delta_{1,\omega}-\delta_1|\big]
\stackrel{\eqref{e12.21}}{\mik} \mathbf{E}\big[|\delta_{1,\omega}-\delta_1|\big]
+ \frac{16d\, \big(\mathbf{E}[K_\omega]+1\big)}{\sqrt[4]{n}} \nonumber \\
& \mik \sqrt{\big|\mathbf{E}[\delta_{1,\omega}^2]-\delta_1^2\big|} +
\frac{16d\, \big(\ave\big[|X_{(1,\dots,d)}|^3\big]+1\big)}{\sqrt[4]{n}}. \nonumber
\end{align}
Finally, since the random tensors $\langle \bbx_\omega:\omega\in \Omega\rangle$
are exchangeable and dissociated,~we~have
\begin{equation} \label{e12.36}
\mathrm{pc}(\bbx) =\big|\mathbf{E}[\delta_{1,\omega}^2]-\delta_1^2\big|.
\end{equation}
Combining \eqref{e12.35} and \eqref{e12.36} the result follows.
\end{proof}
\begin{rem} \label{r12.8}
Note that the proof of Proposition \ref{p12.7} actually yields the validity of
\eqref{e12.21} under the slightly weaker hypothesis that the subtensor
$\bbx_{[4d-2]}=\langle X_i:i\in [4d-2]^d\rangle$ of $\bbx$ is dissociated.
\end{rem}

\subsubsection{Infinitely extendible random tensors} \label{subsubsec12.2.1}

Recall that an exchangeable random tensor $\bbx=\langle X_i: i\in [n]^d\rangle$ is called
\textit{infinitely extendible} if there exists an exchangeable (infinite) random tensor
$\boldsymbol{Y}=\langle Y_i: i\in \nn^d\rangle$ such that
$\boldsymbol{Y}_n\coloneqq \langle Y_i: i\in [n]^d\rangle$ and $\bbx$ have the
same distribution. It is a classical observation that exchangeable random tensors
might fail to be infinitely extendible. Nevertheless, many exchangeable random tensors that
appear in practice are infinitely extendible; \textit{e.g.}, all exchangeable random tensors
whose entries are functions of i.i.d. random variables are infinitely extendible.

Arguably, infinitely extendible exchangeable random tensors are better-behaved;
for instance, it is easy to see that if $\bbx$ satisfies (\hyperref[A1]{$\mathcal{A}$1})
and (\hyperref[A2]{$\mathcal{A}$2}) and it is infinitely extendible, then the parameters
\[ \delta_0(\bbx),\dots,\delta_d(\bbx) \ \ \ \text{ and } \ \ \ \Sigma_0(\bbx),\dots,\Sigma_d(\bbx) \]
associated with $\bbx$ via \eqref{e1.6}, \eqref{e1.7} and \eqref{e1.9},
are all nonnegative. That said, we have the following version of Theorem \ref{t1.4}
for this class of random tensors.
\begin{cor} \label{c12.9}
Let $\bbx,\bbth$ which satisfy \emph{(\hyperref[A1]{$\mathcal{A}$1})},
\emph{(\hyperref[A2]{$\mathcal{A}$2})}, \emph{(\hyperref[A3]{$\mathcal{A}$3})} with $n\meg 4d$,
and assume that $\bbx$ is infinitely extendible. Assume, moreover, that $\seminorm{\bbth}_1=1$,
and let $\alpha \in (0,1)$ such that the following non-degenericity condition holds true
\begin{equation} \label{e12.37}
\delta_1> \max\big\{ \mathrm{pc}(\bbx)^{\frac{\alpha}{2}}, \delta_0^\alpha\big\}
\end{equation}
where $\mathrm{pc}(\bbx)$ is as in \eqref{e3.7}.
Then, setting $\sigma^2\coloneqq \mathrm{Var}\big(\langle \bbth,\bbx\rangle\big)$, we have
\begin{equation} \label{e12.38}
d_K\big( \langle\bbth,\bbx\rangle, \mathcal{N}(0,\sigma^2)\big) \mik E''_1+E''_2+E''_3
\end{equation}
where
\begin{align}
\label{e12.39} E''_1 & \coloneqq 5\mathrm{pc}(\bbx)^{\frac{1-\alpha}{2}}+
5\delta_0^{1-\alpha} +
\frac{\delta_0}{d^2\delta_1}\, \big|\seminorm{\bbth}^2_0-1\big| \\
\label{e12.40} E''_2 & \coloneqq 2^{36}\, \frac{\ave\big[|X_{(1,\dots,d)}|^3\big]}{\delta_1^{3/2}}
\Big( \sum_{j=1}^n \Big|\!\sum_{\substack{i\in[n]^d\\i(1) = j}} \theta_i\Big|^3\Big) \\
\label{e12.41} E''_3 & \coloneqq 3\kappa\,
\frac{1}{d\sqrt{\delta_1}}\, \sum_{s=2}^d \binom{d}{s} \sqrt{s!} \,
\sqrt{\Sigma_s} \, \seminorm{\bbth}_s.
\end{align}
Here, $\kappa=20 d^3 18^d (2d)!$ is as in Theorem \emph{\ref{t1.4}}.
\end{cor}
The asymptotic behavior of the normalized averages $\frac{1}{n^{d-1/2}} \sum_{i\in [n]^d_{\mathrm{Inj}}} X_i$
which correspond to a random tensor $\bbx=\langle X_i: i\in [n]^d\rangle$ as in Corollary \ref{c12.9},
is well-studied in the literature; see, \textit{e.g.}, \cite{BCRT58,DDG21,EW78,Si76}.
For this particular case, Corollary~\ref{c12.9} yields Berry--Esseen bounds under the conditions
\begin{equation} \label{e12.42}
\delta_0=0, \ \ \ \ \mathrm{pc}(\bbx)=0 \ \ \ \text{ and } \ \ \ \delta_1>0.
\end{equation}
These conditions also cover the case of several popular
statistics (\textit{e.g.}, non-degenerate, possibly weighted, U-statistics).
\begin{proof}[Proof of Corollary \emph{\ref{c12.9}}]
Let $\boldsymbol{Y}=\langle Y_i:i\in \nn^d\rangle$ be an exchangeable (infinite) random tensor
such that $\bbx$ and $\boldsymbol{Y}_n$ have the same distribution, where for every integer $N\meg d$
by $\boldsymbol{Y}_N$ we denote the subtensor $\langle Y_i:i\in [N]^d\rangle$ of\, $\boldsymbol{Y}$.
By (\hyperref[A2]{$\mathcal{A}$2}), we see that $\boldsymbol{Y}$ is also symmetric
and its diagonal terms vanish. Therefore, by the Aldous--Hoover theorem
\cite{Ald81,Hoo79}, $\boldsymbol{Y}$~is~a mixture of exchangeable,
dissociated and symmetric infinite random tensors whose diagonal terms vanish.
This property is inherited, of course, to the subtensors of $\boldsymbol{Y}$. In particular,
for every integer $N\meg n$ the finite tensor $\boldsymbol{Y}_N$ is a mixture of exchangeable,
dissociated and symmetric random tensors and, consequently, part (2) of Proposition \ref{p12.7}
can be applied to $\boldsymbol{Y}_N$; also notice that
\begin{enumerate}
\item[---] the entries of $\boldsymbol{Y}_N$ and $\bbx$ have the same absolute third moment,
\item[---] $\delta_s(\bbx)=\delta_s(\boldsymbol{Y}_N)$ and
$\Sigma_s(\bbx)=\Sigma_s(\boldsymbol{Y}_N)$ for every $s\in\{0,\dots,d\}$, and
\item[---] $\mathrm{pc}(\bbx)=\mathrm{pc}(\boldsymbol{Y}_N)$.
\end{enumerate}
Next fix a real tensor $\bbth=\langle \theta_i: i\in [n]^d\rangle$ which satisfies
(\hyperref[A3]{$\mathcal{A}$3}) and with $\seminorm{\bbth}_1=1$. Let $N\meg n$ be an arbitrary integer.
We extend $\bbth$ to a real tensor $\bbth_N=\langle \theta_i: i\in [N]^d\rangle$ by setting
$\theta_i=0$ if $i\notin [n]^d$; clearly we have $\seminorm{\bbth}_s=\seminorm{\bbth_N}_s$
for every $s\in \{0,\dots,d\}$. Moreover, the random variables $\langle \bbth,\bbx\rangle$
and $\langle \bbth_N,\boldsymbol{Y}_N\rangle$ have the same distribution.
By the previous observations and \eqref{e12.22}, we have
\begin{equation} \label{e12.43}
\mathrm{osc}(\boldsymbol{Y}_N) \mik \sqrt{\mathrm{pc}(\bbx)} +
\frac{16d\, \big(\ave\big[|X_{(1,\dots,d)}|^3\big]+1\big)}{\sqrt[4]{N}}
\end{equation}
while, by \eqref{e11.35},
\begin{equation} \label{e12.44}
B_N=\Big\| \frac{1}{N^d} \sum_{i\in [N]^d} Y_i\Big\|_{L_2}^2 \mik \delta_0 + \frac{4d^2}{N}.
\end{equation}
By \eqref{e12.43}, \eqref{e12.44} and \eqref{e12.37}, we see that the non-degenericity condition \eqref{e1.11}
for the random tensor $\boldsymbol{Y}_N$ is satisfied for every sufficiently large positive integer $N$.
The result follows by applying Theorem \ref{t1.4} to $\boldsymbol{Y}_N, \bbth_N$
and using \eqref{e12.43}, and then taking the limit as $N$ goes to infinity.
\end{proof}

\subsection{Anticoncentration of polynomials: proof of Theorem \ref{t3.2}} \label{subsec12.3}

Let $n,d,\kappa,k$ and $\boldsymbol{\xi}=(\xi_1,\dots,\xi_n)$ be as in the statement of the theorem,
and fix a polynomial $f$ as in \eqref{e3.13}. We define a random tensor
$\bbx=\langle X_i: i\in [n]^d\rangle$ and a real tensor
$\bbth=\langle \theta_i:i\in [n]^d\rangle$ by setting $X_i=\theta_i=0$
if $i\notin [n]^d_{\mathrm{Inj}}$, and
\begin{equation} \label{e12.45}
X_i\coloneqq \prod_{r=1}^d \xi_{i_r} - \ave\Big[\prod_{r=1}^d \xi_{i_r}\Big] \ \ \ \ \text{ and } \ \ \ \
\theta_i\coloneqq \frac{a_{\{i_1,\dots,i_d\}}}{d!}
\end{equation}
for every $i=(i_1,\dots,i_d)\in  [n]^d_{\mathrm{Inj}}$, where
$\boldsymbol{\alpha}=\langle a_F: F\in \binom{[n]}{d} \rangle$ are the coefficients of $f$.
It is clear that with these choices we have
$\langle \bbth,\bbx\rangle=f(\boldsymbol{\xi})-\ave\big[f(\boldsymbol{\xi})\big]$ and, consequently,
\begin{enumerate}
\item[---] $\mathrm{Var}\big(\langle \bbth,\bbx\rangle\big)=\mathrm{Var}\big(f(\boldsymbol{\xi})\big)$, and
\item[---] $\mathcal{L}_{\langle \bbth,\bbx\rangle}(\ee)=\mathcal{L}_{f(\boldsymbol{\xi})}(\ee)$ for every $\ee>0$.
\end{enumerate}
Also notice that $\bbx$ and $\bbth$ satisfy the basic assumptions (\hyperref[A1]{$\mathcal{A}$1}),
(\hyperref[A2]{$\mathcal{A}$2}) and (\hyperref[A3]{$\mathcal{A}$3}). Moreover, by a tedious but fairly
straightforward computation, it is not hard to see that
\begin{align}
\label{e12.46} & \big|\delta_s- p^{2d-s}(1-p^s)\big| \mik p^{2d-s}\, \frac{12d^2}{k}
\ \ \text{ for every } s\in \{0,\dots,d\} \\
\label{e12.47} & \big|\Sigma_s-p^{2d-s}(1-p)^s\big| \mik p^{2d-s}(1+p)^s\, \frac{12d^2}{k}
\ \ \text{ for every } s\in \{1,\dots,d\} \\
\label{e12.48} & \mathrm{osc}(\bbx)\mik p^{2d-1}\, \frac{15d^2}{k} \\
\label{e12.49} & B= \Big\| \frac{1}{n^d} \sum_{i\in [n]^d} X_i\Big\|_{L_2}^2 =0
\end{align}
where $p=k/n$. By \eqref{e3.15}, \eqref{e12.46}, \eqref{e12.48} and \eqref{e12.49}, it follows that
the non-degenericity condition \eqref{e1.11} is satisfied with $\alpha=1/2$. Therefore, after appropriately
normalizing, by Theorem~\ref{t1.4} we obtain that
\begin{align} \label{e12.50}
d_K\big(\langle \bbth,\bbx\rangle,\mathcal{N}(0,\sigma^2)& \big) \mik
\frac{8\kappa}{\sqrt{p}}\cdot \frac{1}{\sqrt{n}} + \frac{6}{(1-p)\seminorm{\bbth}^2_1}
\cdot \frac{\seminorm{\bbth}^2_0}{n} \, +\\
& + \frac{2^{38}p^{3/2}}{p^{3d}(1-p)^{3/2}\seminorm{\bbth}^3_1}
\Big( \sum_{j=1}^n \Big|\!\sum_{\substack{i\in[n]^d\\i(1) = j}} \theta_i\Big|^3\Big)\, + \nonumber \\
& + \frac{8\kappa p^{1/2}}{dp^d(1-p)^{1/2}\seminorm{\bbth}_1}\, \sum_{s=2}^d \binom{d}{s} \sqrt{s!} \,
\sqrt{p^{2d-s}(1-p)^s} \, \seminorm{\bbth}_s \nonumber
\end{align}
where we have used \eqref{e3.15} and \eqref{e12.46}--\eqref{e12.49}. Next observe that for every $\ee>0$,
\begin{align} \label{e12.51}
\mathcal{L}_{\langle \bbth,\bbx\rangle}(\ee) & \mik \mathcal{L}_{\mathcal{N}(0,\sigma^2)}(\ee)
+2d_K\big(\langle \bbth,\bbx\rangle,\mathcal{N}(0,\sigma^2)\big) \\
& \hspace{1.5cm} \mik \frac{\ee}{\sqrt{2\pi}\,\sigma}+
2d_K\big(\langle \bbth,\bbx\rangle,\mathcal{N}(0,\sigma^2)\big). \nonumber
\end{align}
On the other hand, by the choice of the random tensor $\bbth$ in \eqref{e12.45} and the definitions
of the relevant seminorms in \eqref{e1.5} and \eqref{e3.14}, for every $s\in \{0,\dots,d\}$ we have
\begin{equation} \label{e12.52}
\seminorm{\bbth}_s= \frac{(d-s)!}{d!}\, \sqrt{s!}\, \seminorm{\boldsymbol{\alpha}}_s.
\end{equation}
Combining \eqref{e12.50}, \eqref{e12.51} and \eqref{e12.52}, we see that \eqref{e3.17} is satisfied.

Finally, in order to verify the estimate for the variance $\sigma^2$ of $f(\boldsymbol{\xi})$
observe that
\begin{equation} \label{e12.54}
\sigma^2=\mathrm{Var}\big(\langle\bbth,\bbx\rangle\big) \stackrel{\eqref{e1.8},\eqref{e1.9}}{=}
\sum_{s=0}^{d} \binom{d}{s}^2\, s!\, \Sigma_s\, \seminorm{\bbth}_s^2 \stackrel{\eqref{e12.52}}{=}
\sum_{s=0}^{d} \Sigma_s\, \seminorm{\boldsymbol{\alpha}}_s^2.
\end{equation}
Thus, \eqref{e3.16} follows from \eqref{e12.54} together with \eqref{e12.46}, \eqref{e12.47}
and \eqref{e3.15}.


\appendix

\numberwithin{equation}{section}

\section{Proof of Lemma \ref*{l7.3}} \label{appendix}

The argument is an interesting and nontrivial instance of the Stein/Chen method of normal approximation
via concentration (see, \textit{e.g.}, \cite[Section 6]{BC14}).

\subsection{\!} \label{subappA.1}

Recall that $n,d$ are positive integers with $d\meg 2$ and $n\meg 4d^2$. We start by setting
\begin{equation} \label{ea.1}
\delta \coloneqq 16\, \frac{\Lambda}{n} \ \ \ \text{ and } \ \ \
\eta_\delta\coloneqq \frac{n}{4}\, \ave\big[|\Xi_1-\Xi_1'|\cdot\min\{|\Xi_1-\Xi_1'|,\delta\}\, |\, \pi_1\big].
\end{equation}		
Notice that, by \eqref{e7.6}, we have
\begin{equation} \label{ea.2}
\delta \meg \frac{n}{4} \, \ave\big[|\Xi_1-\Xi_1'|^3\big].
\end{equation}
Using \eqref{e7.5}, \eqref{e7.6}, \eqref{ea.2} and the fact that
$\min(\alpha,\beta)\meg \alpha - \frac{\alpha^2}{4\beta}$ for every $\alpha\meg 0$ and $\beta>0$,
we obtain that
\begin{align} \label{ea.3}
\ave[\eta_\delta]  & = \frac{n}{4}\, \ave\big[|\Xi_1-\Xi_1'|\cdot\min\{|\Xi_1-\Xi_1'|,\delta\}\big] \\
& \meg \frac{n}{4}\, \ave\big[(\Xi_1-\Xi_1')^2\big] - \frac{n}{16\delta}\, \ave\big[|\Xi_1-\Xi_1'|^3\big]
= 1- \frac{n^2}{16^2\Lambda}\, \ave\big[|\Xi_1-\Xi_1'|^3\big] \meg \frac{3}{4}. \nonumber
\end{align}
\begin{sbl} \label{sbla.1}
We have
\begin{equation} \label{ea.4}
\mathrm{Var}(\eta_\delta) \mik 2^9\delta^2.
\end{equation}
\end{sbl}
We postpone the proof of Sublemma \ref{sbla.1} to the end of this appendix.

\subsection{\!} \label{subappA.2}

Now let $z\in\mathbb{R}$ be arbitrary. We will show that the probabilities
$\mathbb{P}(z-|\Theta|\mik \Xi_1\mik z)$ and $\mathbb{P}\big(z\mik \Xi_1 \mik z+|\Theta|\big)$
are both upper-bounded by the quantity appearing in the right-hand-side of \eqref{e7.20};
clearly, this is enough to complete the proof. The argument  is symmetric, and so
we shall focus on bounding the probability $\mathbb{P}\big(z\mik \Xi_1 \mik z+|\Theta|\big)$.

Let $f_\Theta\colon\rr\to\rr$ be defined by
\begin{equation} \label{ea.5}
f_\Theta(x) \coloneqq
\begin{cases}
-\big(\frac{1}{2}|\Theta|+\delta\big) & \text{if } x\mik z-\delta, \\
-\frac{1}{2}|\Theta|+x-z & \text{if } z-\delta<x<z+|\Theta|+\delta, \\
\frac{1}{2}|\Theta|+\delta & \text{if } x\meg z+|\Theta|+\delta.
\end{cases}
\end{equation}
By (\hyperref[E4]{$\mathcal{E}$4}), the pair $(\pi_1,\pi_2)$ is exchangeable.
Hence, by \eqref{e7.1} and \eqref{e7.19}, we have
\begin{equation} \label{ea.6}
\ave\big[(\Xi_1'-\Xi_1)\big(f_\Theta(\Xi_1') + f_{\Theta'}(\Xi_1)\big)\big]=0
\end{equation}
which in turn implies, after adding $2\,\ave\big[(\Xi_1-\Xi_1')f_\Theta(\Xi_1)\big]$
and multiplying by $\frac{n}{4}$, that
\begin{align}
\label{ea.7} \frac{n}{2}\, & \ave\big[(\Xi_1-\Xi'_1) f_\Theta(\Xi_1)\big] = \\
&= \frac{n}{4}\,\ave\big[(\Xi_1'-\Xi_1)\big(f_{\Theta'}(\Xi_1) - f_{\Theta}(\Xi_1)\big)\big]
+ \frac{n}{4}\, \ave\big[(\Xi_1'-\Xi_1) \big(f_\Theta(\Xi_1') - f_{\Theta}(\Xi_1)\big)\big]. \nonumber
\end{align}
Moreover, by \eqref{e7.4}, we have
\begin{equation} \label{ea.8}
\frac{n}{2}\, \ave\big[(\Xi_1'-\Xi_1)f_\Theta(\Xi_1)\big] = \frac{n}{2}\,
\ave\big[ f_\Theta(\Xi_1)\, \ave[\Xi_1'-\Xi_1\,|\,\pi_1]\big] =\ave\big[f_\Theta(\Xi_1)\Xi_1\big]
\end{equation}
and so, by \eqref{e7.3}, the Cauchy--Schwarz inequality and the fact that
$\|f_\Theta\|_{L_\infty}\mik \delta +\frac{1}{2}|\Theta|$,
\begin{equation} \label{ea.9}
\Big|\frac{n}{2}\,\ave\big[(\Xi_1'-\Xi_1)f_\Theta(\Xi_1)\big]\Big| \mik
\delta\, \ave\big[|\Xi_1|\big] +\frac{1}{2}\,\ave\big[|\Xi_1\Theta|\big] \mik
\delta + \frac{1}{2}\, \|\Theta\|_{L_2}.
\end{equation}
By the definition of $\Theta$ in \eqref{e7.19}, the triangle inequality and \eqref{e7.7},
we thus have
\begin{equation} \label{ea.10}
\Big|\frac{n}{2}\,\ave\big[(\Xi_1'-\Xi_1)f_\Theta(\Xi_1)\big]\Big| \mik
\delta + e^{d} (2d)!\, \sum_{s=2}^{d}\sqrt{\frac{\beta_s}{n^s}}.
\end{equation}

Next observe that $\|f_\Theta - f_{\Theta'}\|_{L_{\infty}}\mik \frac{1}{2}\,\big||\Theta|-|\Theta'|\big|
\mik \frac{1}{2}\,|\Theta-\Theta'|$. Hence,
\begin{align}
\label{ea.11} & \Big|\frac{n}{4}\,\ave\big[(\Xi_1'-\Xi_1)(f_{\Theta'}(\Xi_1) - f_{\Theta}(\Xi_1))\big]\Big|
\mik \Big| \frac{n}{8}\, \ave\big[|\Xi_1'-\Xi_1|\cdot|\Theta - \Theta'|\big] \Big|\\
\mik \frac{n}{8}\, & \|\Xi_1'-\Xi_1\|_{L_2} \|\Theta - \Theta'\|_{L_2}
\stackrel{\eqref{e7.5},\eqref{e7.19}}{\mik}\, \frac{\sqrt{n}}{4} \sum_{s=2}^{d}\|\Xi_s-\Xi_s'\|_{L_2}
\mik 2\,d^2e^{d}(2d)!\, \sum_{s=2}^{d}\sqrt{\frac{\beta_s}{n^s}} \nonumber
\end{align}
where we have used the Cauchy--Schwarz inequality, the triangle inequality and \eqref{e7.8}.

From this point on, the proof proceeds exactly as in \cite[Lemma 2.2]{BC05}.
More precisely, in order to estimate the second term of the right-hand-side of \eqref{ea.7},
by the fundamental theorem of calculus, we have
\begin{align}
\label{ea.12} A & \coloneqq \frac{n}{4}\, \ave\big[(\Xi_1'-\Xi_1) \big(f_\Theta(\Xi_1') - f_{\Theta}(\Xi_1)\big)\big]
 = \frac{n}{4}\,\ave\Big[(\Xi_1'-\Xi_1) \int_{0}^{\Xi_1'-\Xi_1} \!\! f'_\Theta(\Xi_1+t)\, dt\Big] \\
 & \, = \frac{n}{4}\, \ave\Big[(\Xi_1'-\Xi_1) \int_{-\infty}^{+\infty} f'_\Theta(\Xi_1+t)\,
\big(\mathbf{1}_{[0,\Xi_1'-\Xi_1]}(t)-\mathbf{1}_{[\Xi_1'-\Xi_1,0]}(t)\big)\, dt \Big] \nonumber
\end{align}
with the convention that $\mathbf{1}_{[\beta,\alpha]}$ is constantly $0$ if $\alpha<\beta$.
Also observe that, by \eqref{ea.5}, we have $f'_\Theta=\mathbf{1}_{[z-\delta,z+\delta+|\Theta|]}$
which implies that for every $t\in\rr$,
\begin{enumerate}
\item[$\bullet$] $(\Xi_1'-\Xi_1)\,f'_\Theta(\Xi_1+t)\,
\big(\mathbf{1}_{[0,\Xi_1'-\Xi_1]}(t)-\mathbf{1}_{[\Xi_1'-\Xi_1,0]}(t)\big) \meg 0$.
\end{enumerate}
Therefore,
\begin{equation} \label{ea.13}
A \meg \frac{n}{4}\, \ave\Big[\int_{|t|\mik\delta} (\Xi_1'-\Xi_1)\,
\mathbf{1}_{[z-\delta,z+\delta+|\Theta|]}(\Xi_1+t)\,
\big(\mathbf{1}_{[0,\Xi_1'-\Xi_1]}(t)-\mathbf{1}_{[\Xi_1'-\Xi_1,0]}(t)\big)\, dt\Big].
\end{equation}
Moreover, for every $|t|\mik\delta$ we have
\begin{enumerate}
\item[$\bullet$] $(\Xi_1'-\Xi_1)\,
\big(\mathbf{1}_{[0,\Xi_1'-\Xi_1]}(t)-\mathbf{1}_{[\Xi_1'-\Xi_1,0]}(t)\big)\meg 0$ and
\item[$\bullet$] $\mathbf{1}_{[z-\delta,z+\delta+|\Theta|]}(\Xi_1+t) \meg
\mathbf{1}_{[z,z+|\Theta|]}(\Xi_1)$,
\end{enumerate}
and so, by \eqref{ea.13},
\begin{align}
\label{ea.14} & A \meg \frac{n}{4}\,
\ave\Big[ \mathbf{1}_{[z,z+|\Theta|]}(\Xi_1)\, (\Xi_1'-\Xi_1) \int_{|t|\mik\delta}
\big(\mathbf{1}_{[0,\Xi_1'-\Xi_1]}(t)-\mathbf{1}_{[\Xi_1'-\Xi_1,0]}(t)\big)\, dt\Big] \\
 = \ \, & \frac{n}{4}\,
\ave\big[ \mathbf{1}_{[z,z+|\Theta|]}(\Xi_1)\, |\Xi_1'-\Xi_1|\cdot
\min\big\{\delta,|\Xi_1'-\Xi_1|\big\}\big] \nonumber \\
 \stackrel{\eqref{ea.1}}{=} &
\ave\big[\mathbf{1}_{[z,z+|\Theta|]}(\Xi_1)\,\eta_\delta\big] \meg
 \ave\big[\mathbf{1}_{[z,z+|\Theta|]}(\Xi_1)\big]\, \ave[\eta_\delta] -
\big|\ave\big[\mathbf{1}_{[z,z+|\Theta|]}(\Xi_1)\big(\eta_\delta-\ave[\eta_\delta]\big)\big]\big| \nonumber \\
 \stackrel{\eqref{ea.3}}{\meg} & \frac{3}{4}\,\mathbb{P}\big(z\mik \Xi_1\mik z+|\Theta|\big) -
\big|\ave\big[\mathbf{1}_{[z,z+|\Theta|]}(\Xi_1)\big(\eta_\delta-\ave[\eta_\delta]\big)\big]\big|. \nonumber
\end{align}
On the other hand, by Sublemma \ref{sbla.1} and inequality $\alpha\beta\mik\frac{1}{2}\,(\alpha^2+\beta^2)$
applied for ``$\alpha=\frac{1}{\sqrt{2}}\,\mathbf{1}_{[z,z+|\Theta|]}(\Xi_1)$" and
``$\beta = \sqrt{2}\,(\eta_\delta-\ave[\eta_\delta])$", we see that
\begin{align} \label{ea.15}
\ave\big[\mathbf{1}_{[z,z+|\Theta|]}(\Xi_1)\big(\eta_\delta-\ave[\eta_\delta]\big)\big]
& \mik \frac{1}{2}\, \Big( \frac{1}{2}\, \mathbb{P}\big(z\mik \Xi_1\mik z+|\Theta|\big)
+ 2 \mathrm{Var}(\eta_\delta) \Big) \\
& \mik \frac{1}{4}\, \mathbb{P}\big(z\mik \Xi_1\mik z+|\Theta|\big) + 2^9\delta^2. \nonumber
\end{align}
Thus, by \eqref{ea.14} and \eqref{ea.15}, we have
\begin{equation} \label{ea.16}
A\meg \frac{1}{2}\, \mathbb{P}\big(z\mik \Xi_1\mik z+|\Theta|\big) - 2^9\delta^2.
\end{equation}
Invoking the choice of $\delta$ in \eqref{ea.1} and the choice of $A$ in \eqref{ea.12} and
combining \eqref{ea.7}, \eqref{ea.10}, \eqref{ea.11} and \eqref{ea.16}, we conclude that
\begin{equation} \label{ea.17}
\mathbb{P}\big(z\mik \Xi_1\mik z + |\Theta|\big) \mik
2^5\frac{\Lambda}{n}+2^{17} \frac{\Lambda^2}{n^2} +
5d^2 e^d (2d)! \sum_{s=2}^{d}\sqrt{\frac{\beta_s}{n^s}}.
\end{equation}

\subsection{Proof of Sublemma \ref{sbla.1}} \label{subappA.3}

Let $\boldsymbol{\zeta}\colon [n]^2\times [n]^2\to\rr$ be defined by setting
\begin{align} \label{ea.18}
\boldsymbol{\zeta}(i,j,p,q) & \coloneqq |\boldsymbol{\xi}_1(i,p) +
\boldsymbol{\xi}_1(j,q) - \boldsymbol{\xi}_1(i,q) - \boldsymbol{\xi}_1(j,p)| \ \times \\
& \hspace{2.0cm} \times \min\big\{|\boldsymbol{\xi}_1(i,p) + \boldsymbol{\xi}_1(j,q)
-\boldsymbol{\xi}_1(i,q) - \boldsymbol{\xi}_1(j,p)|, \delta\big\} \nonumber
\end{align}
for every $i,j,p,q\in [n]$ with $i\neq j$ and $p\neq q$; otherwise, set $\boldsymbol{\zeta}(i,j,p,q)=0$.
Then observe that, by the choice of $\eta_{\delta}$ in \eqref{ea.1}, we have
\begin{equation} \label{ea.19}
\eta_\delta = \frac{n}{4}\, \ave\big[ |\Xi_1-\Xi_1'| \cdot \min\big\{|\Xi_1-\Xi_1'|,\delta\big\}\,
|\, \pi_1\big] = \frac{1}{4n}\, \sum_{\substack{i,j\in[n]\\i\neq j}}
\boldsymbol{\zeta}\big(i,j,\pi_1(i),\pi_1(j)\big).
\end{equation}
Also recall that, by (\hyperref[E3]{$\mathcal{E}$3}), $\pi_1$ is uniformly distributed on $\mathbb{S}_n$.
Thus, \eqref{ea.19} asserts the random variable $\eta_\delta$ can be expressed as the $Z$-statistic
of order two associated with the tensor $\boldsymbol{\zeta}$ which has a rather special form.

The proof is based on a specific decomposition of this $Z$-statistic. This decomposition,
also introduced by Barbour/Chen \cite{BC05}, is less symmetric than the decomposition described
in Section \ref{sec9} (yet it has enough symmetries in order to be computationally useful),
but it has the advantage that its non-constant components are mean zero random variables.
Of course, this property is very useful for computing the variance. For the convenience
of the reader we shall briefly recall this decomposition; for more information
we refer to \cite{BC05}.

Specifically, for every $i,j,p,q\in [n]$ set
\[ \boldsymbol{\zeta}(i,j,\cdot,q) \coloneqq \frac{1}{n-1}\sum_{\substack{r\in[n]\\r\neq q}}
 \boldsymbol{\zeta}(i,j,r,q), \ \ \ \boldsymbol{\zeta}(i,j,p,\cdot) \coloneqq
 \frac{1}{n-1}\sum_{\substack{r\in[n]\\r\neq p}} \boldsymbol{\zeta}(i,j,p,r), \]
\[ \hspace{0.15cm} \boldsymbol{\zeta}(i,j,\cdot,\cdot) \coloneqq \frac{1}{n(n-1)}
 \sum_{\substack{r,s\in[n]\\r\neq s}} \boldsymbol{\zeta}(i,j,r,s), \ \ \
 \boldsymbol{\zeta}(i,\cdot,p,\cdot) \coloneqq \frac{1}{(n-1)^2}
 \sum_{\substack{r,s\in[n]\\r\neq i,s\neq p}} \boldsymbol{\zeta}(i,r,p,s),\]
\[ \! \boldsymbol{\zeta}(i,\cdot,\cdot,\cdot) \coloneqq \frac{1}{n(n-1)^2}
 \sum_{\substack{r,s,\ell\in[n]\\\ell\neq i,r\neq s}} \! \boldsymbol{\zeta}(i,\ell,r,s),
 \ \ \ \boldsymbol{\zeta}(\cdot,j,\cdot,q) \coloneqq \frac{1}{(n-1)^2}
 \sum_{\substack{r,s\in[n]\\r\neq j,s\neq q}} \! \boldsymbol{\zeta}(r,j,s,q), \]
\[ \hspace{0.3cm} \boldsymbol{\zeta}(\cdot,j,\cdot,\cdot) \coloneqq \frac{1}{n(n-1)^2}
 \sum_{\substack{r,s,\ell\in[n]\\\ell\neq j,r\neq s}} \!\boldsymbol{\zeta}(\ell,j,r,s),
 \ \ \  \boldsymbol{\zeta}(\cdot,\cdot,\cdot,\cdot) \coloneqq \frac{1}{n^2(n-1)^2}
 \sum_{\substack{r,s,k,\ell\in[n]\\ r\neq s,k\neq\ell}} \! \boldsymbol{\zeta}(r,s,k,\ell). \]
Then, for every $i,j,p,q\in [n]$ we define
\begin{align}
\label{ea.20} \boldsymbol{\zeta}_D(i,j,p,q) & \coloneqq \boldsymbol{\zeta}(i,j,p,q) -
\boldsymbol{\zeta}(i,j,\cdot,q)-\boldsymbol{\zeta}(i,j,p,\cdot) +\boldsymbol{\zeta}(i,j,\cdot,\cdot) \\
\label{ea.21} \boldsymbol{\zeta}_D(i,\cdot,p,\cdot) & \coloneqq \boldsymbol{\zeta}(i,\cdot,p,\cdot) -
\boldsymbol{\zeta}(i,\cdot,\cdot,\cdot) \\
\label{ea.22} \boldsymbol{\zeta}_D(\cdot,j,\cdot,q) & \coloneqq \boldsymbol{\zeta}(\cdot,j,\cdot,q) -
\boldsymbol{\zeta}(\cdot,j,\cdot,\cdot).
\end{align}
Notice that for every $i,p,q\in [n]$ we have $\boldsymbol{\zeta}_D(i,i,p,q)=0$; moreover,
for every $i,j\in [n]$,
\begin{equation} \label{ea.23}
\sum_{\substack{p,q\in[n]\\p\neq q}}\boldsymbol{\zeta}_D(i,j,p,q) = 0, \ \ \ \
\sum_{p=1}^{n}\boldsymbol{\zeta}_D(i,\cdot,p,\cdot)=0 \ \ \text{ and } \ \
\sum_{p=1}^{n}\boldsymbol{\zeta}_D(\cdot,i,\cdot,p)=0.
\end{equation}
By \eqref{ea.19} and the previous definitions, we may decompose $\eta_\delta$ as
\begin{align}
\label{ea.24} \eta_\delta & = \frac{1}{4n} \sum_{\substack{i,j\in[n]\\i\neq j}}
\boldsymbol{\zeta}_D\big(i,j,\pi_1(i),\pi_1(j)\big) +
\frac{n-1}{4n}\sum_{i=1}^{n}\boldsymbol{\zeta}_D \big(\cdot,i,\cdot,\pi_1(i)\big) \ + \\
& \hspace{3.0cm} + \frac{n-1}{4n}\sum_{i=1}^{n}\boldsymbol{\zeta}_D\big(i,\cdot,\pi_1(i),\cdot\big)
+\frac{n-1}{4}\boldsymbol{\zeta}(\cdot,\cdot,\cdot,\cdot). \nonumber
\end{align}
Invoking \eqref{ea.23} and the fact that $\pi_1$ is uniformly distributed on $\mathbb{S}_n$
we see, in particular, that $\ave[\eta_\delta] = \frac{n-1}{4}\boldsymbol{\zeta}(\cdot,\cdot,\cdot,\cdot)$.
Therefore,
\begin{align}
\label{ea.25} \mathrm{Var}(\eta_\delta) & \mik
\frac{3}{2^4n^2}\, \ave\Big[\Big(\sum_{\substack{i,j\in[n]\\i\neq j}}
\boldsymbol{\zeta}_D\big(i,j,\pi_1(i),\pi_1(j)\big)\Big)^2\Big] \, + \\
& \hspace{1.0cm} + \frac{3}{2^4}\, \ave\Big[\Big(\sum_{i=1}^{n}\boldsymbol{\zeta}_D
\big(\cdot,i,\cdot,\pi_1(i)\big)\Big)^2\Big] +
\frac{3}{2^4}\, \ave\Big[\Big(\sum_{i=1}^{n}
\boldsymbol{\zeta}_D\big(i,\cdot,\pi_1(i),\cdot\big)\Big)^2\Big] \nonumber \\
& \coloneqq E_1 +E_2 +E_3. \nonumber
\end{align}
Finally, using the Cauchy--Schwarz inequality, \eqref{ea.23} and arguing\footnote{The argument in
this context is actually simpler.} as in Subsection \ref{subsec5.6}, it is not hard to verify that
\begin{equation} \label{ea.26}
E_1\mik \Big(\frac{2^63}{n}+2^63\Big)\delta^2, \ \ \ \
E_2\mik 2^3 3\delta^2 \ \ \text{ and } \ \ E_3\mik 2^3 3\delta^2.
\end{equation}
By \eqref{ea.25}, \eqref{ea.26} and the fact that $n\meg 3$, we conclude that \eqref{ea.4} is satisfied.

\end{document}